\documentclass[12pt,leqno,a4paper]{amsart}
\usepackage{amssymb,enumerate,cite}
\usepackage{amsmath}
\usepackage{txfonts}
\usepackage{mathrsfs}
\usepackage{hyperref}
\usepackage[all]{xy}

\newcommand{\F}{\mathbb{F}}
\newcommand{\OF}{\overline{\F}}
\newcommand{\ZZ}{\mathbb{Z}}

\newcommand{\bG}{\mathbf{G}}
\newcommand{\bH}{\mathbf{H}}
\newcommand{\bM}{\mathbf{M}}
\newcommand{\bL}{\mathbf{L}}
\newcommand{\bP}{\mathbf{P}}
\newcommand{\bT}{\mathbf{T}}
\newcommand{\bV}{\mathbf{V}}
\newcommand{\bc}{\mathbf{c}}

\newcommand{\cB}{\mathcal{B}}
\newcommand{\cE}{\mathcal{E}}

\newcommand{\cO}{\mathcal{O}}
\newcommand{\cN}{\mathcal{N}}
\newcommand{\cR}{\mathcal{R}}
\newcommand{\cS}{\mathcal{S}}

\newcommand{\fS}{\mathfrak{S}}
\newcommand{\fJ}{\mathfrak{J}}
\newcommand{\fe}{\mathfrak{e}}
\newcommand{\fC}{\mathfrak{C}}
\newcommand{\fb}{\mathfrak{b}}

\newcommand{\scF}{\mathscr{F}}
\newcommand{\sC}{\mathscr{C}}

\newcommand{\ts}{\tilde{s}}
\newcommand{\tz}{\tilde{z}}
\newcommand{\tv}{\tilde{v}}
\newcommand{\tw}{\tilde{w}}

\newcommand{\tpsi}{\tilde{\psi}}
\newcommand{\ttheta}{\tilde{\theta}}
\newcommand{\tchi}{\tilde{\chi}}
\newcommand{\tdelta}{\tilde{\de}}
\newcommand{\tzeta}{\tilde{\zeta}}
\newcommand{\tvtheta}{\tilde{\vartheta}}
\newcommand{\tvphi}{\tilde{\vph}}

\newcommand{\htz}{\hat{\tilde{z}}}
\newcommand{\hxi}{\hat{\xi}}
\newcommand{\hZ}{\widehat{Z}}

\newcommand{\wtB}{\widetilde{B}}
\newcommand{\wtG}{\widetilde{G}}

\newcommand{\wtL}{\widetilde{L}}
\newcommand{\wtT}{\widetilde{T}}
\newcommand{\wtC}{\widetilde{C}}
\newcommand{\wtW}{\widetilde{W}}
\newcommand{\wtR}{\widetilde{R}}
\newcommand{\wtN}{\widetilde{N}}
\newcommand{\wtD}{\widetilde{D}}
\newcommand{\wtZ}{\widetilde{Z}}
\newcommand{\wtM}{\widetilde{M}}
\newcommand{\wtV}{\widetilde{V}}
\newcommand{\wcB}{\widetilde{\cB}}
\newcommand{\wtbG}{\widetilde{\mathbf{G}}}
\newcommand{\wtbL}{\widetilde{\mathbf{L}}}
\newcommand{\wtbT}{\widetilde{\mathbf{T}}}
\newcommand{\wtsC}{\widetilde{\mathscr{C}}}
\newcommand{\wtcR}{\widetilde{\mathcal{R}}}
\newcommand{\wtDelta}{\widetilde{\Delta}}
\newcommand{\wtOm}{\widetilde{\Omega}}

\newcommand{\Aut}{\operatorname{Aut}\nolimits}
\newcommand{\Alp}{\operatorname{Alp}\nolimits}
\newcommand{\IBr}{\operatorname{IBr}\nolimits}
\newcommand{\Ind}{\operatorname{Ind}}
\newcommand{\Lin}{\operatorname{Lin}}
\newcommand{\Irr}{\operatorname{Irr}\nolimits}
\newcommand{\Out}{\operatorname{Out}\nolimits}
\newcommand{\Res}{\operatorname{Res}}
\newcommand{\GL}{\operatorname{GL}}
\newcommand{\GU}{\operatorname{GU}}
\newcommand{\Sp}{\operatorname{Sp}}
\newcommand{\Spin}{\operatorname{Spin}}
\newcommand{\GO}{\operatorname{GO}}
\newcommand{\SO}{\operatorname{SO}}
\newcommand{\CSp}{\operatorname{CSp}}
\newcommand{\CSO}{\operatorname{CSO}}
\newcommand{\PSp}{\operatorname{PSp}}
\newcommand{\dz}{\operatorname{dz}}
\newcommand{\diag}{\operatorname{diag}}
\newcommand{\id}{\operatorname{Id}}
\newcommand{\J}{\operatorname{J}}
\newcommand{\Rad}{\operatorname{Rad}}
\newcommand{\St}{\operatorname{St}}
\newcommand{\rk}{\operatorname{rk}}
\newcommand{\Ker}{\operatorname{Ker}}
\newcommand{\C}{\operatorname{C}}
\newcommand{\N}{\operatorname{N}}
\newcommand{\Z}{\operatorname{Z}}

\newcommand{\lc}[1]{{^{#1}\!}} 
\newcommand{\grp}[1]{\langle#1\rangle}
\newcommand{\Grp}[1]{\left\langle#1\right\rangle}
\newcommand{\set}[1]{\{\,#1\,\}}
\newcommand{\Set}[1]{\left\{\,#1\,\right\}}

\let\alp=\alpha
\let\ga=\gamma
\let\veps=\varepsilon
\let\la=\lambda
\let\vph=\varphi
\let\Ga=\Gamma
\let\ka=\kappa
\let\de=\delta
\let\ti=\times

\theoremstyle{theorem}
\newtheorem{mainthm}{Theorem}
\newtheorem{thm}{Theorem}[section]
\newtheorem{lem}[thm]{Lemma}
\newtheorem{amp}[thm]{Assumption}
\newtheorem{prop}[thm]{Proposition}
\newtheorem{cor}[thm]{Corollary}

\theoremstyle{definition}

\newtheorem{rmk}[thm]{Remark}

\newtheorem*{rem}{Remark}
\numberwithin{equation}{section}

\begin{document}
\raggedbottom

\title[Inductive blockwise Alperin weight condition for type $\mathsf B$]{Inductive blockwise Alperin weight condition\\ for type $\mathsf B$ and odd primes}

\author{Zhicheng Feng}
\address{School of Mathematics and Physics, University of Science and Technology Beijing, Beijing 100083, China.}
\email{zfeng@pku.edu.cn}

\author{Conghui Li}
\address{School of Mathematics, Southwest Jiaotong University, Chengdu 611756, China.}
\email{liconghui@swjtu.edu.cn}

\author{Jiping Zhang}
\address{School of Mathematical Sciences, Peking University, Beijing 100871, China.}
\email{jzhang@pku.edu.cn}

\thanks{Supported by NSFC~(11631001, 11901028 and 11901478).}

\begin{abstract}
	By the reduction theorems of Navarro--Tiep  and Sp\"ath, a way to prove the
Alperin weight conjecture and its blockwise version is to verify the co-called inductive Alperin
weight condition and inductive blockwise Alperin weight condition for all finie
simple groups respectively.
In this paper,
we establish the inductive blockwise Alperin weight condition for simple groups of type $\mathsf B$ and odd primes,
using a criterion given by  Brough and  Sp\"ath recently.
\end{abstract}



\date{\today}

\maketitle


\section{Introduction}

Alperin weight conjecture is one of the most important, basic problems in the representation theory of finite groups.
For a finite group $G$ and a prime $\ell$,
we denote by $\dz(G)$ the set of irreducible characters of $G$ of $\ell$-defect zero. 
An $\ell$-weight of $G$ is a pair $(R,\vph)$,
where $R$ is an $\ell$-subgroup of $G$ and $\vph\in\dz(\N_G(R)/R)$.
We view $\vph$ as a character of $\N_G(R)$.
The group $G$ acts by conjugation on the set of its weights. 
Each weight may be assigned to a unique block.
Let $B$ be an $\ell$-block of $G$.
Then a weight $(R,\vph)$ is a $B$-weight if $b^G=B$,
where $b$ is the block of $\N_G(R)$ containing $\vph$ and $b^G$ is the induced block.
We denote by $\Alp(B)$ the set of $G$-conjugacy classes of $B$-weights.
Then the famous Alperin weight conjecture (cf. \cite{Al87}) now reads as follows:
$$|\IBr(B)|=|\Alp(B)|~\text{for for every block $B$ of every finite group $G$}.$$

The classification of finite simple groups seems to provide a way for proving local-global conjectures in representation theory of finite groups.
A reduction of the non-blockwise version of Alperin weight conjecture had been achieved successfully to simple groups by Navarro and Tiep \cite{NT11} in 2011,
and soon after, Sp\"ath reduced the blockwise Alperin weight conjecture to the verification of the co-called \emph{inductive blockwise Alperin weight (BAW) condition} for finite non-abelian simple groups.
In several cases the inductive BAW condition has been verified:
alternating simple groups, Suzuki and Ree groups by Malle \cite{Ma14},  
groups of Lie type in defining characteristic by Sp\"ath \cite{Sp13}, 
groups of certain Lie types in non-defining characteristic by several authors \cite{AHL21,FLZ20a,FLZ20b,FLZ19,FM20,Li19,Li21,Ma14,Sch16}.

Recently, Brough and Sp\"ath \cite{BS20} established a technical result,
which is useful for the verification of the inductive BAW condition for simple groups of Lie type with abelian outer automorphism groups,
especially when the outer automorphism groups are non-cyclic; see Theorem \ref{thm:criterion}.
This makes it possible to consider the simple groups of Lie type $\mathsf B$ and $\mathsf C$ and $\mathsf E_7$.
In this paper, we focus on the groups of type $\mathsf B$ and odd non-defining characteristic $\ell$,
and prove the following:

\begin{mainthm}\label{main-thm}
Let $S=\Omega_{2n+1}(q)$ (with $n\ge 2$) be a finite simple group, where $q$ is a power of some odd prime $p$ and $\ell$ be an odd prime different from $p$.
Assume that unitriangularity of the decomposition matrices at the prime $\ell$ is known
(in the sense of Assumption \ref{uni-assumption}).
Then the inductive blockwise Alperin weight condition holds for $S$ and $\ell$.
\end{mainthm}

We note that the simple groups $\Omega_{2n+1}(2^m)\cong \PSp_{2n}(2^m)$ are treated in \cite{Li21}.
Let $S=\Omega_{2n+1}(q)$ (with odd $q$) and $G=\Spin_{2n+1}(q)$.
To verify the inductive BAW conditions for simple group $S$,
we always need to determine the action of automorphisms on the irreducible Brauer characters and weights of $G$ and its regular embedding $\wtG$ (the corresponding finite special Clifford group).

When considering the action of automorphisms on irreducible Brauer characters,
one approach is to transfer to unitriangular basic sets.
As in Assumption \ref{uni-assumption},
we assume that the decomposition matrices of $\wtG$ and $G$ with respect to the basic sets $\cE(\wtG,\ell')$ and $\cE(G,\ell')$ are unitriangular respectively.
It is worth mentioning that such assumption was proven to be true for unipotent blocks by Brunat, Dudas and Taylor \cite{BDT20} in 2020.

We classify the weights of finite special Clifford groups in odd dimension and odd non-defining characteristic.
The actions of automorphisms on weights can be determined,
and then an equivariant bijection between irreducible Brauer characters and weights for $\wtG$ follows.
By restriction, we also get a classification of weights of spin groups.
Combined with the consideration of blocks later,
the blockwise Alperin weight conjecture holds for both $\wtG$ and $G$.
Finally, the local requirement in the inductive BAW condition can be established by method used for type $\mathsf A$ in \cite{FLZ20a}.
 
We should consider the blocks of $\wtG$.
The $\ell$-blocks of finite reductive groups were classified by Cabanes and Enguehard \cite{CE99} whenever $\ell$ is good and this result was generalised by Kessar and Malle \cite{KM15} to its largest possible generality.
We make use of this classification and describe a combinatorial parameterization for the blocks of $\wtG$.
In addition, we also determine the number of blocks of $G$ covered by a given block of $\wtG$.

This paper is structured as follows.
In Section \ref{sec:Preli},
we introduce some notation and the criterion of Brough and Sp\"ath,
and give a brief summary of the representation theory of finite reductive groups and some basic facts concerning special Clifford groups and spin groups.
Then we review some of the standard facts on irreducible characters of $\wtG$ and consider the characters of $G$ in Section \ref{sec:irr-cl}.
After recall the results for $\SO_{2n+1}(q)$ by J. An in Section \ref{sec:SO},
we classify the weights of $\wtG$ and $G$ and consider the actions of automorphisms in Section \ref{sec:weights}.
Section \ref{sec:blocks} is devoted to the study of the blocks of $\wtG$ and $G$.
In Section \ref{proof-main},
we collect all the information together to obtain an equivariant bijection,
and verify the last piece of the inductive condition, namely the local property,
and finally our main Theorem \ref{main-thm} follows.

\section{Preliminaries}   \label{sec:Preli}


\subsection{General notation}

Let $G$ be a finite group. 
We denote the restriction and induction by $\Res$ and $\Ind$ respectively.
For $N\unlhd G$ we sometimes identify the characters of $G/N$ with the characters of $G$ whose kernel contains $N$.

The cardinality of a set, or the order of a finite group, $X$, is denoted by $|X|$.
If a group $A$ acts on a finite set $X$,
we denote by $A_x$ the stabilizer of $x\in X$ in $A$,
and analogously we denote by $A_{X'}$ the setwise stabilizer of $X'\subseteq X$.

Let $\ell$ be a prime.
If $A$ acts on a finite group $G$ by automorphisms,
then there is a natural action of $A$ on $\Irr(G)\cup\IBr_\ell(G)$ given by $(\lc{a}\chi)(g)=(\chi^{a^{-1}})(g)=\chi(g^a)$ for every $g\in G$, $a\in A$ and $\chi\in\Irr(G)\cup\IBr_\ell(G)$.
For $P\le G$ and $\chi\in\Irr(G)\cup\IBr_\ell(G)$,
we denote by 
$A_{P,\chi}$ denotes the stabilizer of $\chi$ in $A_P$.
Sometimes, $A_\chi$  is also denoted by  $A(\chi)$.
A subgroup $R \le G$ is called \emph{$\ell$-radical} if $R=\cO_\ell(\N_G(R))$.
We also say that $R$ is an \emph{$\ell$-radical subgroup} of $G$.

Let $G$ be abelian.
Then all irreducible characters of $G$ are linear characters and we also write $\Lin(G)$ for $\Irr(G)$.
In addition, $\Lin_{\ell'}(G)$ denotes the subset of $\Lin(G)$ consisting of the elements of $\ell'$-order.

A \emph{Brauer pair} of $G$ (with respect to $\ell$) is a pair $(Q,b)$ such that $Q$ is an $\ell$-subgroup of $G$ and $b$ is an $\ell$-block of $\C_G(Q)$.
Let $B$ be an $\ell$-block of $G$. If $B$ has central defect then we define the \emph{canonical character} of $B$ as the only element of $\Irr(B)$ with $\cO_\ell(\Z(G))$ in its kernel.
A Brauer pair $(Q,b)$ is called a \emph{$B$-Brauer pair} if $\mathrm{Br}_Q(B)b=b$,
where $\mathrm{Br}_Q$ denotes the Brauer homomorphism (see e.g. \cite[\S 11]{Th95}).
A $B$-subpair $(Q, b)$ of $G$ is called \emph{self-centralizing} if $b$ has defect group $\Z(Q)$ in $\C_G(Q)$.
For the basic facts about Brauer pairs,
we refer the reader to \cite[\S40 and \S41]{Th95}.

\subsection{Central products}\label{subsec:cen-prod}

In order to consider the local property (\emph{i.e.}, Condition \ref{thm:criterion} (4)), we should deal with central products.
We recall the following facts from \cite[Lemma 5.1]{IMN07}.

Let $X$ be arbitrary finite group, and let $U_1, \ldots, U_t$ (with $t\ge 2$) be subgroups of $X$ such that $X$ is the central product of $U_1, \ldots, U_t$.
That is,  $X=U_1\cdots U_t$,
$[U_i,U_j]=1$ for $i\ne j$, $Z=\bigcap_{i=1}^{t}U_i$ and $(\prod_{i\ne j}U_j)\cap U_i=Z$ for any $i$.
Let $Y=U_1\ti\cdots\ti U_t$ be the direct  product of $U_1, \ldots, U_t$.

Notice in this situation that $Z\subseteq Z(X)$.
Let $\hZ=\{ (z_1,\ldots,z_t)\mid z_i\in Z, \ z_1\cdots z_t=1  \}$.
Then $\hZ\subseteq Z(Y)$.
Let $\Phi:Y\to X$ be the map defined by $(u_1,\ldots,u_t)\to u_1\cdots u_t$.
It can be checked directly that $\Phi$ is a  group homomorphism which is surjective.
Obviously $\hZ\subseteq \Ker(\Phi)$.
Comparing the orders of these groups, we know that $\hZ=\Ker(\Phi)$, \emph{i.e.}, $X\cong Y/\hZ$.
In this way, we have an embedding $\Irr(X)\hookrightarrow \Irr(Y)$ given by identifying the character of $\Irr(X)$ with their lifts to $Y$.

\subsection{Inductive BAW conditions}

We use the criterion for the inductive BAW condition given by Brough and Sp\"ath \cite[Thm. 4.5]{BS20}.
As we will consider groups of type $\mathsf B$,
we also make some simplification to the assumptions: cyclic quotient $\wtG/G$ and cyclic $D$.

\begin{thm}[{\cite[Thm. 4.5]{BS20}}]   \label{thm:criterion}
	Let $S$ be a finite non-abelian simple group and $\ell$ a prime dividing $|S|$.
	Let $G$ be the full $\ell'$-covering group of $S$,
	$B$ an $\ell$-block of $G$ and assume there are groups $\wtG$, $D$
	such that $G \unlhd \wtG \rtimes D$ and the following hold.
	\begin{enumerate}[\rm(1)]\setlength{\itemsep}{0pt}
	\item\begin{enumerate}[\rm(i)]\setlength{\itemsep}{0pt}
		\item $G=[\wtG,\wtG]$,
		\item $\C_{\wtG D}(G)=\Z(\wtG)$ and $\wtG D/\Z(\wtG) \cong \Aut(G)$,
		\item $\Out(G)$ is abelian,
		\item both $\wtG/G$ and $D$ are cyclic.
		\end{enumerate}
	\item Let $\wcB$ be the union of $\ell$-blocks of $\wtG$ covering $B$.
		There exists a $\Lin_{\ell' }(\wtG/G) \rtimes D_{\wcB}$-equivariant
		bijection $\wtOm_{\wcB}: \IBr(\wcB) \to \Alp(\wcB)$ such that
		\begin{enumerate}[\rm(i)]\setlength{\itemsep}{0pt}
		\item $\wtOm_{\wcB}(\IBr(\tilde B)) = \Alp(\tilde B)$ for every $\tilde B \in \wcB$, and 
		\item $\J_G(\tpsi) = \J_G(\wtOm_{\wcB}(\tpsi))$ for every $\tpsi \in \IBr({\wcB})$.
		\end{enumerate}
	\item For every $\tpsi\in\IBr(\wcB)$,
		there exists some $\psi_0\in\IBr(G\mid\tpsi)$ such that
		$(\wtG\rtimes D)_{\psi_0}=\wtG_{\psi_0}\rtimes D_{\psi_0}$,
	\item In every $\wtG$-orbit on $\Alp(B)$,
		there is a $(R,\vph_0)$ such that
		$(\wtG D)_{R,\vph_0} = \wtG_{R,\vph_0} (GD)_{R,\vph_0}$.
	\end{enumerate}
	Then the inductive BAW condition holds for the block $B$.
\end{thm}

\begin{rmk}\label{condition-J}
For the definitions of $\J_G(\tpsi)$ and $\J_G(\wtOm_{\wcB}(\tpsi))$,
see \cite[\S 2]{BS20}.
In this paper, we consider the case when $S$ is a simple group of type $\mathsf B$ and  $\ell$ is odd,
then one has $\ell\nmid |\wtG/G\Z(\wtG)|$ and the condition (ii) of Theorem \ref{thm:criterion} (2) always holds.
\end{rmk}

\subsection{Background of the representations of finite reductive groups}\label{subsec:pre,lie}

We refer to \cite{GM20} for the notation pertaining to the character theory of finite reductive groups.
Let $\F_q$ be the field of $q$ elements, where $q$ is a power of a prime $p$.
Denote by $\OF_q$ the algebraic closure of the field $\F_q$.

\subsubsection{}
Algebraic groups are usually denoted by boldface letters.
Suppose that $\bG$ is a connected reductive algebraic group over $\OF_q$ and $F:\bG\to\bG$ is a Frobenius endomorphism endowing $\bG$ with an $\mathbb F_q$-structure.
The group of rational points $\bG^F$ is finite.
Let $\bG^*$ be Langlands dual to $\bG$ with corresponding Frobenius endomorphism also denoted $F$.

Let $s$ be a semisimple element of ${\bG^*}^F$.
We denote by $\cE(\bG^F,s)$ the corresponding Lusztig series.
By Lusztig's Jordan decomposition (see e.g. \cite[\S 2.6]{GM20}),
there is a bijection between $\cE(\C_{\bG^*}(s)^F,1)$ and $\cE(\bG^F,s)$
which is denoted as $\fJ^{\bG}_{s}$ when $\Z(\bG)$ is connected.

\subsubsection{}
Let $\ell$ be a prime number different from $p$.
If $s$ is a semisimple $\ell'$-element of ${\bG^*}^F$,
then we denote by $\cE_{\ell}(\bG^F,s)$ the union of the Lusztig series $\cE(\bG^F,st)$,
where $t$ runs through semisimple $\ell$-elements of ${\bG^*}^F$ commuting with $s$.
By a theorem of Brou\'e--Michel \cite[Thm.~9.12]{CE04},
$\cE_{\ell}(\bG^F,s)$ is a union of $\ell$-blocks of $\bG^F$.

Also, $\cE(\bG^F,\ell')$ denotes the set of the union of  Lusztig series $\cE(\bG^F,s)$,
where $s$ runs through the semisimple $\ell'$-element of ${\bG^*}^F$.
Considering the elements of $\cE(\bG^F,\ell')$ as a basic set,
we have the following result due to Geck--Hiss \cite[Thm.~14.4]{CE04}.

\begin{thm}\label{basicset}
	Let $\ell$ be a good prime for $\bG$ and different from $p$.
	Assume that $\ell$ does not divide $(Z(\bG)/Z^\circ (\bG))_F$
	(the largest quotient of $Z(\bG)$ on which $F$ acts trivially).
	Let $s\in{\bG^*}^F$ be a semisimple $\ell'$-element.
	Then $\cE(\bG^F,s)$ forms a basic set of $\cE_\ell(\bG^F,s)$.
\end{thm}

In this paper any algebraic group $\bG$ involved is of classical type and the prime $\ell$ is always odd.
Thus the hypothesis of Theorem \ref{basicset} is always satisfied.

\subsubsection{}
Let $d$ be a positive integer.
We will use Sylow $d$-theory and $d$-Harish-Chandra theory
(see for instance \cite[\S25]{MT11} and \cite[\S3.5]{GM20}).
For an $F$-stable torus $\bT$ of $\bG$,
we denote by $\bT_d$ its Sylow $d$-torus.
Then an $F$-stable Levi subgroup $\bL$ of $\bG$ is called \emph{$d$-split}
if $\bL=\C_{\bG}(Z^\circ (\bL)_d)$.
A character $\zeta\in\Irr (\bL^F)$ is called \emph{$d$-cuspidal}
if $\lc{*}R^{\bL}_{\bM\subseteq \bP}(\zeta)=0$
for any proper $d$-split Levi subgroups $\bM$ of $\bL$
and any parabolic subgroup $\bP$ of $\bL$ containing $\bM$ as Levi complement.

Let $s\in {\bG^*}^F$ be semisimple.
Following \cite[Def.~2.1]{KM15},
we say $\chi\in\cE(\bG^F,s)$ is \emph{$d$-Jordan-cuspidal} if
\begin{itemize}
	\item $\Z^\circ(\C^\circ_{\bG^*}(s))_d=\Z^\circ(\bG^*)_d$, and
	\item $\chi$ corresponds under Jordan decomposition to
		the $\C_{\bG^*}(s)^F$-orbit of a $d$-cuspidal unipotent character of $\C^\circ_{\bG^*}(s)^F$.
\end{itemize}
If $\bL$ is a $d$-split Levi subgroup of $\bG$ and $\zeta\in\Irr(\bL^F)$ is $d$-Jordan-cuspidal,
then $(\bL,\zeta)$ is called a \emph{$d$-Jordan-cuspidal pair} of $(\bG, F)$.

By \cite{BM11,Ta18}, the Mackey formula holds for groups of classical type.
Thus the Lusztig induction $R_{\bL\subseteq \bP}^{\bG}$ is independent of the ambient parabolic subgroup $\bP$ in this paper.
So throughout we always omit the parabolic subgroups when considering Lusztig inductions.	

\subsubsection{}\label{subsec:class-blocks}
Now let $\ell$ be an odd prime number different from $p$
and let 
\begin{equation}\label{equ:def-d-ql}
d=d(q,\ell) \ \text{be the multiplicative order of $q$ modulo $\ell$}.
\addtocounter{thm}{1}\tag{\thethm}
\end{equation}
Cabanes and Enguehard \cite{CE99} gave a label for an arbitrary $\ell$-block of finite groups of Lie type for $\ell\ge 7$ and such result was generalised by Kessar and Malle \cite{KM15} to its largest possible generality.
If $(\bL,\zeta)$ is a $d$-Jordan-cuspidal pair of $\bG$ with $\zeta\in\cE(\bL^F,\ell')$,
then there exists a unique $\ell$-block $b_{\bG^F}(\bL,\zeta)$ of $\bG^F$ such that all irreducible constituents of $R_{\bL}^{\bG}(\zeta)$ lie in $b_{\bG^F}(\bL,\zeta)$.
Under the conditions of \cite[Thm. A (e)]{KM15},
the set of $\bG^F$-conjugacy classes of $d$-Jordan-cuspidal pairs $(\bL,\zeta)$ of $\bG$ with $\zeta\in\cE(\bL^F,\ell')$ form a labeling set for the $\ell$-blocks of $\bG^F$.
Note that the assumption there always holds if every component of $\bG$ is of classical type and $\ell$ is odd.

By \cite[Rmk.~2.2]{KM15},
$d$-Jordan-cuspidality and $d$-cuspidality agree for characters lying in an $\ell'$-series
if $\ell$ is an odd prime and every component of $\bG$ is of classical type.

\subsection{}\label{subsect:Clifford}
We recall some basic definitions and properties of Clifford algebras and Clifford groups.

Let $V$ be an orthogonal space over the field $k=\F_q$ or $\OF_q$
with the quadratic form $Q$ and bilinear form $B(-,-)$ such that
\[ B(u,v)=Q(u+v)-Q(u)-Q(v),\quad u,v\in V. \]
Thus $B(u,u)=2Q(u)$ for any $u \in V$ and the quadratic form $Q$ can be recovered from the bilinear form $B(-,-)$ since $q$ is odd.
So for any element $u$ of $V$,
$u$ is isotropic (\emph{i.e.} $B(u,u)=0$) if and only if $u$ is singular (\emph{i.e.} $Q(u)=0$).

The Clifford algebra $C(V)$ on $V$ is defined to be the $k$-algebra with identity $\fe$
generated by the elements of $V$ modulo the relations
\[ v^2=Q(v)\fe,\quad v \in V. \]
Thus an element $v \in V$ is a unit in $C(V)$ if and only if $v$ is non-isotropic,
in which case, $v^{-1}=Q(v)^{-1}v$.
We also have
\[ uv+vu=B(u,v),\quad u,v\in V; \]
in particular, when $u,v$ are orthogonal, \emph{i.e.} $B(u,v)=0$, we have $uv=-vu$.
Given an orthogonal basis $\veps_1,\veps_2,\dots,\veps_m$ of $V$,
$C(V)$ has a basis as a $k$-linear space
\[ \veps_1^{e_1}\veps_2^{e_2}\cdots\veps_m^{e_m},\quad e_i\in\set{0,1}. \]
The Clifford algebra $C(V)$ has a $\ZZ_2$-gradation structure
\[ C(V) = C_0(V)\oplus C_1(V), \]
and the multiplication of two homogenious elements $a,b$ in $C(V)$ satisfies
\[ ab=(-1)^{\deg{a}\deg{b}}ba. \]
The elements in $C_0(V)$ and $C_1(V)$ are called even and odd respectively,
and $C_0(V)$ is called the special Clifford algebra over $V$.
If there is an orthogonal decomposition
\[ V = V_1 \perp V_2, \]
then for any $a \in C_0(V_1)$ and $b \in C(V_2)$, we have $ab=ba$.

The Clifford group and special Clifford group over $V$ is defined as
\begin{align*}
D(V) &= \set{g \in C(V)^\ti \mid gVg^{-1}=V}, \\
D_0(V) &= \set{g \in C_0(V)^\ti \mid gVg^{-1}=V}.
\end{align*}
For any non-isotropic element $x$ of $V$,
the conjugation of $x$ induces the orthogonal transformation $-\varrho_x$ on $V$,
where $\varrho_x$ is the orthogonal reflection defined by $x$, \emph{i.e.}
\[ \varrho_x:\quad V \to V,\quad v \mapsto v-\frac{B(x,v)}{Q(x)}x. \]
This induces a group homomorphism $\pi: D(V) \to \GO(V)$,
which is surjective if and only if $\dim_kV$ is even by \cite[II.3.1]{Chev}.
The restriction of $\pi$ induces an exact sequence
\[ 1 \to k^\ti\fe \xrightarrow{} D_0(V) \xrightarrow{\pi} \SO(V) \to 1. \]
When $\dim_kV$ is odd, $Z(D_0(V)) = k^\ti\fe$.

There is an anti-automorphism $*$ of $C(V)$ stabilizing $V$ pointwise.
Then $N: D_0(V) \to k^\ti,\ g \mapsto gg^*$ is a group homomorphism,
and $\Spin(V)$ is defined to be $\Ker N$.

Assume $V$ is of plus type and of odd dimension $2n+1$.
Choose a basis of $V$
\begin{equation}\label{equ-basis}
\veps_1,\dots,\veps_n,\veps_0,\eta_1,\dots,\eta_n
\addtocounter{thm}{1}\tag{\thethm}
\end{equation}
such that for $v=x_0\veps_0+x_1\veps_1+\cdots+x_n\veps_n+y_1\eta_1+\cdots+y_n\eta_n$,
we have \[ Q(v) = x_0^2+x_1y_1+\cdots+x_ny_n. \]
Thus the metric matrix of the bilinear form $B(-,-)$ under the basis (\ref{equ-basis}) is
\[ \begin{bmatrix} 0&0&I_n\\ 0&2&0\\ I_n&0&0 \end{bmatrix}. \]
By \cite[II.2.6]{Chev}, there is an odd element $z_V$ in $Z(C(V))$ such that $z_V^2=1$
($z_V$ is the product of some scalar in $k$ with the product of all vectors in an orthogonal basis).
Note that $z_V \in D(V)$ but $z_V \notin D_0(V)$.
Furthermore, by \cite[II.3.4]{Chev},
$D_0(V)$ is generated by the products $vz_V$,
where $v$ runs over all non-isotropic elements of $V$;
note that $(v_1z_V)(v_2z_V)=v_1v_2$ for any two non-isotropic elements $v_1,v_2 \in V$.

\subsection{}\label{subsect:Weyl-D0}
Assume $V$ is an orthogonal space over $k=\F_q$ or $\OF_q$ of plus type and of odd dimension $2n+1$.
We now recall the Weyl groups of $\SO(V)$, $D_0(V)$ and $\Spin(V)$,
which are all isomorphic to $\ZZ_2 \wr \fS(n)$,
where $\fS(n)$ denotes the symmetric group of $n$ symbols.
Assume a basis (\ref{equ-basis}) of $V$ is fixed,
then $\SO(V)$ can be identified with the matrix group $\SO_{2n+1}(q)$.

Note that $\veps_i-\eta_i$ is a non-isotropic element in $V$
and $\varrho_{\veps_i-\eta_i}$ denotes the corresponding orthogonal reflection.
Set
\begin{align*}
W_0 &= \grp{ -\varrho_{\veps_i-\eta_i} \mid i=1,2,\dots,n }, \\
W_1 &= \Set{ \diag\{P_\sigma,1,P_\sigma\} \mid \sigma\in\fS(n) },
\end{align*}
where $P_\sigma$ is the permutation matrix in $\GL_n(k)$ representing $\sigma\in\fS(n)$.
Then the Weyl group of $\SO(V)$ can be identified with $W_0 \rtimes W_1$.
We will often identify $W_1$ with $\fS(n)$.
Denote by $w_{ij}$ the transposition $(i,j)$ and set $w_i=-\varrho_{\veps_i-\eta_i}$.

Set $\tw_i = z_V(\veps_i-\eta_i)$ and
\[ \tw_{ij}=\frac{1}{2}(\veps_i-\veps_j+\eta_i-\eta_j)(\veps_i-\veps_j-\eta_i+\eta_j). \]
Then $w_i=\pi(\tw_i)$, $\tw_i^2=-\fe$ and $w_{ij} = \pi(\tw_{ij})$, $\tw_{ij}^2=\fe$,
where $\pi$ is as in \ref{subsect:Clifford}.
Set
\begin{align*}
\wtW_0 &= \grp{ \tw_i \mid i=1,2,\dots,n } \\
\wtW_1 &= \grp{ \tw_{i,i+1} \mid i=1,2,\dots,n-1 }.
\end{align*}
Then the Weyl group of $D_0(V)$ can be identified with $(\wtW_0\rtimes\wtW_1)/\grp{-\fe}$.

\section{Irreducible characters}\label{sec:irr-cl}

In this section, we consider the irreducible characters of finite special Clifford groups and finite spin groups by specializing Lusztig's Jordan decomposition.

\subsection{Notation}\label{notation-spaces-groups}
Let $p$, $q$, $\F_q$ and $\overline{\F}_q$ be as \S \ref{subsec:pre,lie}.
Throughout we assume that $p$ (and thus $q$) is odd.
Let $n$ be a positive integer and $\bV$ (resp. $\bV^*$) be the orthogonal (resp. symplectic) space over $\OF_q$ of dimension $2n+1$ (resp. $2n$).
Let $\wtbG=D_0(\bV)$, $\bG=\Spin(\bV)$ and $\bH=\SO(\bV)$ be the special Clifford group, the spin group and the special orthogonal group on $\bV$ respectively.
Then $\bG$ is a simply connected simple algebraic group and the containment $\bG\le\wtbG$ is a regular embedding (see \cite[Def. 1.7.1]{GM20}).
Moreover,
$\bH=\wtbG/\Z(\wtbG)$, $\wtbG=\bG \Z(\wtbG)$ and $\bG=[\wtbG,\wtbG]$.
Let $\fe$ be the identity element of $\wtbG$, then $\Z(\wtbG)=\overline{\F}_q^\ti\fe$.
We denote by $\pi:\wtbG\to\bH$ the natural epimorphism.
Considering the dual groups,
we have $\wtbG^*=\CSp(\bV^*)$, $\bG^*=\mathrm{PCSp}(\bV^*)$ and $\bH^*=\Sp(\bV^*)$,
the conformal symplectic group, projective conformal symplectic group and symplectic group on $\bV^*$ respectively.

Denote by $F$ both the standard Frobenius maps on $\wtbG$ and $\wtbG^*$ defining an $\F_q$-structure on them.
Let $V$ (resp. $V^*$) be the orthogonal (resp. symplectic) space over $\F_q$ of dimension $2n+1$ (resp. $2n$).
We also write $V=\bV^F$ and $V^*={\bV^*}^F$.
Then the $F$-fixed pointed of the above groups are groups on $V$ (or $V^*$), \emph{i.e.}, 
$\wtG=\wtbG^F=D_0(V)$, $G=\bG^F=\Spin(V)$ and $H=\bH^F=\SO(V)$,
while $\wtG^*=\wtbG^{*F}=\CSp(V^*)$, $G^*=\bG^{*F}=\mathrm{PCSp}(V^*)$ and $H^*={\bH^*}^F=\Sp(V^*)$.
Also $\Z(\wtG)=\F_q^\ti\fe$ and $\pi(\wtG)=H=\wtG/\Z(\wtG)$.

\subsection{Semisimple elements of conformal symplectic groups}

The semisimple elements of conformal symplectic group are classified by Shinoda \cite{Sh80},
using the methods of \cite{SS70}.
Let $(\F_q^\ti)^2=\{x^2\mid x\in\F_q^\ti \}$.
Denote by $\Irr(\F_q[X])$ the set of all monic irreducible polynomials over $\F_q$. 
Fix a $\xi\in \F_q^\ti$.
For any $\Ga\in\Irr(\F_q[X])$ with $\Ga\ne X$,
denote by $\Ga_\xi^*$ the monic irreducible polynomial
whose roots are $\xi\alp^{-1}$ where $\alp$ runs through the roots of $\Ga$.
Set 
\begin{align*}
\scF_{\xi,0}&=\left\{ X-\xi_0,X+\xi_0 ~\right\}\ \text{if} \ \xi=\xi_0^2\in(\F_q^\ti) ^2 ,\\
\scF_{\xi,0}&=\left\{ X^2-\xi ~\right\}\ \text{if} \ \xi\notin(\F_q^\ti) ^2 ,\\
\scF_{\xi,1}&=\left\{ \Ga\in\Irr(\F_{q}[x])\mid \Ga\neq X,\Ga=\Ga^*_\xi, \Ga\nmid X^2-\xi ~\right\},\\
\scF_{\xi,2}&=\left\{~ \Ga\Ga^* ~|~ \Ga\in\Irr(\F_{q}[x]), \Ga\neq X,\Ga\ne\Ga^*_\xi, \Ga\nmid X^2-\xi ~\right\}
\end{align*}
and $\scF_\xi=\scF_{\xi,0}\cup\scF_{\xi,1}\cup\scF_{\xi,2}$.
Note that $\scF_{\xi,0}$, $\scF_{\xi,1}$, $\scF_{\xi,2}$ and $\scF_\xi$ for $\xi=1$
are respectively $\scF_0$, $\scF_1$, $\scF_2$ and $\scF$ in \cite{An94},
but we will not omit $\xi$ when $\xi=1$,
since otherwise $\scF_\xi$ for $\xi=1$ may be confused with $\scF_1$ in  \cite{An94}.

Given $\Ga\in\scF_\xi$,
denote by $d_\Ga$ its degree and by $\de_\Ga$ its \emph{reduced degree} defined by
$$\de_\Ga=\begin{cases}
d_\Ga & \text{if}\ \Ga\in\scF_{\xi,0}; \\
\frac{1}{2}d_\Ga & \text{if}\ \Ga\in\scF_{\xi,1}\cup \scF_{\xi,2}.
\end{cases}$$
Note that every polynomial in $\scF_{\xi,1}\cup \scF_{\xi,2}$ has even degree,
so $\de_\Ga$ is an integer.
In addition, we also mention a sign $\veps_\Ga$ for $\Ga\in\scF_{\xi,1}\cup \scF_{\xi,2}$ defined by
$$\veps_\Ga=\begin{cases} -1 & \text{if}\ \Ga\in\scF_{\xi,1}; \\
1 & \text{if}\ \Ga\in\scF_{\xi,2}.
\end{cases}$$

By \cite[\S 1]{Sh80},
the polynomials in $\scF_\xi$ serve as the \emph{elementary divisors} for all semisimple elements of $\wtG^*$.
Let $\ts$ be a semisimple element of $\wtG^*$ with multiplier $\xi$,
\emph{i.e.}, $(\ts u_1,\ts u_2)=\xi (u_1,u_2)$ for any $u_1,u_2\in V^*$.
Then $\ts=\prod_{\Ga\in\scF_\xi}\ts_\Ga$,
where $\ts_\Ga$ is a semisimple element having a unique elementary divisor $\Ga$.
We denote by $m_\Ga(\ts)$ the multiplicity of $\Ga$ in $\ts_\Ga$ (and in $\ts$).
Note that $m_\Ga(\ts)$ is even if $\Ga\in\scF_{\xi,0}$.
This decomposition of $\ts$ corresponds to a decomposition of $V^*=\perp_{\Ga\in\scF_\xi} V^*_\Ga$.

\subsection{Centralizers of semisimple elements of conformal symplectic groups}\label{subsec:centralizer}

To use Lusztig's Jordan decomposition,
we also need the centralizers of semisimple elements.
We use the convention to denote $\GU(n,q)$ as $\GL(n,-q)$.
\begin{lem}\label{lem-cen-ss-symp}
With the primary decomposition $\ts= \prod_{\Ga\in\scF_\xi} \ts_\Ga$,
we have a decomposition
\[ \C_{H^*}(\ts) = \prod_\Ga \C_\Ga(\ts_\Ga) ~\textrm{with}~
\C_\Ga(\ts_\Ga) = \C_{\Sp(V_\Ga^*)}(\ts_\Ga), \]
where
\begin{equation}\label{equ:cen-Sp}
\C_\Ga(\ts_\Ga)\cong
\begin{cases}
\Sp(V^*_\Ga)\cong\Sp_{m_\Ga(\ts)}(q) & \text{if}\ \xi\in(\F_q^\ti)^2, \Ga\in\scF_{\xi,0} \\
\Sp_{m_\Ga(\ts)}(q^2) & \text{if}\ \xi\notin(\F_q^\ti)^2, \Ga\in\scF_{\xi,0},\\
\GL_{m_\Ga(\ts)}(\veps_\Ga q^{\de_\Ga}) & \text{if}\ \Ga\in\scF_{\xi,1}\cup\scF_{\xi,2}.
\end{cases}
\addtocounter{thm}{1}\tag{\thethm}
\end{equation}
In particular, $\C_\Ga(\ts_\Ga)\cong \Sp_{m_\Ga(\ts)}(q^{d_\Ga})$ if $\Ga\in\scF_{\xi,0}$.
\end{lem}

\begin{proof}
The decomposition is obvious.
Identify the conformal symplectic group with a matrix group via an appropriate basis.
\begin{enumerate}[(i)]
\item
Assume $\xi=\xi_0^2 \in (\F_q^\ti)^2$ and $\Ga=X\pm\xi_0$.
Then $\ts_\Ga = \pm\xi_0I_{m_\Ga(\ts)}$.
\item
Assume $\xi \notin (\F_q^\ti)^2$ and $\Ga=X^2-\xi$.
Let $\xi_0 \in \F_{q^2}^\ti$ be such that $\xi_0^2=\xi$,
then $\xi_0^{q-1}=-1$ and $\xi_0^q=-\xi_0$.
Let
\[ \ts_{\Ga,0} = I_{m_\Ga(\ts)/2} \otimes \diag\{ \xi_0, \xi_0^q, \xi_0,\xi_0^q \} \]
and set
\[ v_\Ga = I_{m_\Ga(\ts)/2} \otimes
\diag\left\{ \begin{bmatrix} 0&1\\ 1&0 \end{bmatrix}, \begin{bmatrix} 0&1\\ 1&0 \end{bmatrix}\right\}.\]
\item
Assume $\Ga \in \scF_{\xi,2}$ and $b\in\F_{q^{\de_\Ga}}$ is a root of $\Ga$.
Let
\[ \ts_{\Ga,0} = I_{m_\Ga(\ts)} \otimes
\diag\{ b, b^q, \dots, b^{q^{\de_\Ga-1}}, \xi b^{-1},\xi b^{-q},\dots,\xi b^{-q^{\de_\Ga-1}} \} \]
and set
\[ v_\Ga = I_{m_\Ga(\ts)} \otimes 
\diag\left\{ \begin{bmatrix} 0&1\\I_{\de_\Ga-1}&0 \end{bmatrix},
\begin{bmatrix} 0&1\\I_{\de_\Ga-1}&0 \end{bmatrix}\right\}. \]
\item
Assume $\Ga \in \scF_{\xi,1}$ and $b\in\F_{q^{2\de_\Ga}}$ is a root of $\Ga$
such that $b^{q^{\de_\Ga}+1}=\xi$.
Let
\[ \ts_{\Ga,0} = I_{m_\Ga(\ts)} \otimes
\diag\{ b, b^q, \dots, b^{q^{\de_\Ga-1}},
b^{q^{\de_\Ga}}, b^{q^{\de_\Ga+1}}, \dots, b^{q^{d_\Ga-1}} \} \]
and set
\[ v_\Ga = I_{m_\Ga(\ts)} \otimes 
\begin{bmatrix} 0&-1\\I_{d_\Ga-1}&0 \end{bmatrix}. \]
\end{enumerate}
By the construction in (ii) $\sim$ (iv),
$v_\Ga \in \Sp_{m_\Ga(\ts)d_\Ga}(\F_q)$.
By Lang--Steinberg theorem (see for example \cite[Thm.~21.7]{MT11}),
there is $g_\Ga \in \Sp_{m_\Ga(\ts)d_\Ga}(\OF_q)$ such that $g_\Ga^{-1}F(g_\Ga)=v_\Ga$.
Then it is easy to see that $\ts_{\Ga,0} \in \Sp_{m_\Ga(\ts)d_\Ga}(\OF_q)^{v_\Ga F}$
and we may assume $\ts_\Ga = g_\Ga \ts_{\Ga,0} g_\Ga^{-1}$ up to conjugacy.
Since
\[ \C_\Ga(\ts_\Ga) \cong \C_{\Sp_{m_\Ga(\ts)d_\Ga}(\OF_q)}(\ts_{\Ga,0})^{v_\Ga F}, \]
the assertion for $\C_\Ga(\ts_\Ga)$ follows easily by direct calculations.
\end{proof}

\begin{lem}\label{lem-tcent}
With the primary decomposition $\ts= \prod_{\Ga\in\scF_\xi} \ts_\Ga$,
there are $\tau_\Gamma$ for each $\Gamma$ and $\tau=\prod_\Ga\tau_\Ga$ such that the following hold.
\begin{enumerate}[\rm(1)]
\item
We have $\tau_\Ga^{q-1} \in \Z(\C_\Ga(\ts_\Ga))$ for any $\Ga\in\scF_\xi$
and $[\tau_\Ga,\C_\Ga(\ts_\Ga)]=1$	for any $\Ga\notin\scF_{\xi,0}$.
\item
We have $\wtG^*=\grp{\tau,H^*}$,
$\C_{\wtG^*}(\ts)=\grp{\tau,\C_{H^*}(\ts)}$ and $\tau^{q-1}\in \Z(\C_{H^*}(\ts))$.
\item
We have $|\C_{G^*}(\ts):\C_{H^*}(\ts)| = q-1 = |\wtG^*:H^*|$.
\item
Set $\C_{\wtG^*}(\ts)_0=\Grp{\tau,\prod\limits_{\Ga\in\scF_{\xi,0}}\C_\Ga(\ts_\Ga)}$
and $\C_{\wtG^*}(\ts)_+=\prod\limits_{\Ga\in\scF_{\xi,1}\cup\scF_{\xi,2}} \C_\Ga(\ts_\Ga)$,
then $\C_{\wtG^*}(\ts)$ is a central product of $\C_{\wtG^*}(\ts)_0$ and $\C_{\wtG^*}(\ts)_+$
over $\langle \tau^{q-1} \rangle$.
\end{enumerate} 
\end{lem}

\begin{proof}
Let $\zeta$ be a generator of $\F_q^\ti$.
For $\Ga\in\scF_{\xi,0}$, set
\[ \tau_{\Ga,0} =
\begin{cases}
I_{m_\Ga(\ts)/2} \otimes \diag\{ 1,\zeta \}, & \xi=\xi_0^2 \in (\F_q^\ti)^2,\; \Ga=X\pm\xi_0, \\
I_{m_\Ga(\ts)/2} \otimes \diag\{ 1,1,\zeta,\zeta \}, & \xi \notin (\F_q^\ti)^2,\; \Ga=X^2-\xi.
\end{cases} \]
For $\Ga \in \scF_{\xi,2}$, set
\[ \tau_{\Ga,0} = I_{m_\Ga(\ts)\de_\Ga} \otimes \diag\{ 1,\zeta \}. \]
For $\Ga \in \scF_{\xi,1}$,
let $\zeta_1 \in \OF_q$ be such that $\zeta=\zeta_1^{q^{\de_\Ga}+1}$ and set
\[ \tau_{\Ga,0} = I_{m_\Ga(\ts)} \otimes
\diag\{ \zeta_1,\zeta_1^q,\dots,\zeta_1^{q^{\de_\Ga-1}},
\zeta_1^{q^{\de_\Ga}},\zeta_1^{q^{\de_\Ga+1}},\dots,\zeta_1^{q^{d_\Ga-1}} \}; \]
note that $\zeta_1^{q^{\de_\Ga}}=\zeta_1^{-1}\zeta$, $\zeta_1^{q^{\de_\Ga+1}}=\zeta_1^{-q}\zeta$,
$\cdots$, $\zeta_1^{q^{d_\Ga-1}}=\zeta_1^{-q^{\de_\Ga-1}}\zeta$.
For all cases, set $\tau_\Ga=g_\Ga\tau_{\Ga,0}g_\Ga^{-1}$,
where $g_\Ga$ is as in Lemma \ref{lem-cen-ss-symp}.
All the assertions then follow by direct calculations.
\end{proof}

\begin{rem}
The above lemma with a direct calculation proof generalizes slightly \cite[(1A), (1B)]{FS89}.
\end{rem}

\subsection{Irreducible characters}\label{subsec:irr-char}
We will make use of partitions and Lusztig symbols
(see e.g. \cite[\S 1 and \S 5]{Ol93} for the definitions of these combinatorics).
Let $\ts$ be a semisimple element of $\wtG^*$ with multiplier $\xi$.
We denote by $\Psi_\Ga(\ts)$ the set of Lusztig symbols of rank $\frac{m_\Ga(\ts)}{2}$ and odd defect
if $\Ga\in\scF_{\xi,0}$, 
while we define $\Psi_\Ga(\ts)$ to be the set of partitions of $m_\Ga(s)$
if $\Ga\in\scF_{\xi,1}\cup\scF_{\xi,2}$.
Then $\Psi_\Ga(\ts)$ is a labeling set for the unipotent characters of $C_\Ga(\ts_\Ga)$.
Let $\Psi(\ts)=\prod_{\Ga\in\scF_{\xi}}\Psi_\Ga(\ts)$.
Since $[\C_{\wtG^*}(\ts),\C_{\wtG^*}(\ts)]=[\C_{H^*}(\ts),\C_{H^*}(\ts)]$,
the set $\Psi(\ts)$ is a labeling set for the unipotent characters of both $\C_{\wtG^*}(\ts)$ and $\C_{H^*}(\ts)$.
For $\mu\in\Psi(\ts)$,
denote by $\widetilde\psi_{\mu}$ the unipotent character of $\C_{\wtG^*}(\ts)$ corresponding to $\mu$.
Then by Jordan decomposition, the map
\begin{equation}\label{equ:jor-dec-irr}
(\ts,\mu)\mapsto \tchi_{\ts,\mu}=\fJ_{\ts}^{\wtG}(\psi_\mu)
\addtocounter{thm}{1}\tag{\thethm}
\end{equation}
is a bijection from the set of the $\wtG^*$-conjugacy classes of the pairs $(\ts,\mu)$,
where $\ts$ is a semisimple element of $\wtG^*$ and $\mu\in\Psi(\ts)$,
to $\Irr(\wtG)$.

By \cite[Prop.~2.5.20]{GM20},
there is an isomorphism of abelian groups
\begin{equation}\label{equ:Z(G*)-Lin(G)}
\Z(\wtG^*) \rightarrow \Irr(\wtG/G),\quad \tz \mapsto \htz.
\addtocounter{thm}{1}\tag{\thethm}
\end{equation}
We identify $\Z(\wtG^*)$ with $\F_q^\ti$.
For $\tilde z\in \Z(\wtG^*)$ and $\Ga\in\scF_\xi$ for some $\xi\in\F_q^\ti$,
if $\alp\in\overline{\F}_q^\ti$ is a root of $\Ga$,
then we let $\tilde z.\Ga$ be the polynomial in $\scF_{\tz^2\xi}$ such that $\tilde z.\Ga$ has a root $\tilde z\alp$.
Let $(\ts,\mu)$ be as in above paragraph.
Then we define $\tilde z.\mu$ by setting $(\tilde z.\mu)_{\tilde z.\Ga}=\mu_\Ga$.
In particular, if $\tilde z=-1_{V}$,
then we also denote by $-\Ga$, $-\mu$ for $\tilde z.\Ga$, $\tilde z.\mu$.

By \cite[Thm.~4.7.1~(3)]{GM20},
we have the following lemma.

\begin{lem}\label{char-lin-act}
For $\tilde z\in \Z(\wtG)$ and  $\tchi_{\ts,\mu}\in\Irr(\wtG)$, one has $\hat{\tilde z}\tchi_{\ts,\mu}=\tchi_{\tilde z\ts,\tilde z.\mu}$.
\end{lem}

Now we consider the irreducible characters of $G=\Spin_{2n+1}(q)$.

\begin{thm}\label{thm:char-spin}
Let $\tchi_{\ts,\mu}\in\Irr(\wtG)$ where $\ts$ is a semisimple element of $\wtG^*$ with multiplier $\xi$.
Then $\Res^{\wtG}_G(\tchi_{\ts,\mu})$ is not irreducible
if and only if $\mu_{\Ga}=\mu_{-\Ga}$ for every $\Ga$.
\end{thm}

\begin{proof}
By Clifford theory (see e.g. \cite[Lemma 2.1~(i)]{Fe19}),
for every $\tchi\in\Irr(\wtG)$,
$|\Irr(G\mid\tchi)|$ is the number of $\tilde z\in \Z(\wtG)$ such that $\hat{\tilde z}\tchi=\tchi$.
Since $|\wtG/\Z(\wtG)G|=2$,
the restriction of $\chi$ to $G$ is either irreducible or splits into two irreducible constituents.
So we let $z_0=-1_{V^*}$,
the only element in $\Z(\wtG^*)$ of order $2$ (in fact, $z_0\in \Z(G^*)$).
Therefore, $\Res^{\wtG}_G(\tchi_{\ts,\mu})$ is not irreducible
if and only if $\widehat{z_0}\tchi=\tchi$.
By Lemma \ref{char-lin-act},
$\widehat{z_0}\tchi=\tchi$ if and only if the pairs $(\ts,\mu)$ and $(-\ts,-\mu)$ are $\wtG^*$-conjugate.
This completes the proof.
\end{proof}

\begin{rmk}
\begin{enumerate}[(i)]
\item
Note that $\mu_{\Ga}=\mu_{-\Ga}$ implies that $m_\Ga(\ts)=m_{-\Ga}(\ts)$.
\item
Keep the notation in Theorem \ref{thm:char-spin}.
If $\xi=1$, then $\Z(\wtG)$ is contained in the kernel of $\tchi_{\ts,\mu}$.
Then $\Res^{\wtG}_G(\tchi_{\ts,\mu})$ is irreducible if and only if $\Res^{H}_{G/Z(G)}\tchi_{\ts,\mu}$ is irreducible
when $\chi_{\ts,\mu}$ is regarded as a character of $H=\wtG/\Z(\wtG)$.
This case was already considered in \cite[\S 4.1]{FLZ19}. 
\end{enumerate}	
\end{rmk}

\subsection{Actions of automorphisms on characters}\label{subsect:auto-on-chars}

Denote by $F_p$ be the field automorphism of $\wtG$ and $G$
and by $(F_p)^*$ the corresponding field automorphism on $\wtG^*$ and $G^*$
via duality as in \cite[Def.~2.1]{CS13}.
Then $(F_p)^*$ sends a matrix $(a_{ij})$ in $\wtG^*$ to the matrix $(a_{ij}^p)$.
By \cite[Thm.~2.5.1]{GLS98},
the group $\wtG\rtimes\grp{F_p}$ induces all the automorphisms of $G$.
In addition, $\Out(G)\cong \wtG/G\Z(\wtG)\ti \grp{F_p}$.

Let $\sigma=F_p^i$.
Let $\Ga\in\scF_\xi$ with $\xi\in\F_q^\ti$.
Then we define ${^{\sigma^*}\Ga}$ to be the polynomial in $\scF_{\xi^p}$ such that
if $\alp$ is a root of $\Ga$ then $\alp^{p^i}$ is a root of ${^{\sigma^*}\Ga}$.
Let $\mu\in\Psi(\ts)$ for some semisimple element $\ts$ of $\wtG^*$.
Then we set $\lc{\sigma^*}\mu\in\Psi(\sigma^*(\ts))$
such that $(\lc{\sigma^*}\mu)_{{^{\sigma^*}\Ga}}=\mu_\Ga$ for every $\Ga$.

\begin{prop}\label{field-action-IBr}
Let $\sigma\in\langle F_p \rangle$.	Then
$\tchi_{\ts,\mu}^\sigma=\tchi_{\sigma^*(\ts),\lc{\sigma^*}\mu}$.
\end{prop}

\begin{proof}
We first claim that $\grp{\ts}$ is conjugate to $\grp{\sigma^*(\ts)}$ in $\wtG^*$.
In fact, the proof is similar with \cite[Lemma~3.2]{LZ18}.
Let $\sigma=F_p^i$.
In $\wtbG^*$, $\sigma^*(\ts)$ is conjugate to $\ts^p$.
Since $\C_{\wtG^*}(\ts)$ is connected,
$\sigma^*(\ts)$ is conjugate to $\ts^p$ in $\wtG^*$ by \cite[Thm.~21.11]{MT11}.
Then the claim holds.
Recall from (\ref{equ:jor-dec-irr}) that $\tchi_{\ts,\mu}=\fJ_{\ts}^{\wtG}(\psi_\mu)$.
By \cite[Thm.~3.1]{CS13},
${\fJ_{\ts}^{\wtG}(\psi_\mu)}^\sigma=\fJ_{{\sigma^*}^{-1}(\ts)}^{\wtG}(\psi_\mu^{{\sigma^*}^{-1}})$.
Then the assertion follows.
\end{proof}

When considering the action of $\wtG\rtimes\grp{F_p}$ on $\Irr(G)$,
we have the following proposition.

\begin{thm}[{\cite[Thm. B]{CS19}}]\label{act-spin}
Let $\chi\in\Irr(G)$ and $\sigma\in\grp{F_p}$, $g\in\wtG\setminus G\Z(\wtG)$.
If $\chi^\sigma=\chi^g$, then $\chi^\sigma=\chi^g=\chi$.
\end{thm}

By Theorem \ref{basicset},
$\cE(\wtbG^F,\ell')$ and $\cE(\bG^F,\ell')$ are basic sets for $\wtbG^F$ and $\bG^F$ respectively.
Now we propose the following assumption.

\begin{amp}\label{uni-assumption}
\begin{enumerate}[\rm(i)]
	\item $\cE(\wtbG^F,\ell')$ is a unitriangular basic set for $\wtbG^F$.
	\item $\cE(\bG^F,\ell')$ is a unitriangular basic set for $\bG^F$.
\end{enumerate}
\end{amp}

Then we have the following corollary for the action of $\wtG\rtimes\grp{F_p}$ on $\IBr(G)$.

\begin{cor}\label{split-stabilizer}
Assume that Assumption \ref{uni-assumption} holds.
Let $\chi\in\IBr(G)$ and $\sigma\in\grp{F_p}$, $g\in\wtG\setminus G\Z(\wtG)$.
If $\chi^\sigma=\chi^g$, then $\chi^\sigma=\chi^g=\chi$.
\end{cor}

\begin{proof}
By \cite[Lemma 7.5]{CS13},
we may transfer to a unitriangular basic set when considering the action of automorphisms on Brauer characters,
thus the assertion follows by Theorem \ref{act-spin}.
\end{proof}

\section{Radical subgroups and weights of $\SO_{2n+1}(q)$}\label{sec:SO}

In this section,
we recall constructions of weights of $\SO_{2n+1}(q)$ in \cite{An94},
and make a slight modification,
which is useful when considering weights of special Clifford groups and spin groups.

\subsection{Cores}\label{subsec:cores}

Keep the notation in \S \ref{sec:irr-cl}.
Let $\ell$ be an odd prime not dividing $q$ (with odd $q$) and let $e$ be the multiplicative order of $q^2$ modulo $\ell$.
Denote by $v$ the discrete valuation such that $v(\ell)=1$.
Let $\ell^a$ be the exact power of $\ell$ dividing $q^{2e}-1$ (\emph{i.e.}, $a=v(q^{2e}-1)$)
and $\veps\in\{\pm 1\}$  so that $\ell^a\mid q-\veps$.
Let $d=d(q,\ell)$ be defined as in (\ref{equ:def-d-ql}).
Then $e=d$ or $d/2$ according as $d$ is odd or even.
We call that $\ell$ is \emph{linear or unitary} according as $e=d$ or $e=d/2$.
Obviously, $\varepsilon=1$ or $-1$ according to $\ell$ is linear or unitary.

For $\xi\in\F_q$, define $e_\Ga$ for each $\Ga\in\scF_\xi$ as follows.
\begin{enumerate}[(1)]
		\item
	When $\Ga\in\scF_{\xi,1}\cup\scF_{\xi,2}$,
	let $e_\Ga$ be the multiplicative order of $\veps_\Ga q^{\de_\Ga}$ modulo $\ell$.
	\item
	When $\xi\in(\F_q^\ti)^2$ and $\Ga\in\scF_{\xi,0}$, let $e_\Ga=e$.
	\item
	When $\xi\notin(\F_q^\ti)^2$ and $\Ga\in\scF_{\xi,0}$,
	let $e_\Ga=e$ or $\frac{e}{2}$ according to $e$ is odd or even.
	Note that $e_\Ga$ is equal to the multiplicative order of $q^4$ modulo $\ell$;
	see Lemma \ref{lem-cen-ss-symp}.
\end{enumerate}
Also $e_\Ga$ is the multiplicative order of $q^{2d_\Ga}$ modulo $\ell$ if $\Ga\in\scF_{\xi,0}$.
For $\Ga \in \scF_1$, parts (1) and (3) is the same as in \cite{An94}.

We will make use of  the $e_\Ga$-cores of partitions (see \cite[\S 3]{Ol93}) and Lusztig symbols. 
For the definition of $e_\Ga$-cores of Lusztig symbols, we follow \cite[p.~307]{FS86} and \cite[p. 159]{FS89}, which is slightly different from those defined in \cite[\S 5]{Ol93}.
In particular, if $\mu_\Ga$ is a Lusztig symbol, then its $e_\Ga$-core
$\ka_\Ga$ is the Lusztig symbol which is gotten by actually removing $w_\Ga$ $e_\Ga$-hooks (resp. $e_\Ga$-cohooks) from $\mu_\Ga$ and there is no $e_\Ga$-hooks (resp. $e_\Ga$-cohooks) in $\ka_\Ga$ if $\ell$ is linear (resp. $\ell$ is unitary). 
We follow \cite{FS86,FS89} and say $e_\Ga$-cores for both $e_\Ga$-cores and $e_\Ga$-cocores in  \cite{Ol93}.

\subsection{} \label{subsec:4.2}
We first recall some constructions in \cite{An94}.

Let $\ga,\alp$ be non-negative integers.
Denote by $Z_\alp$ the cyclic group of order $\ell^{a+\alp}$
and by $E_\ga$ the extraspecial group of order $\ell^{2\ga+1}$.
An $\ell$-group of symplectic type means the central product $Z_\alp E_\ga$ of $Z_\alp$ and $E_\ga$ over $\Z(E_\gamma)$.
We may assume that $E_\ga$ is of exponent $\ell$
since the extraspecial group of exponent $\ell^2$ will not appear in radical subgroups by the results in \cite{An94}.

Let $m$ be a positive integer.
There is an embedding of $Z_\alp E_\ga$ in $\GL(m\ell^\ga,\veps q^{e\ell^\alp})$;
the images of $Z_\alp$, $E_\ga$, $Z_\alp E_\ga$ are denoted as $Z_{m,\alp,\ga}^0$, $E_{m,\alp,\ga}^0$, $R_{m,\alp,\ga}^0$ respectively;
see for example \cite[\S6.A]{Li19}.
The group $\GL(m\ell^\ga,\veps q^{e\ell^\alp})$ and thus the group $R_{m,\alp,\ga}^0$ can be embedded in the group $\SO^{\veps^m}(2me\ell^{\alp+\ga},q)$ (see \cite[\S2]{An94});
denote by $R_{m,\alp,\ga}$ the image of $R_{m,\alp,\ga}^0$ under this embedding.
For any sequence $\bc=(c_1,c_2,\dots,c_r)$,
let $A_{\bc} = A_{c_1} \wr A_{c_2} \wr \cdots A_{c_r}$ be as in \cite{AF90},
where $A_{c_i}$ is the elementary abelian group of order $\ell^{c_i}$.
Set $R_{m,\alp,\ga,\bc} = R_{m,\alp,\ga} \wr A_{\bc}$,
then $R_{m,\alp,\ga,\bc}$ is a subgroup of $\SO^{\veps^m}(2me\ell^{\alp+\ga+|\bc|},q)$,
where $|\bc|=c_1+c_2+\cdots+c_r$.
The subgroups of the form $R_{m,\alp,\ga,\bc}$ are called basic subgroups.

Since we are considering the odd prime $\ell$,
the radical subgroups of $\GO(V)$ are the same as those of $\SO(V)$.
By \cite{An94}, any radical subgroup $R$ of $\GO(V)$ is conjugate a product of some basic subgroups,
and there are corresponding decompositions
\begin{equation}\label{equ-rad-1}
R = R_0 \ti R_1 \ti \cdots \ti R_u,\quad V = V_0 \perp V_1 \perp \cdots \perp V_u,
\addtocounter{thm}{1}\tag{\thethm}
\end{equation}
where $R_0$ is the trivial group on $V_0$ and
$R_i=R_{m_i,\alp_i,\ga_i,\bc_i}$ is a basic subgroup on $V_i$ for each $i>0$.
Note that the type functions satisfy that $\eta(V_i)=\veps^{m_i}$ and $\eta(V)=\prod_{i=0}^u\eta(V_i)$ by \cite[(1.5)]{FS89} since the dimension of each $V_i$ ($i>0$) is even.
Set $V_+ = V_1 \perp \cdots \perp V_u$ and $R_+ = R_1 \ti \cdots \ti R_u$.

To classify the weights, one also needs to know the centralizers and normalizers of radical subgroups.
Denote by $C_{m,\alp,\ga,\bc}$ and $N_{m,\alp,\ga,\bc}$ the centralizer and normalizer of $R_{m,\alp,\ga,\bc}$ in $\GO^{\veps^m}(2me\ell^{\alp+\ga+|\bc|},q)$ respectively.
Then $C_{m,\alp,\ga,\bc}$ is contained in $\SO^{\veps^m}(2me\ell^{\alp+\ga+|\bc|},q)$,
but $N_{m,\alp,\ga,\bc}$ may not be contained in $\SO^{\veps^m}(2me\ell^{\alp+\ga+|\bc|},q)$.
Assume
\begin{equation}\label{equ-rad-2}
R = R_0 \ti \prod_{i=1}^u R_{m_i,\alp_i,\ga_i,\bc_i}^{t_i}, \quad V = V_0 \perp V_+,
\addtocounter{thm}{1}\tag{\thethm}
\end{equation}
is a radical subgroup of $\SO(V)$,
where $R_0$ is the trivial subgroup on $V_0$ and
$(m_i,\alp_i,\ga_i,\bc_i)\neq(m_j,\alp_j,\ga_j,\bc_j)$ for any $i \neq j$,
then the following decomposition of abstract groups
\[ \N_{\GO(V)}(R) = \GO(V_0) \ti \prod_{i=1}^u N_{m_i,\alp_i,\ga_i,\bc_i} \wr \fS(t_i) \]
corresponds to a decomposition of spaces similar as in (\ref{equ-rad-1}).
But the normalizer $\N_{\SO(V)}(R)$ of $R$ in $\SO(V)$ do not have such a decomposition.
To fill the gap, An \cite{An94} first classified the weights $\GO(V_+)$ afforded by $R_+$,
and then reduces the classification of weights of $\SO(V)$ afforded by $R$ to that of $\GO(V_+)$;
for details, see \cite[\S4]{An94}.

\subsection{}\label{subsect:Mod-SO}
For our purpose, we make a slight modification to the above constructions.
For a radical subgroup $R$ of $\SO(V)$ as in (\ref{equ-rad-1}),
we construct a twisted version $R^{tw}$ of $R$ as follows.

Keep the settings in \S\ref{notation-spaces-groups}.
Decompose $\bV$ as follows:
\[ \bV= \bV_0 \perp \bV_1 \perp \cdots \perp \bV_u,\quad \bV_i = \perp_{j=1}^{\ell^{|\bc_i|}} \bV_{ij}, \]
where $\dim\bV_i = 2m_ie\ell^{\alp_i+\ga_i+|\bc_i|}$ and $\dim \bV_{ij} = 2m_ie\ell^{\alp_i+\ga_i}$ for $i>0$.
Choose a basis for $\bV_0$ and each $\bV_{ij}$
such that the union of all these basis can be rearranged as the form (\ref{equ-basis}).
We also denote $\bV_+ = \bV_1 \perp \cdots \perp \bV_u$.

Assume $i>0$.
Let the embedding $R_{m_i,\alp_i,\ga_i}^0$ of $Z_{\alp_i} E_{\ga_i}$ in $\GL(m_i\ell^{\ga_i},\veps q^{e\ell^{\alp_i}})$ be as in \cite[\S6.A]{Li19}.
Similar as in \cite[\S6.A]{Li19}, there is a hyperbolic embedding:
\begin{align*}
\hbar:\quad &\GL(m_i\ell^{\ga_i},\veps q^{e\ell^{\alp_i}}) \to \SO(\bV_{ij}) \\
A &\mapsto \diag\{ A,F(A),\dots,F^{e\ell^{\alp_i}-1}(A),A^{-t},F(A^{-t}),\dots,F^{e\ell^{\alp_i}-1}(A^{-t}) \},
\end{align*}
where $A^{-t}$ denotes the inverse of the transposition of $A$.
Note that when $\veps=-1$, $A^{-t} = F^{e\ell^{\alp_i}}(A)$.
Denote by $R_{m_i,\alp_i,\ga_i,j}^{tw}$ the image of $\hbar(R_{m_i,\alp_i,\ga_i}^0)$
under the natural embedding $\SO(\bV_{ij}) \to \SO(\bV)$.
For the sequence $\bc_i$, we can view $A_{\bc_i}$ as a subgroup of $W_1$ in \S\ref{subsect:Weyl-D0}.
Set
\[ R_{m_i,\alp_i,\ga_i,\bc_i}^{tw}
= \left( \prod_{j=1}^{\ell^{|\bc_i|}} R_{m_i,\alp_i,\ga_i,j}^{tw} \right) \rtimes A_{\bc_i}; \]
of course, $R_{m_i,\alp_i,\ga_i,\bc_i}^{tw} \cong R_{m_i,\alp_i,\ga_i} \wr A_{\bc_i}$.
Set
\begin{equation}\label{equ-rad-3}
R^{tw} = \prod_{i=1}^u R_i^{tw},\quad R_i^{tw} = R_{m_i,\alp_i,\ga_i,\bc_i}^{tw}.
\addtocounter{thm}{1}\tag{\thethm}
\end{equation}
Here, since each $ R_i^{tw}$ is already a subgroup of $\SO(\bV)$,
there is no need to include $R_0^{tw}$ in the decomposition.
If $\ell$ is linear, let $v_{m_i,\alp_i,\ga_i,\bc_i}$ be
\[ \id_{\bV_0+\cdots+\bV_{i-1}} \ti
\left( I_{m_i\ell^{\ga_i}} \otimes
\diag\left\{ \begin{bmatrix}0&1\\I_{e\ell^{\alp_i}-1}&0\end{bmatrix},\begin{bmatrix}0&1\\I_{e\ell^{\alp_i}-1}&0\end{bmatrix} \right\} \otimes I_{\ell^{|\bc_i|}} \right)
\ti \id_{\bV_{i+1}+\cdots+\bV_u}, \]
while if $\ell$ is unitary, let $v_{m_i,\alp_i,\ga_i,\bc_i}$ be
\[ (-1)^{m_i}\id_{\bV} \cdot \left(
\id_{\bV_0+\cdots+\bV_{i-1}} \ti
\left( I_{m_i\ell^{\ga_i}} \otimes \begin{bmatrix}0&1\\I_{2e\ell^{\alp_i}-1}&0\end{bmatrix} \otimes I_{\ell^{|\bc_i|}} \right) \ti \id_{\bV_{i+1}+\cdots+\bV_u} \right). \]
Note that $\det v_{m_i,\alp_i,\ga_i,\bc_i} =1$ since $\dim\bV$ is odd.
Set $v = \prod_{i=1}^u v_{m_i,\alp_i,\ga_i,\bc_i}$, then $v \in \SO(V)$.
It is easy to see that $R^{tw} \leq \SO(\bV)^{vF}$.

By Lang--Steinberg theorem,
there is $g \in \SO(\bV)$ such that $g^{-1}F(g)=v$.
Then there is an isomorphism
\begin{align*}
\iota:\quad \SO(\bV)^{vF}
&\to \SO(\bV)^F=H \\
x &\mapsto gxg^{-1}.
\end{align*}
By the construction and results in \cite{An94},
$\iota(R^{tw})$ is conjugate to the radical subgroup $R$ of $\SO(V)$ as in (\ref{equ-rad-1}).
Via the isomorphism $\iota$, we can transfer the problems to the twisted group $\SO(\bV)^{vF}$.
Thus we call $R_{m_i,\alp_i,\ga_i,\bc_i}^{tw}$ a twisted basic subgroup of $\SO(V)$.

\begin{rem}
The point of constructions above is that
we view each component $R_i^{tw}$ of $R^{tw}$ in (\ref{equ-rad-3}) as a subgroup of $\SO(\bV)^{vF}$ instead of a subgroup of $\SO^{\veps^{m_i}}(2m_ie\ell^{\alp_i+\ga_i+|\bc_i|},q)$ as in \cite{An94}.
Then we can multiply some elements by $-\id_{\bV}$ to adjust their determinants to be $1$;
see the definition of $v_{m_i,\alp_i,\ga_i,\bc_i}$ above when $\ell$ is unitary.
So we can avoid to involve the general orthogonal group $\GO(V_+)$ as in \cite{An94}.
See \S\ref{subsect:CN-SO} for the consideration of the normalizers of radical subgroups.
This is convenience when considering the Clifford group,
since the map $\pi: D(V) \to \GO(V)$ is not surjective when $\dim V$ is odd.
Also, the nature of multiplication by $-\id_{\bV}$ can be explained by the minus symbol in
\[ \pi(z_\bV x) = -\varrho_x \]
for any non-isotropic vector $x$ in $\bV$,
where $z_\bV$ is as in \S\ref{subsect:Clifford}.
We make the following convention.
\begin{center}
\emph{We will always work in the twisted group $\SO(\bV)^{vF}$\\
and ``tw'' will be dropped in the sequel unless otherwise stated. }
\end{center}
Additionally, all the subgroups will means subgroups of $\SO(\bV)$ and all the matrices involved mean elements in $\SO_{2n+1}(\OF_q)$ with an appropriate chosen basis of $\bV$.
For example, when we want to focus on one component, we will write $R_{m_i,\alp_i,\ga_i,\bc_i}$ as $R_{m,\alp,\ga,\bc}$, which is still viewed as a subgroup of $\SO(\bV)$.
Similar conventions will be used for all related constructions.
For example, when $\ell$ is unitary,
$v_{m,\alp,\ga,\bc}$ above will be abbreviated as
\[ (-1)^m\id_{\bV} \cdot
\left( I_{m\ell^\ga} \otimes \begin{bmatrix}0&1\\I_{2e\ell^\alp-1}&0\end{bmatrix} \otimes I_{\ell^{|\bc|}} \right). \]
Note that our method here does not apply for the special orthogonal groups of even dimension, in which case, the method in \cite{An94} to introduce $\GO(V_+)$ seems inevitable.
Also, our modification is based on the classification of radical subgroups achieved already in \cite{An94}.
\end{rem}

\subsection{Centralizers and normalizers}\label{subsect:CN-SO}
Denote by $C_{m,\alp,\ga}^0$ and $N_{m,\alp,\ga}^0$
the centralizer and normalizer of $R_{m,\alp,\ga}^0$ in $\GL(m\ell^\ga,\veps q^{e\ell^\alp})$ respectively.
By \cite[\S3.A]{FLZ20a}, $C_{m,\alp,\ga}^0$ and $N_{m,\alp,\ga}^0$ are as follows.
\begin{enumerate}[(1)]
\item
$C_{m,\alp,\ga}^0 = \GL(m,\veps q^{e\ell^\alp}) \otimes I_{\ell^\ga}$;
\item
$N_{m,\alp,\ga}^0 = C_{m,\alp,\ga}^0M_{m,\alp,\ga}^0$
is the central product of $C_{m,\alp,\ga}^0$ and $M_{m,\alp,\ga}^0$
over $C_{m,\alp,\ga}^0 \cap M_{m,\alp,\ga}^0 = \Z(E_{m,\alp,\ga}^0)$
and $M_{m,\alp,\ga}^0/\Z(E_{m,\alp,\ga}^0) \cong \Sp(2\ga,\ell)$.
\end{enumerate}
These results are essentially contained in \cite{An94}
except that part (2) improves the corresponding result in \cite{An94} slightly.

Set $C_{m,\alp,\ga} = \hbar(C_{m,\alp,\ga}^0)$
and $C_{m,\alp,\ga,\bc} = C_{m,\alp,\ga} \otimes I_{\ell^{|\bc|}}$.
When $\ell$ is linear,
set $V_{m,\alp,\ga} = \grp{v_{m,\alp,\ga},\de_{m,\alp,\ga}}$,
where $v_{m,\alp,\ga}$ is as in \S\ref{subsect:Mod-SO} and
\[ \de_{m,\alp,\ga}= (-1)^m\id_{\bV} \cdot \left( I_{m\ell^\ga} \otimes
\begin{bmatrix}0&I_{e\ell^\alp}\\I_{e\ell^\alp}&0\end{bmatrix} \right). \]
When $\ell$ is unitary,
set $V_{m,\alp,\ga} = \grp{v_{m,\alp,\ga}}$,
where $v_{m,\alp,\ga}$ is again as in \S\ref{subsect:Mod-SO}.
Note that $V_{m,\alp,\ga}$ is cyclic in both cases.
Set \[ N_{m,\alp,\ga} = \hbar(N_{m,\alp,\ga}^0)V_{m,\alp,\ga}
= C_{m,\alp,\ga} M_{m,\alp,\ga} V_{m,\alp,\ga}, \]
where $M_{m,\alp,\ga} = \hbar(M_{m,\alp,\ga}^0)$.
Set \[ N_{m,\alp,\ga,\bc}
= N_{m,\alp,\ga}/R_{m,\alp,\ga} \otimes \N_{\fS(\ell^{|\bc|})}(A_{\bc}), \]
where $\otimes$ is as in \cite[(1.5)]{AF90},
but with the sign convention in \S\ref{subsect:Mod-SO}, then
\[ N_{m,\alp,\ga,\bc}/R_{m,\alp,\ga,\bc} \cong
N_{m,\alp,\ga}/R_{m,\alp,\ga} \times \N_{\fS(\ell^{|\bc|})}(A_{\bc})/A_{\bc}, \]
where $\N_{\fS(\ell^{|\bc|})}(A_{\bc})/A_{\bc} \cong \GL(c_1,\ell) \ti \GL(c_2,\ell) \ti \cdots \ti \GL(c_r,\ell)$.

We rearrange the components of the (twisted) radical subgroup $R$ of $\SO(\bV)^{vF}$ in (\ref{equ-rad-3}) such that
\begin{equation}\label{equ-rad-4}
R = \prod_{i=1}^u R_{m_i,\alp_i,\ga_i,\bc_i}^{t_i},
\addtocounter{thm}{1}\tag{\thethm}
\end{equation}
where $(m_i,\alp_i,\ga_i,\bc_i)\neq(m_j,\alp_j,\ga_j,\bc_j)$ for any $i \neq j$.
Then we have as abstract groups that
\begin{align*}
C &:= \C_{\SO(\bV)}(R)^{vF} \cong
\SO(\bV_0)^F \ti \prod_{i=1}^u C_{m_i,\alp_i,\ga_i,\bc_i}^{t_i}, \\
N &:= \N_{\SO(\bV)}(R)^{vF}
\cong \SO(\bV_0)^F \ti \prod_{i=1}^u N_{m_i,\alp_i,\ga_i,\bc_i} \wr \fS(t_i)
\end{align*}

All the above results follows from results in \cite{An94},
except that $\det N_{m_i,\alp_i,\ga_i,\bc_i} = 1$ always holds by our construction.
Note also that the groups $\N_{\fS(\ell^{|\bc|})}(A_{\bc})$
and $\fS(t_i)$ above are subgroups of $W_1$ in \S\ref{subsect:Weyl-D0} for a suitable chosen basis.

\subsection{Weights of $\SO_{2n+1}(V)$}\label{weights-SO}
With the results above,
the weights of $\SO(V)$ can be constructed in the same way as in \cite{An94},
which we recall as follows using the similar notation as in \cite{LZ18},
since this process will be used later.
Let $\scF'_1$ be the subsets of all polynomials in $\scF_1$  whose roots are of $\ell'$-order.

Recall that $C_{m,\alp,\ga}/Z_{m,\alp,\ga} \cong \GL(m,\veps q^{e\ell^\alp})/\Z(\GL(m,\veps q^{e\ell^\alp}))$,
where $Z_{m,\alp,\ga} = \Z(R_{m,\alp,\ga})$.
By \cite[\S4]{FS82}, the set $\dz(C_{m,\alp,\ga}/Z_{m,\alp,\ga})$ is not empty if and only if $\ell\nmid m$,
in which case, any irreducible characters of defect zero of $C_{m,\alp,\ga}/Z_{m,\alp,\ga}$ is determined by a polynomial $\Delta\in\cE'_\alp$ with $d_\Delta=m$.
Let $\Ga = \cN_\alp(\Delta) \in \scF_1'$
(see \S\ref{prelim-weights} below for the definitions of $\cE'_\alp$ and $\cN_\alp$),
then $m,\alp$ are determined by $\Ga$ via $me\ell^\alp = e_\Ga \de_\Ga$
and are denoted as $m_\Ga$ and $\alp_\Ga$ respectively.
Denote $R_{m_\Ga,\alp_\Ga,\ga}$, $C_{m_\Ga,\alp_\Ga,\ga}$, $N_{m_\Ga,\alp_\Ga,\ga}$ as $R_{\Ga,\ga}$, $C_{\Ga,\ga}$, $N_{\Ga,\ga}$ respectively; similar convention will be used for other related constructions.
Denote the corresponding character of $\GL(m_\Ga,\veps q^{e\ell^{\alp_\Ga}})$ by $\theta_\Ga^0$
and the induced character of $C_{\Ga,\ga}$ by $\theta_{\Ga,\ga} = \hbar(\theta_\Ga^0 \otimes I_{\ell^\ga})$.
By \cite[\S3]{An94}, $\Ga$ determines $\theta_{\Ga,\ga}$ up to $N_{\Ga,\ga}$-conjugacy;
in the sequel, for each $\Ga$, a representative $\theta_{\Ga,\ga}$ of the $N_{\Ga,\ga}$-conjugacy class is fixed.
Obviously, $\hbar(N_{\Ga,\ga}^0) = C_{\Ga,\ga}M_{\Ga,\ga}$ stabilizes $\theta_{\Ga,\ga}$,
and $|N_{\Ga,\ga}(\theta_{\Ga,\ga}):\hbar(N_{\Ga,\ga}^0)| = \beta_\Ga e_\Ga$ by results in \cite[\S3]{An94}.

Let $M_{\Ga,\ga}^0$ be as in \S\ref{subsect:CN-SO}.
Note that $M_{\Ga,\ga}^0/\Z(E_{\Ga,\ga}^0) \cong \Sp(2\ga,\ell)$ has a unique irreducible character of defect zero, namely the Steinberg character, denoted as $\St_\ga$;
use also $\St_\ga$ to denote its inflation to $M_{\Ga,\ga}^0$ and the induced character of $M_{\Ga,\ga}$.
Let $\vartheta_{\Ga,\ga} = \theta_{\Ga,\ga}\St_\ga$,
then $\dz(C_{\Ga,\ga}M_{\Ga,\ga}/R_{\Ga,\ga} \mid \theta_\Ga) = \set{\vartheta_{\Ga,\ga}}$ and $N_{\Ga,\ga}(\theta_{\Ga,\ga}) = N_{\Ga,\ga}(\vartheta_{\Ga,\ga})$.
Finally $|\Irr(N_{\Ga,\ga}(\theta_{\Ga,\ga}) \mid \theta_{\Ga,\ga})| = \beta_\Ga e_\Ga$.
See \cite[\S3]{An94} for details.

Let $\cR_{\Ga,\de}$ be the set of all basic subgroups of the form $R_{\Ga,\ga,\bc}$ with $\ga+|\bc|=\de$.
Label the basic subgroups in $\cR_{\Ga,\de}$ as $R_{\Ga,\de,1},R_{\Ga,\de,2},\dots$,
denote the corresponding canonical character $\theta_\Ga \otimes I_{\ell^\de}$ as $\theta_{\Ga,\de,i}$,
and use this convention for all related constructions.
When $\Ga,\Ga' \in \scF_1'$ are such that $m_\Ga=m_{\Ga'}$ and $\alp_\Ga=\alp_{\Ga'}$,
label the basic subgroups in $\cR_{\Ga,\de}$ and $\cR_{\Ga',\de}$ such that $R_{\Ga,\de,i}=R_{\Ga',\de,i}$.
Set \[ \sC_{\Ga,\de} = \bigcup_i
\dz( N_{\Ga,\de,i}(\theta_{\Ga,\de,i})/R_{\Ga,\de,i} \mid \theta_{\Ga,\de,i} ). \]
By \cite[(4A)]{An94}, $|\sC_{\Ga,\de}| = \beta_\Ga e_\Ga \ell^\de$.
Assume $\sC_{\Ga,\de} = \set{\psi_{\Ga,\de,i,j}}$,
where $\psi_{\Ga,\de,i,j}$ is a character of $N_{\Ga,\de,i}(\theta_{\Ga,\de,i})$.

Recall that $e_\Ga=e$ if $\Ga\in\scF_{1,0}$ and $e_\Ga$ is the multiplicative order of $\veps_\Ga q^{\de_\Ga}$ modulo $\ell$ if $\Ga\in\scF_1\setminus \scF_{1,0}$.
Let $i\Alp(H)$ be the set of all $H^*$-conjugacy classes of triples $(s,\ka,K)$ such that
\begin{enumerate}[(1)]
\item
$s$ is a semisimple $\ell'$-elements of $H^*$;
\item
$\ka = \prod_\Ga \ka_\Ga$,
where $\ka_\Ga$ is the $e_\Ga$-core of some partition of $m_\Ga(s)$ for $\Ga\in\scF_1\setminus\scF_{1,0}$
and is the $e$-core of some Lusztig symbol of rank $\frac{m_\Ga(s)}{2}$ of odd defect for $\Ga\in\scF_{1,0}$;
\item
$K = \prod_\Ga K_\Ga$,
where $K_\Ga: \bigcup_\de \sC_{\Ga,\de} \to \set{\ell\textrm{-cores}}$ satisfying
$\sum_{\de,i,j} \ell^\de |K_\Ga(\psi_{\Ga,\de,i,j})| = \omega_\Ga$
with $\omega_\Ga$ determined by
\begin{enumerate}[(i)]
\item $m_\Ga(s) = |\ka_\Ga|+e_\Ga\omega_\Ga$ if $\Ga\in\scF_1\setminus\scF_{1,0}$;
\item $m_\Ga(s) = 2\rk\ka_\Ga + 2e\omega_\Ga$ if $\Ga\in\scF_{1,0}$.
\end{enumerate}
\end{enumerate}

Let $(R,\vph)$ be a weight of $H$ with $\vph$ lying over a canonical character $\theta$ of $C$.
There are corresponding decompositions of $C$ and $N$ as follows
\[ C = C_0 \ti C_+,\quad N = N_0 \ti N_+, \]
where $C_0=N_0=\SO(\bV_0)^F$.
Then $\theta$ and $\vph$ can be decomposed as
$\theta = \theta_0 \ti \theta_+$ and $\vph = \vph_0 \ti \vph_+$,
where $\theta_0=\vph_0 \in \dz(\SO(\bV_0)^F)$ and $\vph_+ \in \Irr(N_+\mid\theta_+)$.
Assume $\vph_0 = \chi_{s_0,\ka}$,
where $s_0$ is a semisimple $\ell'$-element of $\Sp(\bV_0^*)^F$,
$\ka=\prod_\Gamma\ka_\Gamma$ with $\ka_\Ga$ an $e_\Ga$-core for each $\Ga\in\scF'_1$.
Using the notation above, there is a decomposition
\[ \theta_+ = \prod_{\Ga,\de,i} \theta_{\Ga,\de,i}^{t_{\Ga,\de,i}}. \]
Let $s=s_0\prod_{\Ga,\de,i} s_{\Ga,\de,i}$,
where $s_{\Ga,\de,i}$ is a primary semisimple element of $\Sp(\bV_{\Ga,\de,i}^*)^{vF}$.
Then $N_+(\theta_+)$ has a corresponding decomposition
\[ N_+(\theta_+) = \prod_{\Ga,\de,i} N_{\Ga,\de,i}(\theta_{\Ga,\de,i})\wr\fS(t_{\Ga,\de,i}). \]
Then as in \cite[p.145]{LZ18},
$\vph_+$ defines a $K$ as above and
the label $(s,\ka,K)$ in $i\Alp(H)$ can be given to the weight $(R,\vph)$.

\section{Weights for special Clifford groups and spin groups} \label{sec:weights}

In this section, we classify the weights of special Clifford groups and spin groups.

\subsection{Radical subgroups of special Clifford groups}\label{subsect:radical-D0}
We start with an obvious lemma.

\begin{lem}\label{lem-radical-D0}
Let $\pi$ be the natural surjective map $\wtG \to H$.
Then the map
\[ \Rad(\wtG) \to \Rad(H),\quad \wtR \mapsto \pi(\wtR) \]
is a bijection with the inverse
\[ \Rad(H) \to \Rad(\wtG),\quad R \mapsto \cO_\ell(\pi^{-1}(R)), \]
which induces a bijection between conjugacy classes of radical subgroups.
Furthermore, if $\wtR \in \Rad(\wtG)$ and $R=\pi(\wtR)$, then $\N_H(R) = \pi(\N_{\wtG}(\wtR))$ and $\N_{\wtG}(\wtR) = \pi^{-1}(\N_H(R))$.
\end{lem}

Let $R^{tw} \leq \SO(\bV)^{vF}$ be a twisted radical subgroup of $H$ as in (\ref{equ-rad-3}).
Let $\tv_{m_i,\alp_i,\ga_i,\bc_i}$ be an element in $\wtW_0\rtimes\wtW_1$
such that $\pi(\tv_{m_i,\alp_i,\ga_i,\bc_i}) = v_{m_i,\alp_i,\ga_i,\bc_i}$
(when $\ell$ is unitary, note that
\begin{equation}\label{equ-v-W-unitary}
\begin{bmatrix}0&1\\I_{2e\ell^{\alp_i}-1}&0\end{bmatrix}=
\diag\left\{ \begin{bmatrix}0&1\\I_{e\ell^{\alp_i}-1}&0\end{bmatrix},\begin{bmatrix}0&1\\I_{e\ell^{\alp_i}-1}&0\end{bmatrix} \right\}
\cdot\begin{bmatrix} I_{e\ell^{\alp_i}-1}&0&0\\ 0&0&0&1\\ 0&0&I_{e\ell^{\alp_i}-1}&0\\ 0&1&0&0\end{bmatrix},
\addtocounter{thm}{1}\tag{\thethm}
\end{equation}
thus $v_{m_i,\alp_i,\alp_i,\bc_i} \in W_0 \rtimes W_1$).
Set $\tv=\prod_i\tv_{m_i,\alp_i,\ga_i,\bc_i}$,
then $\pi:\, D_0(\bV) \to \SO(\bV)$ induces an exact sequence
(note that $\Z(D_0(\bV))$ is connected)
\[ 1 \to \F_q^\ti\fe \longrightarrow D_0(\bV)^{\tv F} \xrightarrow{\pi} \SO(\bV)^{vF} \to 1. \]
Let $\wtR^{tw} = \cO_\ell(\pi^{-1}(R^{tw}))$,
then $\wtR^{tw}$ is the twisted version of a radical subgroup $\wtR$ of $\wtG$ by Lemma \ref{lem-radical-D0}.
As in \S\ref{sec:SO}, we make the following convention.
\begin{center}
\emph{We will always work in the twisted groups $D_0(\bV)^{\tv F}$\\
and $\SO(\bV)^{vF}$ and ``tw'' will be dropped in the sequel unless otherwise stated.}
\end{center}

Recall that each component $R_{m_i,\alp_i,\ga_i,\bc_i}$ is viewed as a subgroup of $\SO(\bV)^{vF}$.
Let $\wtR_{m_i,\alp_i,\ga_i,\bc_i} = \cO_\ell(\pi^{-1}(R_{m_i,\alp_i,\ga_i,\bc_i}))$,
then $\wtR_{m_i,\alp_i,\ga_i,\bc_i} = \wtR_{m_i,\alp_i,\ga_i} \wr A_{\bc_i}$,
where $A_{\bc_i}$ is a subgroup of $\wtW_1$.
Then $\wtR = \prod_i \wtR_i$ is a central product over $\cO_\ell\left(\F_q^\ti \fe\right)$,
where $\wtR_i = \wtR_{m_i,\alp_i,\ga_i,\bc_i}$,
and for each $i$ there is an exact sequence
\[ 1 \to \cO_\ell\left(\F_q^\ti \fe\right) \rightarrow \wtR_i \xrightarrow{\pi} R_i \to 1,  \]
Set $Z_i = Z_{m_i,\alp_i,\ga_i}$, then $Z = \Z(R) = \prod_i Z_i$.
Set $\wtZ_i = \wtZ_{m_i,\alp_i,\ga_i} = \cO_\ell(\pi^{-1}(Z_i))$,
then $\wtZ := \cO_\ell(\pi^{-1}(Z)) = \prod_i \wtZ_i$
is a central product over $\cO_\ell\left(\F_q^\ti \fe\right)$.

\subsection{}\label{prelim-weights}
To continue, we give some preliminary.

\begin{lem}\label{lem-defect0}
Let $\tchi_{\ts,\la}$ be an irreducible ordinary character of $\wtG$.
Then $\tchi_{\ts,\la}$ is the canonical character in a block of $\wtG$ with the defect group $\cO_\ell(\Z(\wtG))$
if and only if $\ts$ is a semisimple $\ell'$-element of $\wtG^*$ and $\la_\Ga$ is an $e_\Ga$-core for each $\Ga$.
\end{lem}

\begin{proof}
Similar as \cite[(11B)]{FS89}, using \cite[Prop.~11.5.6]{DM20} and Lemma \ref{lem-tcent}.
\end{proof}

\begin{lem}\label{lem-lc-cen}
There is an isomorphism
\[ \omega:\quad \wtG^*/G^* \to \Irr(Z(\wtG)),\ \ts G^* \mapsto \omega(\xi), \]
where $\ts$ is a semisimple element of $\wtG^*$ with multiplier $\xi$.
Let $\tchi_{\ts,\la}$ be an irreducible ordinary character of $\wtG$ with $\xi$ the multiplier of $\ts$,
then the above isomorphism can be chosen
such that the linear character of $\Z(\wtG)$ induced by $\tchi_{\ts,\la}$ is $\omega(\xi)$.
Thus $\grp{-\fe} \in \Ker\tchi_{\ts,\la}$ if and only if the multiplier $\xi$ is a square in $\F_q^\ti$.
\end{lem}

\begin{proof}
Let $\wtbT^*$ be a maximal torus of $\wtbG^*$ containing $\ts$.
Let $(\wtbT,\hat{\ts})$ be a pair corresponding to $(\wtbT^*,\ts)$.
Let $\wtT^*=\wtbT^{*F}$ and $T^* = \wtT^* \cap H^*$.
Via the duality, $\Irr(\Z(\wtG))$ is isomorphic to $\wtG^*/G^* \cong \wtT^*/T^*$.
So the isomorphism and the next assertion follow.
Note that the duality induces isomorphism
$\wtT^*/T^*\Z(\wtG^*) \cong \wtG^*/G^*\Z(\wtG^*) \cong \Irr(\grp{-\fe})$,
which means that $\Res^{\wtG}_{\grp{-\fe}} \tchi_{\ts,\la} = \tchi_{\ts,\la}(1)\Res^{\wtT}_{\grp{-\fe}}\hat{\ts}$,
so the last assertion follows.
\end{proof}

Beside the sets $\scF_{\xi,0}$, $\scF_{\xi,1}$, $\scF_{\xi,2}$, $\scF_\xi$,
we need some other sets of polynomials.
Let $\alp$ be a natural number.
For any monic polynomial $\Delta$ in $\F_{q^{2e\ell^\alp}}[X]$,
let $\wtDelta$ be the monic polynomial in $\F_{q^{2e\ell^\alp}}[X]$
whose roots are $\omega^{-q^{e\ell^\alp}}$ with $\omega$ running over all roots of $\Delta$.
As in \cite[p.160]{FS89}, set
\begin{align*}
\cE_{\alp,0} &= \Set{ \Delta\in\Irr\left(\F_{q^{e\ell^\alp}}[X]\right) \mid \Delta \neq X }, \\
\cE_{\alp,1} &= \Set{ \Delta\in\Irr\left(\F_{q^{2e\ell^\alp}}[X]\right) \mid \Delta \neq X, \wtDelta=\Delta }, \\
\cE_{\alp,2} &= \Set{ \Delta\wtDelta \mid \Delta\in\Irr\left(\F_{q^{2e\ell^\alp}}[X]\right), \wtDelta\neq\Delta },
\end{align*}
and 
\begin{equation}\label{def-cE}
 \cE_\alp = \begin{cases}
\cE_{\alp,0}, & \veps=1, \\
\cE_{\alp,1} \cup \cE_{\alp,2}, & \veps=-1.
\end{cases}
\addtocounter{thm}{1}\tag{\thethm}
\end{equation} 
Let $\scF'_\xi$ (resp. $\cE'_\alp$) be the subsets
of all polynomials in $\scF_\xi$ (resp. $\cE_\alp$) whose roots are of $\ell'$-order.

For any square $\xi=\xi_0^2$ ($\xi_0 \in \F_q^\ti$),
define the map
\[ \xi_0\cdot:\quad \scF_1 \to \scF_\xi,\quad \Ga \mapsto \xi_0 \cdot \Ga \]
be such that the roots of $\xi_0\cdot\Ga$ are the multiplications of roots of $\Ga$ by $\xi_0$.
Assume $\xi$ is a non-square.
If $\veps=1$, define the map
\[ \cN_{\alp,\xi}:\quad \cE_\alp \to \scF_\xi,\quad \Delta \mapsto \cN_{\alp,\xi}(\Delta) \]
be such that the set of roots of $\cN_{\alp,\xi}(\Delta)$ is the orbit of the roots of $\Delta$
under the actions of the following two operations on $\OF_q^\ti$
\[ b \mapsto b^q;\quad b \mapsto \xi b^{-1}. \]
If $\veps=-1$,
choose $\xi_1 \in \OF_q^\ti$ such that $\xi=\xi_1^{q^{e\ell^\alp}+1}$
and define the map
\[ \cN_{\alp,\xi}:\quad \cE_\alp \to \scF_\xi,\quad \Delta \mapsto \cN_{\alp,\xi}(\Delta) \]
be such that the roots of $\cN_{\alp,\xi}(\Delta)$ are
\[ (\xi_1 b)^{q^i},\quad i=0,1,\dots \]
where $b$ runs over the set of all roots of $\Delta$.
Denote also by $\xi_0\cdot$ and $\cN_{\alp,\xi}$
their restriction to the subsets $\scF'_1$ and $\cE'_\alp$ respectively.
See also \cite[p.161]{FS89} for a related map $\cN_\alp$.
See \S\ref{subsect:canonical} and Lemma \ref{lem-tN-Gamma(theta-Gamma)} below
for applications of the above definitions.

\subsection{Canonical characters}\label{subsect:canonical}

Keep the notational conventions in \S\ref{subsect:CN-SO}, \ref{subsect:radical-D0}, \ref{prelim-weights}.
Set $\wtC_i:=\wtC_{m_i,\alp_i,\ga_i,\bc_i}=\pi^{-1}(C_{m_i,\alp_i,\ga_i,\bc_i})$ and $\wtC:=\pi^{-1}(C)$.
Then $\wtC = \wtC_0 \prod_i \wtC_i$ is a central product over $\F_q^\ti\fe$,
where $\wtC_0 = D_0(\bV_0)^F$.
We also have for each $i>0$ an exact sequence
\[ 1 \to \F_q^\ti  \fe \rightarrow \wtC_i \xrightarrow{\pi} C_i \to 1.  \]
Let $\ttheta \in \dz(\wtC/\wtZ)$,
then $\ttheta = \ttheta_0 \prod_i \ttheta_i$,
where $\ttheta_0 \in \dz\left(\wtC_0/\cO_\ell(\Z(\wtC_0))\right)$
and $\ttheta_i \in \dz\left(\wtC_i/\wtZ_i\right)$.
Thus $\ttheta_0 = \chi_{\ts_0,\ka}^{\wtC_0}$,
where $\ts_0$ is a semisimple $\ell'$-element of $\CSp(\bV_0^*)^F$ with multiplier $\xi$
and $\ka_\Gamma$ is an $e_\Gamma$-core for each $\Gamma\in\scF_\xi$ by Lemma \ref{lem-defect0}.
Note that $\ttheta$, $\ttheta_0$ and $\ttheta_i$ induce the same linear character of $\F_q^\ti\fe$,
which is $\omega(\xi)$ by Lemma \ref{lem-lc-cen}.
The key step in our construction is to describe $\ttheta_i$.
We will focus on the component $\wtC_{m_i,\alp_i,\ga_i}$ in the following paragraphs
and thus drop the subscript $i$.
In the sequel,
we set $\wtM_{m,\alp,\ga} = \pi^{-1}(M_{m,\alp,\ga})$
and then $\wtC_{m,\alp,\ga}\wtM_{m,\alp,\ga}$ is a central product over $\F_q^\ti\fe$.

\subsubsection{}
Assume $\grp{-\fe} \leq \Ker\omega(\xi)$,
then $\xi$ is a square in $\F_q^\ti$ by Lemma \ref{lem-lc-cen}.
Fix $\xi_0 \in \F_q$ such that $\xi=\xi_0^2$.
Recall that we identify $\xi_0$ with an element in $\Z(\wtG^*)$,
and denote by $\hxi_0$ the image of $\xi_0$ under the isomorphism (\ref{equ:Z(G*)-Lin(G)}).
We can choose the isomorphism (\ref{equ:Z(G*)-Lin(G)})
such that $\Res^{\wtG}_{\F_q^\ti\fe}\hxi_0 = \omega(\xi)$.
Since there is an exact sequence
\[ 1 \to \grp{-\fe} \rightarrow \F_q^\ti\fe \rightarrow \wtG/G, \]
$\Res^{\wtG}_{\F_q^\ti \fe} \hxi_0$ is independent of the choice of $\xi_0$,
which is the reason why we can choose an arbitrary $\xi_0$ such that $\xi=\xi_0^2$,
but once we make such a choice, we should fix it for all the subsequent constructions.
We denote in the sequel also by $\hxi_0$ its restriction to any subgroups.
For $\ttheta \in \dz(\wtC_{m,\alp,\ga}/\wtZ_{m,\alp,\ga})$ inducing $\omega(\xi)$ on $\F_q^\ti\fe$,
we have $\ttheta = \hxi_0\theta$ with $\theta \in \dz(C/Z)$.
There is a $\Ga_0 \in \scF_1'$ such that $\theta = \theta_{\Ga_0,\ga}$, and we denote $\ttheta = \ttheta_{\Ga,\ga}$,
where $\Ga = \xi_0\cdot\Ga_0 \in \scF'_\xi$ is as in \ref{prelim-weights}.
As before, we will replace the subscripts $m,\alp$ by $\Ga$ in all the relevant constructions.
It is easy to see that $\dz(\wtC_{\Ga,\ga}\wtM_{\Ga,\ga}/\wtR_{\Ga,\ga} \mid \ttheta_{\Ga,\ga})
= \set{\tvtheta_{\Ga,\ga}:= \ttheta_{\Ga,\ga}\St_\ga}$.
Here, note that $\tvtheta_{\Ga,\ga} = \hxi_0\theta_{\Ga_0,\ga}\St_\ga$
with $\theta_{\Ga_0,\ga}\St_\ga$ viewed as a character of $\wtC_{\Ga,\ga}\wtM_{\Ga,\ga}$ by inflation.

\subsubsection{}
Assume $\grp{-\fe} \nleq \Ker\omega(\xi)$, then $\xi$ is a non-square.
To consider this case, we need some preparation.

Set $L_{m,\alp,\ga} = \hbar(\GL(m\ell^\ga,\veps q^{e\ell^\alp}))$
and $\wtL_{m,\alp,\ga} = \pi^{-1}(L_{m,\alp,\ga})$.
Note that $C_{m,\alp,\ga}M_{m,\alp,\ga} \leq L_{m,\alp,\ga}$,
thus $\wtC_{m,\alp,\ga}\wtM_{m,\alp,\ga} \leq \wtL_{m,\alp,\ga}$.
The map
\[ \hbar^*:\quad A \mapsto
\diag\{ A,F(A),\dots,F^{e\ell^\alp-1}(A),A^{-t},F(A^{-t}),\dots,F^{e\ell^\alp-1}(A^{-t}) \} \]
from $\GL(m\ell^{\ga},\veps q^{e\ell^\alp})$ to $\Sp(2me\ell^\alp,\OF_q)$
is also called a \emph{hyperbolic embedding}.
Set $L_{m,\alp,\ga}^* = \hbar^*(\GL(m\ell^{\ga},\veps q^{e\ell^\alp}))$,
then $L_{m,\alp,\ga}^*$ is in dual with $L_{m,\alp,\ga}$.
Let $\zeta$ be a generator of $\F_q^\ti$ and $\zeta_1 \in \OF_q^\ti$
satisfying that $o(\zeta_1)=(q-1)(q^{e\ell^\alp}+1)$ and $\zeta=\zeta_1^{q^{e\ell^\alp}+1}$.
Set
\begin{equation}\label{equ-tau-m,alpha,gamma}
\tau_{m,\alp,\ga} =
\begin{cases}
I_{me\ell^\alp} \otimes \diag\{1,\zeta\}, & \veps=1;\\
I_{me\ell^\alp} \otimes \diag\{ \zeta_1,\zeta_1^q,\dots,\zeta_1^{q^{e\ell^\alp-1}},
\zeta_1^{q^{e\ell^\alp}},\zeta_1^{q^{e\ell^\alp+1}},\dots,\zeta_1^{q^{2e\ell^\alp-1}} \}, & \veps=-1.
\end{cases}
\addtocounter{thm}{1}\tag{\thethm}
\end{equation}
Then $\wtL_{m,\alp,\ga}^* := \grp{\tau_{m,\alp,\ga},L_{m,\alp,\ga}^*}$ is in dual with $\wtL_{m,\alp,\ga}$.
Direct calculation shows that
\[ \Z(\wtL_{m,\alp,\ga}^*) =
\begin{cases}
\grp{\tau_{m,\alp,\ga}} \ti \Z(L_{m,\alp,\ga}^*), & \veps=1; \\
\grp{\tau_{m,\alp,\ga}}, & \veps=-1.
\end{cases} \]
Thus $|\Z(\wtL_{m,\alp,\ga}^*)| = (q-1)(q^{e\ell^\alp}-\veps)$.
The duality gives an isomorphism
\begin{equation}\label{equ-L-Z-Lin}
\Z(\wtL_{m,\alp,\ga}^*) \cong
\Irr\left( \wtL_{m,\alp,\ga}/[\wtL_{m,\alp,\ga},\wtL_{m,\alp,\ga}] \right).
\addtocounter{thm}{1}\tag{\thethm}
\end{equation}
Consequently, $|[\wtL_{m,\alp,\ga},\wtL_{m,\alp,\ga}]| = |[L_{m,\alp,\ga},L_{m,\alp,\ga}]|$,
thus $[\wtL_{m,\alp,\ga},\wtL_{m,\alp,\ga}] \cap (\F_q^\ti \fe) = \set{\fe}$
from the following exact sequence
\[ 1 \to [\wtL_{m,\alp,\ga},\wtL_{m,\alp,\ga}] \cap (\F_q^\ti \fe) \xrightarrow{}
[\wtL_{m,\alp,\ga},\wtL_{m,\alp,\ga}] \xrightarrow{\pi} [L_{m,\alp,\ga},L_{m,\alp,\ga}] \to 1. \]
So there is an injective map
\begin{equation}\label{equ-Z-tL/[tL,tL]}
\F_q^\ti \fe \hookrightarrow \wtL_{m,\alp,\ga}/[\wtL_{m,\alp,\ga},\wtL_{m,\alp,\ga}].
\addtocounter{thm}{1}\tag{\thethm}
\end{equation}
More precisely, by the duality, we have when $\veps=1$ that
\begin{equation}\label{equ-tL/[tL,tL]}
\wtL_{m,\alp,\ga}/[\wtL_{m,\alp,\ga},\wtL_{m,\alp,\ga}]
\cong (\F_q^\ti \fe) \ti L_{m,\alp,\ga}/[L_{m,\alp,\ga},L_{m,\alp,\ga}]
\cong \ZZ_{q-1} \ti \ZZ_{q^{e\ell^\alp}-1};
\addtocounter{thm}{1}\tag{\thethm}
\end{equation}
while when $\veps=-1$, we have
\begin{equation}
\wtL_{m,\alp,\ga}/[\wtL_{m,\alp,\ga},\wtL_{m,\alp,\ga}]
\cong \ZZ_{(q-1)(q^{e\ell^\alp}+1)}.
\addtocounter{thm}{1}\tag{\thethm}
\end{equation}

Assume now that $\veps=1$.
Let $k$ be the natural number such that $\xi = \zeta^k$
and set $\tau(\xi) = \tau_{m,\alp,\ga}^k$.
Denote by $\widehat{\tau(\xi)}$ the image of $\tau(\xi)$ under the isomorphism (\ref{equ-L-Z-Lin}),
which can be chosen such that
$\Res^{\wtL_{m,\alp,\ga}}_{\F_q^\ti \fe} \widehat{\tau(\xi)} = \omega(\xi)$
and $L_{m,\alp,\ga} \leq \Ker\widehat{\tau(\xi)}$.
We denote also by $\widehat{\tau(\xi)}$ its restriction to any subgroup of $\wtL_{m,\alp,\ga}$.
For $\ttheta \in \dz(\wtC_{m,\alp,\ga}/\wtZ_{m,\alp,\ga})$ inducing $\omega(\xi)$ on $\F_q^\ti\fe$,
we have $\ttheta = \widehat{\tau(\xi)}\theta$ with $\theta \in \dz(C_{m,\alp,\ga}/Z_{m,\alp,\ga})$.
Recall that $C_{m,\alp,\ga}/Z_{m,\alp,\ga} \cong \GL(m,q^{e\ell^\alp})/\Z(\GL(m,q^{e\ell^\alp}))$.
Note that $\theta = \chi_\Delta$ when viewed as a character of $\GL(m,q^{e\ell^{\alp_\Ga}})$,
where $\chi_\Delta\in\Irr(\GL(m,q^{e\ell^\alp}))$ is the irreducible character
labeled by the semisimple element of $\GL(m,q^{e\ell^\alp})$
with unique elementary divisor $\Delta\in\cE'_\alp$ of multiplicity one.
Let $\Ga = \cN_{\alp,\xi}(\Delta) \in \scF'_\xi$ with $\cN_{\alp,\xi}$ as in \S\ref{prelim-weights}.
We denote $\ttheta$ as $\ttheta_{\Ga,\ga}$.

Assume then $\veps=-1$.
Choose $\xi_1 \in \OF_q$ such that $\xi=\xi_1^{q^{e\ell^\alp}+1}$ as in \S\ref{prelim-weights}.
Let $k$ be the natural number such that $\xi_1=\zeta_1^k$
and set $\tau(\xi_1) = \tau_{m,\alp,\ga}^k$.
Denote by $\widehat{\tau(\xi_1)}$ the image of $\tau(\xi_1)$ under the isomorphism (\ref{equ-L-Z-Lin}),
which can be chosen such that $\Res^{\wtL_{m,\alp,\ga}}_{\F_q^\ti \fe} \widehat{\tau(\xi_1)} = \omega(\xi)$.
Note $\Res^{\wtL_{m,\alp,\ga}}_{\F_q^\ti \fe} \widehat{\tau(\xi_1)}$ is independent of the choice of $\xi_1$,
which is the reason why we can choose an arbitrary $\xi_1$ such that $\xi=\xi_1^{q^{e\ell^\alp}+1}$,
but once we make such a choice,
we should fix it for all the subsequent constructions.
Denote also by $\widehat{\tau(\xi_1)}$ its restriction to any subgroup of $\wtL_{m,\alp,\ga}$.
For $\ttheta \in \dz(\wtC_{m,\alp,\ga}/\wtZ_{m,\alp,\ga})$ inducing $\omega(\xi)$ on $\F_q^\ti\fe$,
we have $\ttheta = \widehat{\tau(\xi_1)}\theta$
with $\theta = \chi_\Delta\in \dz(C_{m,\alp,\ga}/Z_{m,\alp,\ga})$,
where $\Delta \in \cE'_\alp$ is as in the above case.
Let $\Ga = \cN_{\alp,\xi}(\Delta) \in \scF'_\xi$ with $\cN_{\alp,\xi}$ as in \S\ref{prelim-weights}.
Again, denote $\ttheta$ as $\ttheta_{\Ga,\ga}$.

When the canonical character considered is $\ttheta_{\Ga,\ga}$,
we will replace the subscripts $m,\alp$ in all the relevant constructions by $\Ga$ as before.
Since $\widehat{\tau(\xi)}$ or $\widehat{\tau(\xi_1)}$ is a linear character of $\wtL_{\Ga,\ga}$ which contains $\wtC_{\Ga,\ga}\wtM_{\Ga,\ga}$,
it is easy to see that $\dz(\wtC_{\Ga,\ga}\wtM_{\Ga,\ga}/\wtR_{\Ga,\ga} \mid \ttheta_{\Ga,\ga})
= \set{ \tvtheta_{\Ga,\ga} := \ttheta_{\Ga,\ga}\St_\ga }$.

\subsection{}\label{subsect:key-lemma-weights}
The key Lemma \ref{lem-tN-Gamma(theta-Gamma)} below
will make it possible to construct weight of special Clifford groups as in \S\ref{weights-SO}.
To prove this lemma, we need some preparation.

Set $\wtN_{m,\alp,\ga,\bc} = \pi^{-1}(N_{m,\alp,\ga,\bc})$.
Assume first that $\bc=(0)$ and keep the notation in \ref{subsect:CN-SO}.
When $\ell$ is linear, let $\tv_{m,\alp,\ga}$ and $\tdelta_{m,\alp,\ga}$
be pre-images of $v_{m,\alp,\ga}$ and $\de_{m,\alp,\ga}$ in $\wtW_0 \rtimes \wtW_1$ respectively;
when $\ell$ is unitary,
let $\tv_{m,\alp,\ga}$ be any pre-image of $v_{m,\alp,\ga}$ in $\wtW_0 \rtimes \wtW_1$.
Then we have that \[ \wtN_{m,\alp,\ga} = \begin{cases}
\grp{ \wtC_{m,\alp,\ga}\wtM_{m,\alp,\ga},\tv_{m,\alp,\ga},\tdelta_{m,\alp,\ga} }, & \veps=1, \\
\grp{ \wtC_{m,\alp,\ga}\wtM_{m,\alp,\ga},\tv_{m,\alp,\ga} }, & \veps=-1.
\end{cases} \]
Now,
\begin{equation}\label{equ-tN-m,alpha,gamma,c}
\wtN_{m,\alp,\ga,\bc}/\wtR_{m,\alp,\ga,\bc} \cong
\wtN_{m,\alp,\ga}/\wtR_{m,\alp,\ga} \times \N_{\fS(\ell^{|\bc|})}(A_{\bc})/A_{\bc},
\addtocounter{thm}{1}\tag{\thethm}
\end{equation}
where $\N_{\fS(\ell^{|\bc|})}(A_{\bc})/A_{\bc} \cong \GL(c_1,\ell) \ti \GL(c_2,\ell) \ti \cdots \ti \GL(c_r,\ell)$.
Here, note that $A_{\bc}$ and $\fS(\ell^{|\bc|})$ are subgroups of $\wtW_1$,
thus the right hand side of (\ref{equ-tN-m,alpha,gamma,c}) is a direct product.

\begin{lem}\label{lem-NL-m,a,gamma-D0}
Keep the notational conventions in \S\ref{subsect:canonical}.
Let $A \in \GL(m\ell^\ga,\veps q^{e\ell^\alp})$.
Denote by $|A|$ its determinant in $\GL(m\ell^\ga,\veps q^{e\ell^\alp})$.
For each $A \in \GL(m\ell^\ga,\veps q^{e\ell^\alp})$,
we can choose pre-image in $\wtL_{m,\alp,\ga}$ of $\hbar(A) \in L_{m,\alp,\ga}$ under $\pi$
such that the following hold.
\begin{enumerate}[\rm(1)]
\item
When $\veps=1$,
$\tv_{m,\alp,\ga}$ stabilizes the two components in (\ref{equ-tL/[tL,tL]}) and
\[ \tdelta_{m,\alp,\ga}^{-1} \tilde{A} \tdelta_{m,\alp,\ga}
= |A|^{1+q+\cdots+q^{e\ell^\alp-1}} \widetilde{A^{-t}}
\mod [\wtL_{m,\alp,\ga}, \wtL_{m,\alp,\ga}]. \]
\item
When $\veps=-1$, we have
\[ F_q(\tilde{A}) = \tilde{A}^q \mod [\wtL_{m,\alp,\ga}, \wtL_{m,\alp,\ga}]. \]
\end{enumerate}
\end{lem}

\begin{proof}
Modulo $[\wtL_{m,\alp,\ga}, \wtL_{m,\alp,\ga}]$, it suffices to assume $A=\diag\{b,1,\dots,1\}$.
Furthermore there is no loss of generality to assume that $m=1$, $\ga=0$.
Set then $\de:=e\ell^\alp$ and choose a basis of $\bV_{1,\alp,0}$
\[ \veps_1,\veps_2,\dots,\veps_\de, \eta_1,\eta_2,\dots,\eta_\de \]
such that $B(\veps_i,\veps_j)=B(\eta_k,\eta_l)=0$ and $B(\veps_i,\eta_j)=\de_{ij}$,
where $\de_{ij}$ is the Kronecker symbol.
Thus \[ \hbar(A) = \diag\{ b, b^q, \dots, b^{q^{\de-1}}, b^{-1}, b^{-q}, \dots, b^{-q^{\de-1}} \}. \]

Assume $\veps=1$.
We can write $\tv_{1,\alp}$ and $\tdelta_{1,\alp}$ explicitly:
\[ \tv_{1,\alp} = \prod_{i=1}^{\de-1} \tw_{i,i+1},\quad \tdelta_{1,\alp} = \prod_{i=1}^\de \tw_i, \]
where $\tw_{i,i+1},\tw_i$ are as in \S\ref{subsect:Weyl-D0}.
Note that $b \in \F_{q^\de}^\ti$ and set
\[ \tilde{A} = \prod_{i=1}^\de (\veps_i+\eta_i)(\veps_i+b^{q^{i-1}}\eta_i);\quad
\widetilde{A^{-t}} = \prod_{i=1}^\de (\veps_i+\eta_i)(\veps_i+b^{-q^{i-1}}\eta_i). \]
Then direct calculation shows that $\tv_{1,\alp}F(\tilde{A})\tv_{1,\alp}^{-1}=\tilde{A}$, so $\tilde{A} \in D_0(\bV)^{\tv F}$.
The action of $\tdelta_{1,\alp}$ follows also from direct calculation.

Assume $\veps=-1$.
We can write $\tv_{1,\alp}$ explicitly (see (\ref{equ-v-W-unitary})):
\[ \tv_{1,\alp} = \left( \prod_{i=1}^{\de-1} \tw_{i,i+1} \right) \tw_\de. \]
Note that in this case $b^{q^\de+1}=1$ and set
\[ \tilde{A} = b_0\prod_{i=1}^\de (\veps_i+\eta_i)(\veps_i+b^{q^{i-1}}\eta_i), \]
where $b_0 \in \OF_q$ is such that $b_0^{q-1}=b$,
then direct calculation shows that $\tv_{1,\alp}F(\tilde{A})\tv_{1,\alp}^{-1}=\tilde{A}$, so $\tilde{A} \in D_0(\bV)^{\tv F}$.
The assertion in this case also follows from direct calculation.
\end{proof}

\begin{lem}\label{lem-tN-Gamma(theta-Gamma)}
Keep the conventions and notation above and assume $\ttheta_{\Ga,\ga}$ is as in \S\ref{subsect:canonical}.
Then $m_\Ga e\ell^{\alp_\Ga} = e_\Ga\de_\Ga$ and $|\dz(\wtN_{\Ga,\ga}/\wtR_{\Ga,\ga}\mid\ttheta_{\Ga,\ga})| = \beta_\Ga e_\Ga$.
\end{lem}

\begin{proof}
By the construction, we have that $\wtN_{\Ga,\ga}(\ttheta_{\Ga,\ga}) = \wtN_{\Ga,\ga}(\tvtheta_{\Ga,\ga})$,
where $\tvtheta_{\Ga,\ga} = \ttheta_{\Ga,\ga}\St_\ga$ is as in \S\ref{subsect:canonical}.
Since $\wtN_{\Ga,\ga}/\wtC_{\Ga,\ga}\wtM_{\Ga,\ga} \cong \ZZ_{2e}$ is cyclic,
we may assume that $\ga=0$ and in particular $\wtC_\Ga=\wtL_\Ga$.

Assume $\xi = \xi_0^2$ ($\xi_0\in\F_q$) is a square.
Then $\ttheta_\Ga = \hxi_0 \theta_{\Ga_0}$ with $\Ga\in\scF'_1$ and $\Ga = \xi_0\cdot\Ga_0$.
Since $\hxi_0$ is a linear character of $D_0(\bV)^{vF}$,
we have that $\wtN_\Ga(\ttheta_\Ga) = \wtN_\Ga(\theta_{\Ga_0})$.
From $\Ga = \xi_0\cdot\Ga_0$,
we have that $e_\Ga = e_{\Ga_0}$.
Then the result in this case follows from the result in \cite[\S3]{An94} for $\SO(\bV)^{vF}$.

Assume $\xi$ is a non-square and $\ell$ is linear.
Then $\ttheta_\Ga = \widehat{\tau(\xi)} \theta$
with $\theta=\chi_\Delta \in \dz(C_\Ga/Z_\Ga)$ and $\Ga = \cN_{\alp_\Ga,\xi}(\Delta)$.
By Lemma \ref{lem-NL-m,a,gamma-D0} (1),
we have that $^{\tv_\Ga}\widehat{\tau(\xi)} = \widehat{\tau(\xi)}$
and $^{\tdelta_\Ga}\widehat{\tau(\xi)} = \widehat{\tau(\xi)} \cdot \hxi$,
where $\hxi$ is the image of $\xi1_{L_\Ga^*}$ under the (appropriately chosen) isomorphism
\[ \Z(L_\Ga^*) \cong \Irr(L_\Ga/[L_\Ga,L_\Ga]). \]
Note that $\tv_\Ga$ and $\tdelta_\Ga$ act on $\chi_\Delta$
as field automorphism $F_q$ and graph automorphism respectively.
Then $^{\tv_\Ga}\ttheta_\Ga = \widehat{\tau(\xi)}\chi_{F_q(\Delta)}$ and $^{\tdelta_\Ga}\ttheta_\Ga = \widehat{\tau(\xi)} \cdot \hxi\chi_{\Delta^{-1}} = \widehat{\tau(\xi)}\chi_{\xi\Delta^{-1}}$,
where the roots of $F_q(\Delta)$ and $\xi\Delta^{-1}$ are respectively the $q$-th powers of roots of $\Delta$ and $\xi$ times the inverses of roots of $\Delta$.
From this and the definition of $\cN_{\alp,\xi}$,
$\Ga$ in fact determines a $\wtN_\Ga$-conjugacy class of $\ttheta_\Ga$; see \cite[(3J)]{An94}.
The equality $m_\Ga e\ell^{\alp_\Ga} = e_\Ga\de_\Ga$ can be proved similarly as \cite[p.22]{An94}.
On the other hand,
by the above description of the actions of $\wtN_\Ga$ on $\ttheta_\Ga$,
$\wtN_\Ga(\ttheta_\Ga)$ can be determined by similar arguments in \cite[\S3]{An94}.

Assume $\xi$ is a non-square and $\ell$ is unitary.
Let $\ttheta_\Ga = \widehat{\tau(\xi_1)} \theta$
with $\theta=\chi_\Delta \in \dz(C_\Ga/Z_\Ga)$ and $\Ga = \cN_{\alp_\Ga,\xi}(\Delta)$.
By Lemma \ref{lem-NL-m,a,gamma-D0} (2),
$^{\tv_\Ga}\ttheta_\Ga = \widehat{\tau(F_q(\xi_1))}\chi_{F_q(\Delta)}
= \widehat{\tau(\xi_1)}\widehat{\xi_1^{q-1}}\chi_{F_q(\Delta)}
= \widehat{\tau(\xi_1)}\chi_{\xi_1^{q-1}F_q(\Delta)}$.
The rest is similar as the case when $\ell$ is linear.
\end{proof}

\begin{rem}
We list some special cases of the above lemma.
Assume $\xi$ is a non-square and $\ell$ is linear; recall that $e$ is odd in this case.
\begin{enumerate}[(1)]
\item
Assume $\Delta=X^2-\xi$, then $\Ga=X^2-\xi$.
Thus $\de_\Ga=\beta_\Ga=2$, $e_\Ga=e$, $m_\Ga=2$, $\alp_\Ga=0$.
Note that $\xi=\xi_0^2$ for some $\xi_0\in\F_{q^2}^\ti$ and $\xi_0^q=-\xi_0$.
From the action of $\wtN_\Ga$ in the proof of Lemma \ref{lem-tN-Gamma(theta-Gamma)},
we have $\wtN_\Ga(\ttheta_\Ga) = \wtN_\Ga$.
\item
Assume $\Delta=X\pm1$, then $\Ga=(X\pm1)(X\pm\xi) \in \scF'_{\xi,2}$.
Thus $\de_\Ga=\beta_\Ga=1$, $\veps_\Ga=1$, $e_\Ga=e$, $m_\Ga=1$, $\alp_\Ga=0$.
From the action of $\wtN_\Ga$ in the proof of Lemma \ref{lem-tN-Gamma(theta-Gamma)},
we have that $\wtN_\Ga(\ttheta_\Ga) = \grp{\wtC_\Ga,\tv_\Ga}$.
\item
Assume $\Delta=(X-b)(X-b^{q^e})$ with $b \in \F_{q^{2e}}$ and $b^{q^e+1}=\xi$,
then $\Ga=\prod_{i=0}^{2e-1}(X-b^i) \in \scF'_{\xi,1}$ (note that $b^{q^e} = \xi b^{-1}$).
Thus $\beta_\Ga=1$, $\de_\Ga=e$, $\veps_\Ga=-1$, $e_\Ga=2$, $m_\Ga=2$, $\alp_\Ga=0$.
From the action of $\wtN_\Ga$ in the proof of Lemma \ref{lem-tN-Gamma(theta-Gamma)},
we have that $\wtN_\Ga(\ttheta_\Ga) = \grp{\wtC_\Ga,\tdelta_\Ga}$.
\end{enumerate}
Assume now $\xi$ is a non-square and $\ell$ is unitary.
\begin{enumerate}[(1)]
\item
Assume $\Ga=X^2-\xi$ with $\de_\Ga=\beta_\Ga=2$.
Note that $\xi=\xi_0^2$ for some $\xi_0\in\F_{q^2}^\ti$ and $\xi_0^q=-\xi_0$.
We may assume $\xi_0=\xi_1^{\frac{q^e+1}{2}}$.
\begin{enumerate}[(i)]
\item
If $e$ is odd, then $e_\Ga=e$, $m_\Ga=2$, $\alp_\Ga=0$,
$\Delta=\left(X-\xi_1^{\frac{q^e-1}{2}}\right)\left(X-\xi_1^{-\frac{q^e(q^e-1)}{2}}\right)$
(note that $\xi_1^{\frac{q^{2e}-1}{2}} = \xi_1^{\frac{(q-1)(q^e+1)}{2}} = -1$ and $\xi_1^{-\frac{q^e(q^e-1)}{2}} = -\xi_1^{\frac{q^e-1}{2}}$).
\item
If $e$ is even, then $e_\Ga=\frac{e}{2}$, $m_\Ga=1$, $\alp_\Ga=0$, $\Delta=X-\xi_1^{\frac{q^e-1}{2}}$.
\end{enumerate}
From the action of $\wtN_\Ga$ in the proof of Lemma \ref{lem-tN-Gamma(theta-Gamma)},
we have that $\wtN_\Ga(\ttheta_\Ga) = \grp{\wtC_\Ga,\tv_\Ga}$ or $\grp{\wtC_\Ga,\tv_\Ga^2}$
according to $e$ is odd or even.
\item
Assume $\Delta=X\pm1$, then $\Ga = (X\pm\xi_1)(X\pm\xi_1^q)\cdots(X\pm\xi_1^{q^{2e-1}}) \in \scF'_{\xi,1}$.
Thus $\beta_\Ga=1$, $\veps_\Ga=-1$, $\de_\Ga=e$, $e_\Ga=1$, $m_\Ga=1$, $\alp_\Ga=0$.
Note that $\xi_1^{q^e} = \xi\xi_1^{-1}$.
From the action of $\wtN_\Ga$ in the proof of Lemma \ref{lem-tN-Gamma(theta-Gamma)},
we have that $\wtN_\Ga(\ttheta_\Ga) = \wtC_\Ga$.
\item
Assume $c \in \F_{q^{2e}}^\ti$ and $b=\xi_1c\in\F_{q^e}^\ti$
such that $\Delta=(X-c)(X-c^{-q^e})$ and $\Ga = \prod_{i=0}^{e-1}(X-b^{q^i})(X-\xi b^{-q^i}) \in \scF'_{\xi,2}$;
note that $\xi_1c^{-q^e} = \xi b^{-1}$.
Thus $\beta_\Ga=1$, $\veps_\Ga=1$, $\de_\Ga=e$, $e_\Ga=2$, $m_\Ga=2$, $\alp_\Ga=0$.
From the action of $\wtN_\Ga$ in the proof of Lemma \ref{lem-tN-Gamma(theta-Gamma)},
we have that $\wtN_\Ga(\ttheta_\Ga) = \grp{\wtC_\Ga,\tv_\Ga^e}$.
\end{enumerate}
\end{rem}

\subsection{Weights of the special Clifford groups}
\label{subsect:weights-D0}
As in \S\ref{weights-SO},
we denote $\wtR_{m_\Ga,\alp_\Ga,\ga,\bc}$ as $\wtR_{\Ga,\ga,\bc}$
when the corresponding canonical character is $\ttheta_{\Ga,\ga,\bc}$.
Let $\wtcR_{\Ga,\de}$ be the set of all the (twisted) basic subgroups
of the form $\wtR_{\Ga,\ga,\bc}$ with $\ga+|\bc|=\de$.
Label the basic subgroups in $\wtcR_{\Ga,\de}$ as $\wtR_{\Ga,\de,1},\wtR_{\Ga,\de,2},\dots$
and denote the canonical character $\ttheta_\Ga \otimes I_{\ell^\de}$ as $\ttheta_{\Ga,\de,i}$.
When $\Ga,\Ga' \in \scF'_\xi$ are such that $m_\Ga=m_{\Ga'}$ and $\alp_\Ga=\alp_{\Ga'}$,
we label the basic subgroups in $\wtcR_{\Ga,\de}$ and $\wtcR_{\Ga',\de}$
such that $\wtR_{\Ga,\de,i}=\wtR_{\Ga',\de,i}$.
Set \[ \wtsC_{\Ga,\de} = \bigcup_i
\dz\left( \wtN_{\Ga,\de,i}(\ttheta_{\Ga,\de,i})/\wtR_{\Ga,\de,i}^{tw} \mid \ttheta_{\Ga,\de,i} \right), \]
then using (\ref{equ-tN-m,alpha,gamma,c}),
we can see similarly as in \cite[(4A)]{An94} that $|\wtsC_{\Ga,\de}| = \beta_\Ga e_\Ga \ell^\de$.
Assume $\wtsC_{\Ga,\de} = \set{\tpsi_{\Ga,\de,i,j}}$,
where $\tpsi_{\Ga,\de,i,j} \in \Irr(\wtN_{\Ga,\de,i}(\ttheta_{\Ga,\de,i}))$.

Let $i\Alp(\wtG)$ be the $\wtG^*$-conjugacy classes of triples $(\ts,\ka,K)$ such that
\begin{enumerate}[(1)]
\item
$\ts$ is an $\ell'$ semisimple elements of $\wtG^*$ with multiplier $\xi$;
\item
$\ka = \prod_\Ga \ka_\Ga$,
where $\ka_\Ga$ is the $e_\Ga$-core of some partition of $m_\Ga(\ts)$ for $\Ga\notin\scF_{\xi,0}$ and
is the $e$-core of some Lusztig symbol of rank $\dfrac{m_\Ga(\ts)}{2}$ with odd defect for $\Ga\in\scF_{\xi,0}$;
\item
$K = \prod_\Ga K_\Ga$, where $K_\Ga: \bigcup_\de \wtsC_{\Ga,\de} \to \set{\ell\textrm{-cores}}$ satisfying
$\sum_{\de,i,j} \ell^\de |K_\Ga(\tpsi_{\Ga,\de,i,j})| = \omega_\Ga$ with $\omega_\Ga$ determined by
\begin{enumerate}[(i)]
\item $m_\Ga(s) = |\ka_\Ga|+e_\Ga\omega_\Ga$ if $\Ga\notin\scF_{\xi,0}$;
\item $m_\Ga(s) = 2\rk\ka_\Ga + 2e_\Ga\omega_\Ga$ if $\Ga\in\scF_{\xi,0}$.
\end{enumerate}
\end{enumerate}

Let $(\wtR,\tvphi)$ be a (twisted version of) weight of $\wtG$
and assume $\tvphi$ lies over a canonical character $\ttheta$ of $\wtC$.
Note that there are decompositions
$\wtC=\wtC_0\wtC_+$, $\wtN=\wtN_0\wtN_+$, $\ttheta = \ttheta_0\ttheta_+$ and
\[ \ttheta_+ = \prod_{\Ga,\de,i} \ttheta_{\Ga,\de,i}^{t_{\Ga,\de,i}},\quad
\wtN_+(\ttheta_+) = \prod_{\Ga,\de,i} \wtN_{\Ga,\de,i}(\ttheta_{\Ga,\de,i})^{t_{\Ga,\de,i}}\rtimes\fS(t_{\Ga,\de,i}). \]
Here, on the right hand side, $\wtN_{\Ga,\de,i}(\ttheta_{\Ga,\de,i})^{t_{\Ga,\de,i}}$ is a central product.
With the above information,
we can give a label $(\ts,\ka,K)$ in $i\Alp(\wtG)$ to $(\wtR,\tvphi)$ in the same way as in \ref{weights-SO}.

\begin{rem}
Let $(\wtR,\tvphi)$ be a weight of $\wtG$
inducing the linear character $\omega(\xi)$ of $\F_q^\ti\fe$.
When $\xi=\xi_0^2$ is a square,
let $\hxi_0$ be the linear character of $\wtG$ as in \S\ref{subsect:canonical}.
Then $(R,\hxi_0^{-1}\tvphi)$ is a weight of $H=\SO(V)$.
This gives another method for weights when $\xi$ is a square.
But the arguments we use above can give a uniform treatment
for both the cases when $\xi$ is a square or not.
\end{rem}

\subsection{Weights of spin groups}

\begin{lem}\label{lem-Lin-wtheta}
Let $\ttheta_{\Ga,\de,i}$ be a canonical character as in \S\ref{subsect:weights-D0}
and $\tz\in \Z(\wtG^*)$,
then $\hat{\tz}.\ttheta_{\Ga,\de,i} = \ttheta_{\tz.\Ga,\de,i}$.
\end{lem}

\begin{proof}
We use the construction of canonical characters in \S\ref{subsect:canonical}
and it suffices to consider the canonical character $\ttheta_\Ga \in \dz(\wtC_\Ga/\Z(\wtR_\Ga))$.
Assume $\Ga \in \scF_\xi$.

First, assume $\xi=\xi_0^2$ ($\xi_0^2 \in \F_q$) is a square in $\F_q$.
Then $\ttheta_\Ga = \hat{\xi_0}\theta_{\Ga_0}$,
where $\hat{\xi_0}$ is the image of $\xi_0$ under the isomorphism $\Z(\wtG^*) \cong \Irr(\wtG/G)$,
$\theta_{\Ga_0} \in \dz(C_\Ga/\Z(R_\Ga))$ and $\Ga = \xi_0.\Ga_0$.
Then $\hat{\tz}\ttheta_\Ga = \hat{\tz}\hat{\xi_0}\theta_{\Ga_0} = \widehat{\tz\xi_0}\theta_{\Ga_0}
= \ttheta_{\tz\xi_0.\Ga_0} = \ttheta_{\tz.\Ga}$.

Then, assume $\xi$ is a non-square and $\ell$ is linear.
Then $\wtC_\Ga = \F_q^\ti\fe \ti C_\Ga$ and $\ttheta_\Ga = \widehat{\tau(\xi)}\theta$,
where $\tau(\xi)$ is as in \S\ref{subsect:canonical},
$\widehat{\tau(\xi)}$ is the image of $\tau(\xi)$ under the isomorphism (\ref{equ-L-Z-Lin}),
and $\theta=\chi_\Delta$ is as in \S\ref{subsect:canonical},
and $\Ga = \cN_{\alp_\Ga,\xi}(\Delta)$ with $\cN_{\alp_\Ga,\xi}$ as in \S\ref{prelim-weights}.
So $\hat{\tz}\ttheta_\Ga = (\hat{\tz}\widehat{\tau(\xi)}) (\hat{\tz}\chi_\Delta)$.
Note that $\F_q^\ti\fe/\grp{-\fe}$ is isomorphic to a subgroup of $\wtC_\Ga/[\wtC_\Ga,\wtC_\Ga]$ of index $2$,
then the isomorphisms $\Z(\wtG^*) \cong \Irr(\wtG/G)$ and (\ref{equ-L-Z-Lin})
can be chosen such that $\hat{\tz} = \widehat{\tau(\tz^2)}$ when restricted to $\F_q^\ti\fe$,
so $\hat{\tz}\widehat{\tau(\xi)} = \widehat{\tau(\tz^2\xi)}$.
Since $C_\Ga \cong \GL(m_\Ga,q^{e\ell^{\alp_\Ga}})$ and the map $C_\Ga \to \wtG/G$ is surjective,
we have $\hat{\tz}\chi_\Delta = \chi_{\tz.\Delta}$.
So $\hat{\tz}.\ttheta_\Ga = \widehat{\tau(\tz^2\xi)}\chi_{\tz.\Delta}= \ttheta_{\cN_{\alp_\Ga,\tz^2\xi}(\tz.\Delta)}
= \ttheta_{\tz.\Ga}$ by the definition of $\cN_{\alp_\Ga,\tz^2\xi}$.

Finally, assume $\xi$ is a non-square and $\ell$ is unitary.
Then $\wtC_\Ga/[\wtC_\Ga,\wtC_\Ga] \cong \mathbb Z_{(q-1)(q^{e\ell^{\alp_\Ga}}+1)}$
and $\ttheta_\Ga = \widehat{\tau(\xi_1)}\theta$,
where $\xi_1$ and $\tau(\xi_1)$ is as in \S\ref{subsect:canonical},
$\widehat{\tau(\xi_1)}$ is the image of $\tau(\xi_1)$ under the isomorphism (\ref{equ-L-Z-Lin}),
$\theta=\chi_\Delta$ is as in \S\ref{subsect:canonical},
and $\Ga = \cN_{\alp_\Ga,\xi}(\Delta)$ with $\cN_{\alp_\Ga,\xi}$ as in \S\ref{prelim-weights}.
So $\hat{\tz}.\ttheta_\Ga = \hat{\tz}\widehat{\tau(\xi_1)}\chi_\Delta$.
Note that the map $\wtC_\Ga \to \wtG/G$ is surjective,
then the isomorphisms $\Z(\wtG^*) \cong \Irr(\wtG/G)$ and (\ref{equ-L-Z-Lin})
for the case when $\ell$ is unitary
can be chosen such that $\hat{\tz} = \widehat{\tau(\tz)}$ when viewed as characters of $\wtC_\Ga$,
so $\hat{\tz}\widehat{\tau(\xi_1)} = \widehat{\tau(\tz\xi_1)}$ in this case.
So $\hat{\tz}.\ttheta_\Ga = \widehat{\tau(\tz\xi_1)}\chi_{\Delta}
= \ttheta_{\cN_{\alp_\Ga,\xi'}(\Delta)} = \ttheta_{\tz.\Ga}$
with $\xi'=(\tz\xi_1)^{q^{e\ell^{\alp_\Ga}}+1}$
by the definition of $\cN$ for the case when $\ell$ is unitary.
\end{proof}

\begin{prop}\label{Lin-Alp(tG)}
Let $(\wtR,\tvphi)$ be a weight of $\wtG$ with label $(\ts,\ka,K)$ and $\hat{\tz} \in \Lin_{\ell'}(\wtG/G)$,
then $\hat{\tz}\cdot(\wtR,\tvphi)=(\wtR,\hat{\tz}.\tvphi)$ has label $(\tz\ts,\tz.\ka,\tz.K)$.
\end{prop}

\begin{proof}
This follows similarly as \cite[Prop.~5.12]{Fe19} with \cite[Lemma 5.10]{Fe19} replaced by Lemma \ref{lem-Lin-wtheta}.
\end{proof}

\begin{lem}\label{lem:rad-spin}
Let $\wtR$ be a radical subgroup of $\wtG$.
\begin{enumerate}[\rm(1)]
\item
We have that $\wtR = \left(G\cap\wtR\right)\mathcal{O}_\ell(\Z(\wtG))$
and the map $\wtR \mapsto G\cap\wtR$ induces a bijection
between the set of radical subgroups of $\wtG$ and that of $G$.
\item
We have that $\wtC=\C_{\wtG}(\wtR)=\C_{\wtG}(G\cap\wtR)$, $\wtN=\N_{\wtG}(\wtR)=\N_{\wtG}(G\cap\wtR)$,
and thus $C:=\C_G(G\cap\wtR)=G\cap\wtC$ and $N:=\N_G(G\cap\wtR)=G\cap\wtN$.
\item
The inclusion induces surjective maps $\wtC \to \wtG/G$ and $\wtN \to \wtG/G$.
\item
The map $\wtR \mapsto G\cap\wtR$ induces a bijection
between the set of $\wtG$-conjugacy classes of radical subgroups of $\wtG$ and that of $G$.
\end{enumerate}
\end{lem}

\begin{proof}
(1) and (2) follows from the facts that $\ell$ is odd,
$|\wtG:G\Z(\wtG)|=2$ and $G\Z(\wtG)$ is a central product of $G$ and $\Z(\wtG)$ over $\Z(G)$ with $|\Z(G)|=2$.

For (3), it suffices to prove $\wtC \to \wtG/G$ is surjective,
which follows if we prove that the map $\wtC_{m,\alp,\ga,\bc} \to \wtG/G$ from one component is surjective.
Without loss of generality,
we may assume as in the proof of Lemma \ref{lem-NL-m,a,gamma-D0}
that $m=1$, $\ga=0$ and set $\de:=e\ell^\alp$,
and choose a basis $\veps_1,\veps_2,\dots,\veps_\de,\eta_1,\eta_2,\dots,\eta_\de$ of $\bV_{1,\alp,0}$
as in the proof of Lemma \ref{lem-NL-m,a,gamma-D0}.
We use some constructions in Lemma \ref{lem-NL-m,a,gamma-D0}.
When $\ell$ is linear, set
\[ \tilde{A}(b) = \prod_{i=1}^\de (\veps_i+\eta_i)(\veps_i+b^{q^{i-1}}\eta_i),\quad b\in \F_{q^\de}^\ti. \]
Note that
\[ N(\tilde{A}(b)) = \prod_{i=1}^\de (\veps_i+\eta_i)(\veps_i+b^{q^{i-1}}\eta_i)(\veps_i+b^{q^{i-1}}\eta_i)(\veps_i+\eta_i)
= b^{1+q+\cdots+q^{\de-1}}. \]
Then $\tilde{A}(b) \in \wtC_{1,\alp,0}$
and the image of the set of all $\tilde{A}(b)$ with $b$ running through $\F_{q^\de}^\ti$ is already surjective.
When $\ell$ is unitary,
set \[ \tilde{A}(b_0) = b_0\prod_{i=1}^\de (\veps_i+\eta_i)(\veps_i+b^{q^{i-1}}\eta_i), \]
where $b_0 \in \bar{F}_p$ is such that $b_0^{q-1}=b$ and $o(b_0)\mid(q-1)(q^\de+1)$.
Note that \[ N(\tilde{A}(b_0))
= b_0^2\prod_{i=1}^\de (\veps_i+\eta_i)(\veps_i+b^{q^{i-1}}\eta_i)(\veps_i+b^{q^{i-1}}\eta_i)(\veps_i+\eta_i)
= b_0^2b^{1+q+\cdots+q^{\de-1}} = b_0^{q^\de+1}. \]
Then $\tilde{A}(b_0) \in \wtC_{1,\alp,0}$
and the image of the set of all $\tilde{A}(b_0)$
with $b_0$ running through all elements in $\bar{F}_p^\ti$ with order dividing $(q-1)(q^\de+1)$
is already surjective.

(4) follows from (3),
since $|\wtG:\wtN|=|G:N|$ by (3).
\end{proof}

We note that Lemma \ref{lem:rad-spin} (3) generalised \cite[Lemma 2.15 and 2.16]{FLZ19}.

\begin{prop}
Let $(\wtR,\tvphi)$ be a weight of $\wtG$ with label $(\ts,\ka,K)$.
Use the notation in the previous lemma.
Let $\vph$ be an irreducible constituent of $\Res^{\wtN}_N \tvphi$.
Then $(G\cap\wtR,\vph)$ is a weight of $G$ and all weights of $G$ can be obtained by this way.
The number of irreducible constituents of $\Res^{\wtN}_N \tvphi$ is $1$ or $2$,
and this number is $2$ if and only if $\widehat{-1}$ fixes the label $(\ts,\ka,K)$.
\end{prop}

Then a parametrization of weights of $G$ follows from the above proposition.

\section{Blocks of special Clifford groups and spin groups} \label{sec:blocks}

We should consider the blocks.
The results on the classification for the blocks of the finite reductive groups via $d$-cuspidal pairs (see \S \ref{subsec:pre,lie}) can be used to determined the blocks of the finite special Clifford groups $\wtG$ and in this section we  shall follow this method.

Keep the notation in \S \ref{subsec:cores}.
Let $\ts$ be a semisimple element of $\wtG^*$ with multiplier $\xi$.
For $\Ga\in\scF_\xi$, denote by $\fC_\Ga(\ts)$ the set of  $\ka_\Ga$ such that there exists $\mu_\Ga\in\Psi_\Ga(\ts)$ satisfying that $\ka_\Ga$ is an $e_\Ga$-core of $\mu_\Ga$. 
Denote $\fC(\ts):=\prod\limits_{\Ga\in\scF_\xi} \fC_\Ga(\ts)$.

\subsection{The blocks of finite special Clifford groups}\label{subsec:blocks-clifford}

Let $d=d(q,\ell)$ be defined as in (\ref{equ:def-d-ql}).
Let $\wtbL^*$ be a $d$-split Levi subgroup of $\wtbG^*$ and $\bM^*=\wtbL^*\cap \bH^*$. Then $\bM^*$ is an $F$-stable $d$-split Levi subgroup of $\bH^*$ and $\wtbL^*=\Z(\wtbG^*)\bM^*$.
Let $\bV^*=\bV_0^*\perp \bV_+^*$ be an orthogonal decomposition such that $\bM^*=\bM_0^*\ti\bM_+^*$,  $\bM_0^*=\Sp(\bV_0^*)$ and $\bM_+^*\le \Sp(\bV_+^*)$.
Let $M^*={\bM^*}^F$, $M^*_0={\bM_0^*}^F$ and $M^*_+={\bM_+^*}^F$.

Let $\wtL^*=(\wtbL^*)^F=\langle \tau, M^*_0 \rangle M^*_+$ is a central product of $\langle \tau, M_0^* \rangle$ and $M+^*$, where $\tau$ satisfies that $\wtG^*=\langle \tau, G^* \rangle$.
Here $M_0^*=\Sp(V^*_0)$ and $M_+^*$ is isomorphic to a direct product of general linear or unitary groups $\GL_{-}(\veps q^e)$.
Moreover, $\tau=\tau_0\ti\tau_+$ where $\tau_i\in\CSp(V_i^*)$ with $i\in\{0,+\}$ and $[\tau_+,M_+]=1$.

Let $\ts$ be a semisimple $\ell'$-element of $\wtG^*=(\wtbG^*)^F$ with multiplier $\xi$ such that $\ts\in \wtbL^*$.
We first suppose that $M^*_+\cong \prod\limits_{\Ga\in\scF_\xi}\prod\limits_{i=1}^{w_\Ga(\ts)} \GL_{\frac{e_\Ga \de_\Ga}{e}}(\veps q^e)$, where $w_\Ga(\ts)$ is an non-negative integer.
In order to ensure $\ts\in \wtbL^*$ we assume that $e_\Ga w_\Ga(\ts)\le m_\Ga(\ts)$.
Let $\ts=\ts_0\ti \ts_+$ be the corresponding decomposition of $\ts$ in $\wtbL^*$ so that $\ts_0\in \langle \tau, L^*_0 \rangle$ and $\ts_+\in L_+^*$.
Then $\C_{\bM^*_+}(\ts_+)$ is a torus and
$\C_{M^*_+}(\ts_+)\cong \prod\limits_{\Ga\in\scF_\xi}\prod\limits_{i=1}^{w_\Ga(\ts)} \GL_1((\veps q^e)^{e_\Ga \de_\Ga/e})$.
Let $V_0=\perp_\Ga V_{0,\Ga}$ be the decomposition of $V_0$ corresponding to $\ts_0 = \prod_\Ga \ts_{0,\Ga}$.
Then
$\C_{M^*_0}(\ts_0)=\prod\limits_{\Ga} C_\Ga(\ts_{0,\Ga})$,
where $C_\Ga(\ts_{0,\Ga})$ is defined as in \S \ref{subsec:centralizer}.
Let $\wtbL$ be an $F$-stable $d$-split Levi subgroup of $\wtbG$ dual to $\wtbL^*$ as in \cite[Prop.~13.9]{CE04}.

We let $\la=\la_0\ti\la_+$ with $\la_i\in \Irr(\C_{M^*_i}(\ts_i))$ with $i\in\{0,+\}$ and $\la_+=1_{\C_{M^*_+}(\ts_+)}$. \label{def-la0}
In addition, we write $\la_0=\prod\limits_{\Ga} \la_{0,\Ga}$ with $\la_{0,\Ga}\in\Irr(C_\Ga(\ts_{0,\Ga}))$. Let $\ka\in\fC(\ts)$.
We assume that $\ka_{\Ga}$ is the partition (resp. Lusztig symbol) corresponding to the unipotent character $\la_{0,\Ga}$ if $\Ga\in\scF_{\xi,1}\cup\scF_{\xi,2}$ (resp. $\Ga\in\scF_{\xi,0}$).
Then by \cite[Cor. 4.6.5 and 4.6.16]{GM20}, $\la_{0,\Ga}$ is a $d$-cuspidal unipotent character and thus $\la_0$ is $d$-cuspidal.
Moreover, $\la$ is also $d$-cuspidal. Let $\wtL^*=(\wtbL^*)^F$.
We mention that $[\C_{\wtL^*}(\ts),\C_{\wtL^*}(\ts)]=[\C_{M^*}(\ts),\C_{M^*}(\ts)]$, so there exists $\widetilde\la\in\Irr(\C_{\wtL^*}(\ts))$ such that $\la=\Res_{\C_{M^*}(\ts)}^{\C_{\wtL^*}(\ts)}(\widetilde \la)$.
By \cite[Lemma 2.3]{KM15}, $\widetilde\la$ be the $d$-cuspidal unipotent character of $\C_{\wtL^*}(\ts)$.
Let $\tzeta=\fJ_{\ts}^{\wtbL}(\widetilde\la)$.
Then $(\wtbL,\tzeta)$ be a $d$(-Jordan)-cuspidal pair of $\wtbG$ with $\tzeta\in\cE(\wtbL^F,\ts)$. 
We mention that the ${\wtbG}^F$-conjugacy class of $d$-cuspidal pair $(\wtbL,\tzeta)$ above can be uniquely determined by the $(\wtbG^*)^F$-conjugacy class of the pair $(\ts,\ka)$ with semisimple $\ell'$-element and $\ka\in\fC(\ts)$.

Let $\wtbG(\ts)$ be the (connected) reductive group defined as in \cite[3.1]{En08}, which is dual to $\C_{\wtbG^*}(\ts)$.
We will make use of \cite[Thm.~1.6]{En08}, a result of Enguehard and we note that this result holds for classical groups and odd primes since the Mackey formula holds for groups of classical type (cf. \cite{BM11,Ta18}).
Then by \cite[Thm.~1.6 (B)]{En08} there is a bijection between the unipotent $\ell$-blocks of $\wtbG(\ts)^F$ and the $\ell$-blocks in $\cE(\wtbG^F,\ts)$.
Let $b$ be the unipotent $\ell$-block of $\wtbG(\ts)^F$ corresponding to $b_{(\wtbG^*)^F}(\wtbL,\tzeta)$.
By \cite[Thm.]{CE94}~(or \cite[Thm.~22.9]{CE04}),
the unipotent blocks of $\wtbG(\ts)$ are labeled by unipotent $d$-cuspidal pairs. We assume that $(\wtbL(\ts),\tzeta(\ts))$ is the unipotent $d$-cuspidal pair of $\wtbG(\ts)$ such that $b=b_{\wtbG(\ts)^F}(\wtbL(\ts),\tzeta(\ts))$.
Now we recall the construction of $(\wtbL(\ts),\tzeta(\ts))$ given in \cite[\S 3]{En08}: $\wtbL(\ts)$ is dual to $\C_{\wtbL^*}(\ts)$ and $\tzeta(\ts)=\fJ_{1}^{\wtbL(\ts)}(\widetilde\la)$.

For every semisimple $\ell$-element $\tilde t$ of $C_{\wtbG^*}(\ts)$, we consider the characters in $\Irr(b)\cap \cE(\wtbG(\ts)^F,\tilde t)$ now.
Let $\xi'$ be the multiplier of $\tilde t$. Now $\xi'$ is of $\ell$-order since $\tilde t$ is an $\ell$-element, thus $\xi'\in (\F_q^\ti)^2$, \emph{i.e.}, there exists $\xi'_0\in\F_q^\ti$ such that $\xi'={\xi'_0}^2$.
Thus the only possible elementary divisor of $\tilde t$ in $\scF_{\xi',0}$ is $X-\xi_0'$.
Denote by $\wtbG(\ts)(\tilde t)$ the $F$-stable Levi subgroup of $\wtbG(\ts)$ in duality with $\C_{\wtbG(\ts)^*}(\tilde t)$.
By \cite[Thm.~(iii)]{CE94}, if $\Irr(b)\cap \cE(\wtbG(\ts)^F,\tilde t)$ is not empty, then $\wtbL(\ts)$ is also a $d$-split Levi subgroup of $\wtbG(\ts)(\tilde t)$.
In addition, $\Irr(b)\cap \cE(\wtbG(\ts)^F,\tilde t)$ are the characters occurring in some $R^{\wtbG(\ts)}_{\wtbG(\ts)(\tilde t)}(\hat {\tilde t}\cdot\tchi_{\tilde t})$, where $\tchi_{\tilde t}$ is a unipotent character of $\wtbG(\ts)(\tilde t)$ in the $d$-Harish-Chandra series
$\cE(\wtbG(\ts)(\tilde t), (\wtbL(\ts),\tzeta(\ts)))$.
By \cite[Thm.~3.3.22]{GM20}, $\fJ_{\tilde t}^{\wtbG(\ts)}(\tchi_{\tilde t})=\pm R^{\wtbG(\ts)}_{\wtbG(\ts)(\tilde t)}(\hat {\tilde t}\cdot\tchi_{\tilde t})$.

Now we consider the $d$-Harish-Chandra series
$\cE(\wtbG(\ts)(\tilde t), (\wtbL(\ts),\tzeta(\ts)))$.
Let $\bH(\ts)$ be dual to $\C_{\bH^*}(\ts)$, 
$\bM(\ts)$ be dual to $\C_{\bM^*}(\ts)$,
and $\bH(\ts)(\tilde t)$ be the $F$-stable Levi subgroup of $\bH(s)$ dual to $\C_{\bH(s)^*}(\tilde t)$.
Then the sequences
$$1\to \overline{\F}_q^\ti \fe \to \wtbG(\ts)\to \bH(\ts)\to 1$$
$$1\to \overline{\F}_q^\ti \fe \to \wtbL(\ts)\to \bM(\ts)\to 1$$
and
$$1\to \overline{\F}_q^\ti \fe \to \wtbG(\ts)(\tilde t)\to \bH(\ts)(\tilde t)\to 1$$
are exact.
Let $\la(\ts)=\fJ_{1}^{\bM(\ts)}(\la)$.
Then $\tzeta(\ts)$ is the inflation of $\la(\ts)$ to $\wtbL(\ts)^F$.
Since we only consider the unipotent characters here, we may as well transfer 
from $\wtbG(\ts)$, $\wtbG(\ts)(\tilde t)$, $\wtbL(\ts)$, $\tzeta(\ts)$
to $\bH(\ts)$, $\bH(\ts)(\tilde t)$, $\bM(\ts)$, $\la(\ts)$, \emph{i.e.}, we consider the $d$-Harish-Chandra series
$\cE(\bH(\ts)(\tilde t), (\bM(\ts),\la(\ts)))$ since $(\bM(\ts),\la(\ts))$ is a $d$-cuspidal pair of $\bH(\ts)(\tilde t)$.
We note that ${\bH(\ts)^*}^F=\prod\limits_{\Ga\in\scF_\xi}C_\Ga(\ts_\Ga)$ so that 
\begin{equation}\label{equ:Hs}
\bH(\ts)^F\cong\prod\limits_{\Ga\in\scF_{\xi,0}}\SO_{m_\Ga(\ts)+1}(q^{d_\Ga})\ti \prod\limits_{\Ga\in\scF_{\xi,1}\cup\scF_{\xi,2}}\GL_{m_\Ga(\ts)}(\veps_\Ga q^{\de_\Ga}).
\addtocounter{thm}{1}\tag{\thethm}
\end{equation}
Let $\tilde t=\prod_\Ga \tilde t_\Ga$ with $\tilde t_\Ga\in C_\Ga(\ts_\Ga)$.
Then $\C_{\bH(\ts)^*}(t)^F=\prod\limits_{\Ga\in\scF_\xi}\C_{C_\Ga(\ts_\Ga)}(\tilde t_\Ga)$.
We mention that by \ref{equ:cen-Sp},
$\C_{C_\Ga(\ts_\Ga)}(\tilde t_\Ga)=\prod\limits_{\Delta} C_{\Ga,\Delta} (\ts_\Ga,\tilde t_\Ga)$ where $\Delta$ runs through the elementary divisor of $\tilde t_\Ga$ and $C_{\Ga,\Delta}(\ts_\Ga,\tilde t_\Ga)$ is isomorphic to a general linear or unitary group if $\Ga\in \scF_{\xi,1}\cup\scF_{\xi,2}$ or $\Delta\ne X-\xi_0'$, while $C_{\Ga,\Delta} (\ts_\Ga,\tilde t_\Ga)$ is a subgroup of $C_\Ga(\ts_\Ga)$ with the same type if
$\Ga\in \scF_{\xi,0}$ and $\Delta=X-\xi_0'$.
Let $\psi$ be a unipotent character of $\bH(\ts)(\tilde t)$ with $\psi=\prod\limits_{\Ga}\prod\limits_{\Delta} \psi_{\Ga,\Delta}$, where $\psi_{\Ga,\Delta}$ is the unipotent character of $C_{\Ga,\Delta}(\ts_\Ga,\tilde t_\Ga)^*$.
Let $\mu_{\Ga,\Delta}$ be the partition corresponding to $\psi_{\Ga,\Delta}$ if $\Ga\in \scF_{\xi,1}\cup\scF_{\xi,2}$ or $\Delta\ne X-\xi_0'$ while $\mu_{\Ga,\Delta}$ the Lusztig symbol corresponding to $\psi_{\Ga,\Delta}$ if
$\Ga\in \scF_{\xi,0}$ and $\Delta=X-\xi_0'$.
Then by \cite[Cor.~4.6.5, 4.6.7 and 4.6.16]{GM20},
$\mu_{\Ga,X-\xi_0'}$ has $e_\Ga$-core $\ka_\Ga$ for every elementary divisor $\Ga$ of $\ts$.

From this, we give the characters in $\ell$-block $b_{(\wtbG^*)^F}(\wtbL,\tzeta)$.
For every semisimple $\ell$-element $\tilde t$ of $\C_{\wtbG^*}(\ts)$ so that $\C_{\wtbG(\ts)^*}(\tilde t)=\C_{\wtbG^*}(\ts\tilde t)$, by \cite[Thm.~1.6 (B.1.b)]{En08},
\begin{equation}\label{Jor-blocks-char}
\Irr(b_{(\wtbG^*)^F}(\wtbL,\tzeta))\cap \cE(\wtbG^F,\ts\tilde t)=\fJ_{\ts\tilde t}^{\wtbG}\circ (\fJ_{\tilde t}^{\wtbG(\ts)})^{-1} (\Irr(b)\cap \cE(\wtbG(\ts)^F,\tilde t)).
\addtocounter{thm}{1}\tag{\thethm}
\end{equation}
Therefore, we have the following.

\begin{thm}
\begin{enumerate}[\rm(1)]
	\item The map $(\ts,\ka)\mapsto \wtB(\ts,\ka)$ is a bijection from  the set of the $\wtG^*$-conjugacy classes of the pairs $(\ts,\ka)$, where $\ts$ is a semisimple $\ell'$-element of $\wtG^*$ and $\ka\in\fC(\ts)$, to the $\ell$-blocks of $\wtG$.
	\item Let $\tilde t$ be a semisimple $\ell$-element of $\wtG^*$ commuting with $\ts$ and $\mu'\in\Psi(\ts \tilde t)$. Then the character $\tchi_{\ts\tilde t,\mu'}$ lies in block $\wtB(\ts,\ka)$ if and only if $\mu'_{\Ga_{\tilde t}}$ has $e_\Ga$-core $\ka_\Ga$ for every elementary divisor $\Ga$ of $\ts$.
\end{enumerate}
\end{thm}

Here the polynomial $\Ga_{\tilde t}$ is defined as follows. Let $\Ga\in\scF_\xi$ for some $\xi$ be a polynomial which has a root $\alp\in\overline{\F}_q^\ti$ of $\ell'$-order and let $\tilde t$ be a semisimple $\ell$-element with multiplier $\xi'$. So there exists $\xi_0'\in\F_q^\ti$ such that $\xi'={\xi_0'}^2$.
Then we wet $\Ga_{\tilde t}$ to be the polynomial in $\scF_{\xi\xi'}$ which has a root $\alp\xi_0$.

\begin{proof}
By	the above arguments, from a pair $(\ts,\ka)$ we can define a $d$-cuspidal pair $(\wtbL,\tzeta)$ of $(\wtbG,F)$ and Thus we have an $\ell$-block
$\wtB(\ts,\ka)=b_{\wtbG^F}(\wtbL,\tzeta)$.	Then (ii) follows from the above arguments and (\ref{Jor-blocks-char}).
Also, it is easy to check that the union of $\Irr(\wtB(\ts,\ka))$ are exactly $\cE_\ell(\wtbG^F,\ts)$ when $\ka$ runs through $\fC(\ts)$. Thus we complete the proof.
\end{proof}

\subsection{The defect groups and Brauer pairs}

We give a family of Brauer pairs (associated with certain radical $\ell$-subgroups) of $\wtB=\wtB(\ts,\ka)=b_{\wtbG^F}(\wtbL,\tzeta)$.
For an $\ell$-subgroup of $H$,
we let $\wtR$ be the Sylow $\ell$-subgroup of $\pi^{-1}(R)$.  Then $\pi^{-1}(R)=\wtR\ti \mathcal O_{\ell'}(\Z(\wtG))$ and we have exact sequences
$$1\to \mathcal{O}_\ell(\F_q^\ti\fe) \to \wtR\to R\to 1,$$
$$1\to \overline{\F}_q^\ti \fe \to \C_{\wtbG}(\wtR)\to \C_{\bH}(R)\to 1$$
and 
$$1\to \F_q^\ti \fe \to \C_{\wtG}(\wtR)\to \C_{H}(R)\to 1.$$

If $R$ is a radical $\ell$-subgroup of $H=\SO(V)$, then we have a orthogonal decomposition of 
\begin{equation}\label{qua:dec-V}
V=V_0\perp V_1 \perp\cdots\perp V_u,
\addtocounter{thm}{1}\tag{\thethm}
\end{equation}
 and
$R$ is conjugate to 
\begin{equation}\label{qua:dec-R}
R_0 \ti R_1 \ti \cdots \ti R_u,
\addtocounter{thm}{1}\tag{\thethm}
\end{equation}
where $R_0$ is a trivial group of $\SO(V_0)$ and $R_i=R_{m_i,\alp_i,\ga_i,\bc_i}$ ($i\geqslant1$) is a basic subgroup of $\SO(V_i)$, as described in \S \ref{subsec:4.2}.
Thus $\C_H(R)=C_0\ti C_1\ti\cdots\ti C_u$, where $C_i\le \C_{\SO(V_i)}(R_i)$ such that $C_0=\SO(V_0)$ and $C_i\cong \GL_{m_i}(\veps q^{e\ell^{\alp_i}})$ for $i\ge 1$.
In fact, $\C_{\bH}(R)$ is a connected reductive group.
Denote $C_+=\prod\limits_{i=1}^{u} C_i$.

For $\Ga\in\scF_\xi$ for some $\xi\in\F_q^\ti$, recall that we let $m_\Ga$ and $\alp_\Ga$ be non-negative integers (determined by $\Ga$) such that $m_\Ga e\ell^{\alp_\Ga}=e_\Ga \de_\Ga$ and $\ell\nmid m_\Ga$.
We define a Brauer pair now.
Assume that $V_0=\bV_0^F$ and $V_+=\bV_+^F$  
where $V_0$ and $V_+=V_1 \perp\cdots\perp V_u$ as in (\ref{qua:dec-V}) and $\bV_0$, $\bV_+$ are defined from $(\ts,\ka)$ as in \S \ref{subsec:blocks-clifford}.
Then $\ts_i\in \CSp(V^*_i)$ for $i\in\{0,+\}$.
We let $V_+^*=\perp_\Ga(V_+^*)_\Ga$ such that $(\ts_i)_\Ga\in \CSp((V^*_+)_\Ga)$.
Thus we rewrite (\ref{qua:dec-V}) and (\ref{qua:dec-R}):
$$V=V_0\perp V_+,\ V_+=\perp_\Ga(V_+)_\Ga,\ (V_+)_\Ga=\perp_i (V_+)_{\Ga,i}$$ and 
\begin{equation}\label{equ:dec-R-bloc}
R=R_0\ti R_+,\ R_+=\prod_\Ga (R_+)_\Ga,\ (R_+)_\Ga=\prod_i (R_+)_{\Ga,i},
\addtocounter{thm}{1}\tag{\thethm}
\end{equation}
where $R_0$ is the trivial subgroup of $\SO(V_0)$ and $(R_+)_{\Ga,i}$ is a basic subgroup of $\SO((V_+)_{\Ga,i})$.
We assume that $(R_+)_{\Ga,i}=R_{m_\Ga,\alp_\Ga,*,*}$ for all $i$.

We let $\wtR$ be the Sylow $\ell$-subgroup of $\pi^{-1}(R)$.
According to \cite[Prop.~1.2.4]{En08}, there exists an $\ell$-subgroup $R'$ (resp. $\wtR'$) of ${\bH^*}^F$ (resp. $(\wtbG^*)^F$) such that $\C_{\bH}(R)$ and $\C_{\bH^*}(R')$ (resp. $\C_{\wtbG}(\wtR)$ and $\C_{\wtbG^*}(\wtR')$)
are in duality and satisfy additional properties.
We denote by $\C_{\bH}(R)^*$ (resp. $\C_{\wtbG}(\wtR)^*$) for $\C_{\bH^*}(R')$ (resp. $\C_{\wtbG^*}(\wtR')$).

Let $\bL^*_{R,\ts}:=\C_{\C_{\bH}(R)^*}(\ts)$.
Then $\C_H(R)^*=C_0^* \ti C_+^*$, where $C_0^*=\Sp(V_0^*)$ and ${\bL^*_{R,\ts}}^F=\C_{\C_H(R)^*}(\ts)=\C_{C_0^*}(\ts_0)\ti \C_{C_+^*}(\ts_+)$.
Then $\C_{C_+^*}(\ts_+)\cong \prod\limits_{\Ga,i} \GL_1(\veps^{m_\Ga} q^{e_\Ga \de_\Ga})$.
Note that $C_0=M_0$.
Then we let $\la_R=\la_0\ti 1_{C_{C_+^*}(\ts_+)}$, where $\la_0$ is a unipotent $d$-cuspidal character of $\C_{C_0^*}(\ts_0)$ defined as in page \pageref{def-la0}.
Thus $\la_R$ is a unipotent $d$-cuspidal character of ${\bL^*_{R,\ts}}^F$.
Let  $\wtbL^*_{\wtR,\ts}:=\C_{\C_{\wtbG}(\wtR)^*}(\ts)$. and let $\widetilde\la_{\wtR}$ be the  unipotent character of $\C_{\C_{\wtG}(\wtR)^*}(\ts)$ such that $\Res^{(\wtbL^*_{\wtR,\ts})^F}_{{\bL^*_{R,\ts}}^F}(\widetilde\la_R)=\la_R$.
Then $\widetilde\la_R$ is also $d$-cuspidal.
Let $\wtbL_{\wtR}^*=\C_{\C_{\wtbG}(\wtR)^*}(\Z(\wtbL^*_{\wtR,\ts})_d)$.
Then $\wtbL_{\wtR}^*$ is an $F$-stable $d$-split Levi subgroup of $\C_{\wtbG}(\wtR)^*$.
Let $\wtbL_{\wtR}$ be the $F$-stable $d$-split Levi subgroup of $\C_{\wtbG}(\wtR)$ dual to $\wtbL^*_{\wtR}$ as in \cite[Prop.~13.9]{CE04} and $\tzeta_{\wtR}=\fJ^{\wtbL_{\wtR}}_{\ts}(\widetilde\la_{\wtR})$.
Then $(\wtbL_{\wtR},\tzeta_{\wtR})$ is a $d$-cuspidal pair of $\C_{\wtbG}(\wtR)$.
Let $\fb_{\wtR}=b_{\C_{\wtbG}(\wtR)^F}(\wtbL_{\wtR},\tzeta_{\wtR})$ and $\ttheta_{\wtR}=\fJ^{\C_{\wtbG}(\wtR)}_{\ts}(\widetilde\la_{\wtR})$.
Then $\fb_{\wtR}$ is the block of $\C_{\wtG}(\wtR)$ containing $\ttheta_{\wtR}$.

\begin{prop}\label{prop:Brauer,pair}
$(\wtR,\fb_{\wtR})$ is a self-centralizing  $\wtB$-Brauer pair and $\ttheta_{\wtR}$ is the canonical character of $\fb_{\wtR}$. 
\end{prop}

\begin{proof}
Let $\wtbL^*_{\ts}=\C_{\wtbL^*}(\ts)$.
Then following \cite[1.4.2]{En08}, $(\wtbL^*_{\ts},\widetilde\la)$ is a unipotent $d$-cuspidal pair of $\C_{\wtbG^*}(\ts)$ associated to $(\wtbL,\tzeta)$.
On the other hand,  $(\wtbL^*_{\wtR,\ts},\widetilde \la_{\wtR})$ is a unipotent $d$-cuspidal pair of $\C_{\C_{\wtbG}(\wtR)^*}(\ts)$ associated to $(\wtbL_{\wtR},\tzeta_{\wtR})$.
By the construction above, it is easy to check directly that $$[\wtbL^*_{\ts},\wtbL^*_{\ts}]=[\wtbL^*_{\wtR,\ts},\wtbL^*_{\wtR,\ts}]$$ and $$\Res^{(\wtbL^*_{\ts})^F}_{[\wtbL^*_{\ts},\wtbL^*_{\ts}]^F}(\widetilde\la)=\Res^{(\wtbL^*_{\wtR,\ts})^F}_{[\wtbL^*_{\wtR,\ts},\wtbL^*_{\wtR,\ts}]^F}(\widetilde \la_{\wtR}).$$ Hence by \cite[Prop.~4.2.4]{En08}, 
$(\wtR,\fb_{\wtR})$ is a  $\wtB$-Brauer pair.

Now we prove that $(\wtR,\fb_{\wtR})$  self-centralizing and $\ttheta_{\wtR}$ is a canonical character of $\fb_{\wtR}$, \emph{i.e.},  $\Z(\wtR)\subseteq \Ker(\ttheta_{\wtR})$ and $\ttheta_{\wtR}$ is of $\ell$-defect zero when regarding as a character of $\C_{\wtG}(\wtR)/\Z(\wtR)$.
Since $\ts$ is an $\ell'$-element, we know $\Z(\wtR)\subseteq \Ker(\ttheta_{\wtR})$.
Now we prove that $\ttheta_{\wtR}(1)_\ell=|\C_{\wtG}(\wtR)/\Z(\wtR)|_\ell$.
By \cite[Cor.~2.6.6]{GM20},
$\ttheta_{\wtR}(1)_\ell=|\C_{\wtG}(\wtR)^*:\C_{\C_{\wtG}(\wtR)^*}(\ts)|_\ell \cdot \widetilde\la_{\wtR}(1)_\ell$.
So it suffices to show 
$\widetilde\la_{\wtR}(1)_\ell=|\C_{\C_{\wtG}(\wtR)^*}(\ts)|_\ell/|\Z(\wtR)|_\ell$.
By \cite[Cor.~4.3.4 and 4.4.18]{GM20},
$\la_R$ is of $\ell$-defect zero when regarding as a character of $\C_{\C_H(R)^*}(\ts)/\mathcal{O}_\ell(\Z(\C_{\C_H(R)^*}(\ts)))$.
On the other hand, it can be checked directly that $|\Z(\C_{\C_H(R)^*}(\ts))|_\ell=|\Z(R)|$.
Thus we prove the claim and this completes the proof.
\end{proof}

\begin{rem}
Note that the canonical character $\ttheta_{\wtR}$ has been constructed in \S \ref{subsect:canonical}.
\end{rem}

Write $R^{m,\alp,\beta}$ for $R_{m,\alp,0,\bc}$ with $\bc=(1,\ldots,1)$ ($\beta$ terms).
Now we give a special radical subgroup which is then a defect group.
In (\ref{equ:dec-R-bloc}),
let $(R_+)_\Ga=\prod\limits_{i} (R^{m_\Ga,\alp_\Ga,i})^{t_i}$, where $t_i$ is integers such that $w_\Ga=\sum_i t_i\ell^i$ is the $\ell$-adic expansion of $w_\Ga$.
For such $R$, we let $\widetilde D$ be the Sylow $\ell$-subgroup of $\pi^{-1}(R)$. Let $\fb_{\widetilde D}$ be the block of $\C_{\wtG}(\widetilde D)$ defined as above.

\begin{prop}
$\wtD$ is a defect group of $\wtB$ and 
$(\wtD,\fb_{\wtD})$ is a maximal Brauer $\wtB$-pair.
\end{prop}

\begin{proof}
By \cite[Thm.~1.6]{En08}, the defect groups of block $b_{\wtbG^F}(\wtbL,\tzeta)$ and block $b_{\wtbG(\ts)^F}(\wtbL(\ts),\tzeta(\ts))$ are isomorphic.
On the other hand, the block $b_{\wtbG(\ts)^F}(\wtbL(\ts),\tzeta(\ts))$ dominates block
$b_{\bH(\ts)^F}(\bM(\ts),\la(\ts))$.
Note that
$$1\to \F_q^\ti \fe \to \wtbG(\ts)^F\to \bH(\ts)^F\to 1.$$
Then if  $R'$ is a defect group of $b_{\bH(\ts)^F}(\bM(\ts),\la(\ts))$, then the Sylow $\ell$-subgroup of $\pi^{-1}(R')$ is a defect group of $b_{\wtbG(\ts)^F}(\wtbL(\ts),\tzeta(\ts))$.

By (\ref{equ:Hs}), $\bH(\ts)^F$ is a direct product of special orthogonal groups and general linear and unitary groups.
We rewrite $\bH(\ts)^F=\prod\limits_{\Ga\in \scF_\xi} H(\ts)_\Ga$, where $H(\ts)_\Ga\cong\SO_{m_\Ga(\ts)+1}(q^{d_\Ga})$ if $\Ga\in\scF_{\xi,0}$, while
$H(\ts)_\Ga\cong\GL_{m_\Ga(\ts)}(\veps_\Ga q^{\de_\Ga})$ if $\Ga\in\scF_{\xi,1}\cup\scF_{\xi,2}$.
Let $b_{\bH(\ts)^F}(\bM(\ts),\la(\ts))=\boxtimes_\Ga b_\Ga$, where $b_\Ga$ is a unipotent block of $H(\ts)_\Ga$.
In addition, the unipotent characters in $\Irr(b_\Ga)$ correspond to the partitions or Lusztig symbols whose $e_\Ga$-core are $\ka_\Ga$. 
The defect group $R'=\prod\limits_{\Ga\in\scF_\xi}R'_\Ga$, where $R'_\Ga$ is a defect group of $b_\Ga$.
On the other hand, the defect groups of general linear and unitary groups (resp. special orthogonal groups) are determined in \cite{FS82} (resp. \cite{FS89}).

We let $R_{\alp}^i$ be the iterated wreath product $\mathbb Z_{\ell^{a+\alp_\Ga}}\wr \mathbb Z_\ell\cdots \wr\mathbb Z_\ell$ (with $i$ factors $\mathbb Z_\ell$).
If $w_\Ga=\sum_i t_i\ell^i$ is the $\ell$-adic expansion of $w_\Ga$, then $R_\Ga\cong \prod\limits_i (R_{\alp_\Ga}^{i})^{t_i}$ by \cite{FS82,FS89}.
So $R'\cong R$.
Therefore, the group $\wtD$ defined above is isomorphic to a defect group of $\wtB$. 
On the other hand, $(\wtD,\fb_{\wtD})$ is a $\wtB$-Brauer pair by Proposition \ref{prop:Brauer,pair}.
So the assertion holds.
\end{proof}

\subsection{The blocks of finite spin groups}

In this section, we determine the number of $\ell$-blocks of $\bG^F=G=\Spin(V)$ covered by a given $\ell$-block of $\wtbG^F=\wtG$. This was considered by Enguehard \cite[\S 3]{En08}.
The method used in this section is also similar with those in \cite[\S 4]{Fe19} and \cite[\S 5.2 and Appendix B]{FLZ19}, where the number of blocks of $\mathrm{SL}_n(\pm q)$, $\Sp_{2n}(q)$ and $\SO_{2n}^\pm(q)$ covered by a fixed block of $\mathrm{GL}_n(\pm q)$, $\CSp_{2n}(q)$ and $\CSO_{2n}^\pm(q)$ respectively are given by considering $d$-cuspidal pairs.

\begin{lem}\label{relation-blocks}
	\begin{enumerate}[\rm(1)]
		\item Let $(\wtbL,\tzeta)$ be an $d$-cuspidal pair of $\wtbG$ and $B$ an $\ell$-block of $\bG^F$ covered by $\wtB=b_{\wtbG^F}(\wtbL,\tzeta)$, then $B=b_{\bG^F}(\bL,\zeta)$, where $\bL=\wtbL\cap \bG$ and $\zeta\in \Irr(\bL^F\mid \tzeta)$.
		\item Let $(\bL,\zeta)$ be an $d$-cuspidal pair of $\bG$ and $\wtB$ an $\ell$-block of $\wtbG^F$ which covers $B=b_{\bG^F}(\bL,\zeta)$, then $\wtB=b_{\wtbG^F}(\wtbL,\tzeta)$ for some $d$-cuspidal pair $(\wtbL,\tzeta)$ of $\wtbG$ satisfying that $\bL=\wtbL\cap \bG$ and $\zeta\in\Irr(\bL^F\mid \tzeta)$.
	\end{enumerate}
\end{lem}

\begin{proof}
	This follows by \cite[Lemma~3.7 and 3.8]{KM15} and is entirely analogous with \cite[Prop.~4.5]{Fe19} and \cite[Prop.~5.6]{FLZ19}.
\end{proof}

\begin{lem}\label{lem:num-covered}
Let $\wtB=b_{\wtbG^F}(\wtbL,\tzeta)$ be an $\ell$-block of $\wtbG^F$ and $\bL=\wtbL\cap\bG$.
Then the number of $\ell$-blocks of $\bG^F$ covered by $\wtB$ is  $|\Irr(\bL^F\mid \tzeta)|$.
\end{lem}

\begin{proof}
By Lemma \ref{relation-blocks}, each $\ell$-block of $G$ covered by $\wtB$ is of form $b_{\bG^F}(\bL,\zeta)$, where $\bL=\wtbL\cap \bG$ and $\zeta\in \Irr(\bL^F\mid \tzeta)$.
By the classification of blocks of finite reductive groups (see \S \ref{subsec:class-blocks}), in order to prove this lemma, it suffices to show that any two $d$-cuspidal pairs $(\bL,\zeta_1)$ and $(\bL,\zeta_2)$ with $\zeta_1, \zeta_2\in \Irr(\bL^F\mid \tzeta)$ are not $\bG^F$-conjugate if $\zeta_1\ne\zeta_2$, which follows by \cite[Prop.~2.3.2]{En08}.
\end{proof}

Let $z_0=-1_{V^*}$ be as in the proof of Theorem  \ref{thm:char-spin}. For any blocks $\wtB$ of $\wtG$. Then $\wtB$ covers one or two blocks of $G$ since $|\wtG/\Z(\wtG) G|=2$. So by Corollary \ref{cor:num-covered}, we have

\begin{cor}\label{cor:num-covered}
Let $\wtB=b_{\wtbG^F}(\wtbL,\tzeta)$.
Then $\wtB$  covers two blocks of $G$ if and only if $\hat z_0 \tzeta=\tzeta$.
\end{cor}

Let  $\cE:=\cE_1$ be defined as in (\ref{def-cE}).
Then $\cE$ is the set of elementary divisors for the semisimple conjugacy classes of $\GL_{-}(\veps q^e)$.
For $\Delta\in \cE$, if $\alp$ is a root of $\Delta$ in $\overline{\F}_q^\ti$, then we define $-\Ga$ to be the polynomial in $\cE$ such that $-\alp$ is a root of $-\Ga$.

For $\Ga\in\scF_\xi$ with $\xi\in\F_q^\ti$, we define  $\Ga_{(e)}$ to be a polynomial in $\cE$ such that $\Ga$ and $\Ga_{(e)}$ have a common root in $\overline{\F}_q^\ti$.
Thus $\Ga=\mathcal N_{1,\xi}(\Ga_{(e)})$.
Note that $\Ga_{(e)}$ is not uniquely defined. However, we note that whether $\Ga_{(e)}$ is $\Ga$ or not does not depend on the choice of $\Ga_{(e)}$.


\begin{thm}
Let $\wtB=\wtB(\ts,\ka)$ where $\ts$ is a semisimple $\ell'$-element of $\wtG^*=(\wtbG^*)^F$ with multiplier $\xi$ and $\ka\in\fC(s)$. 	Then   $\wtB$  covers two blocks of $G$ if and only if
	$(-\Ga)_{(e)}=\Ga_{(e)}$ for every $\Ga\in\scF_\xi$ with $w_\Ga>0$ and $\ka_\Ga=\ka_{-\Ga}$ for every $\Ga\in\scF_\xi$.
\end{thm}

\begin{proof}
	We make use of Corollary \ref{cor:num-covered}.
Recall that $\tzeta=\fJ_{\ts}^{\wtbL}(\widetilde\la)$.
Then by \cite[Thm.~4.7.1~(3)]{GM20},
$\hat z_0 \tzeta=\fJ_{-\ts}^{\wtbL}(\widetilde\la)$.
So $\hat z_0 \tzeta=\tzeta$ if and only if $\fJ_{\ts}^{\wtbL}(\widetilde\la)=\fJ_{-\ts}^{\wtbL}(\widetilde\la)$.

Recall from the beginning of \S \ref{subsec:blocks-clifford} that $\ts_0\in \langle \tau, L^*_0 \rangle$ and $\ts_+\in L_+^*$.
Thus
$\fJ_{\ts}^{\wtbL}(\widetilde\la)=\fJ_{-\ts}^{\wtbL}(\widetilde\la)$ if and only if $(\ts_0,\ka)$ and $(-\ts_0,-\ka)$ are $\langle \tau, L^*_0 \rangle$-conjugate and $\ts_+$ and $-\ts_+$ are $L^*_+$-conjugate.
Then this assertion is similar with the arguments in the proof of \cite[Cor.~4.12]{Fe19}.
\end{proof}

\begin{rmk}
	Let $B$ be a block of $G$ covered by $\wtB=\wtB(\ts,\ka)$ where $\ts$ is a semisimple $\ell'$-element of $\wtG^*=(\wtbG^*)^F$ with multiplier $\xi$. Then the blocks of $\wtG$ covering  $B$ are $\wtB(\tilde z\ts,\ka)$, where $\tilde z$ runs through $\mathcal O_{\ell'}(\Z(\wtG^*))$.
Precisely,
\begin{enumerate}[(i)]
	\item if  $\ka_\Ga=\ka_{-\Ga}$ for every $\Ga\in\scF_\xi$, then there are $(q-1)_{\ell'}/2$ blocks of $\wtG$ covering $B$, \emph{i.e.}, $\wtB(\tilde z\ts,\ka)$ where $\tilde z$ runs through a complete set of representatives of $\langle -1_{V^*}\rangle$-cosets in $\mathcal O_{\ell'}(\Z(\wtG^*))$, and
	\item otherwise there are $(q-1)_{\ell'}$ blocks of $\wtG$ covering $B$, \emph{i.e.}, $\wtB(\tilde z\ts,\ka)$ where $\tilde z$ runs through $\mathcal O_{\ell'}(\Z(\wtG^*))$.
\end{enumerate}
\end{rmk}

\section{The inductive conditions}\label{proof-main}

\subsection{}
First, we give an equivariant bijection for $\wtG$.

\begin{prop}\label{field-action-weights}
Let $(\wtR,\tvphi)$ be a weight of $\wtG$ with label $(\ts,\ka,K)$ and $\sigma \in \langle F_p \rangle$,
then $(\wtR,\tvphi)^\sigma$ has label $(\sigma^*(\ts),{^{\sigma^*}\!\ka},{^{\sigma^*}\!K})$.
\end{prop}

\begin{proof}
This can be proved by similar method as in \cite[\S6.C]{Li19}.
\end{proof}

\begin{prop}\label{equiv-bijection}
There is a blockwise bijection between $\cE(\wtG,\ell')$ and $\Alp(\wtG)$ equivariant under the action of $\Lin_{\ell'}(\wtG/G)\rtimes\langle F_p \rangle$.
\end{prop}

\begin{proof}
Then for each $\Ga$, $K_\Ga$ corresponds to a $\beta_\Ga e_\Ga$-tuples of partitions
using the method in \cite[(1A)]{AF90}.
When $\Ga\notin\scF_{\xi,0}$,
the $e_\Ga$-core $\ka_\Ga$ and the $e_\Ga$-tuple of partitions give a partition $\mu_\Ga$ of $m_\Ga(\ts)$.
Assume $\Gamma\in\scF_{\xi,0}$.
Since $\ka_\Ga$ is non-degenerate of odd defect,
$\ka_\Ga$ and the $2e_\Ga$-tuple of partitions above give a Lusztig symbol $\mu_\Ga$ of rank $m_\Ga(\ts)/2$ of odd defect; see \cite[\S4]{Li19}.
Then the triple $(\ts,\ka,K)$ corresponds to a $\tchi_{\ts,\mu}$ as in \S\ref{subsec:irr-char}.
This gives a bijection between $\cE(\wtG,\ell')$ and $\Alp(\wtG)$.
The bijection is blockwise by the results in Proposition \ref{prop:Brauer,pair}.
The equivariance follows from Lemma \ref{char-lin-act}, Proposition \ref{field-action-IBr},
Proposition \ref{Lin-Alp(tG)} and Proposition \ref{field-action-weights}.
\end{proof}

Recall that $\cE(\wtbG^F,\ell')$ and $\cE(\bG^F,\ell')$ are basic sets for $\wtbG^F$ and $\bG^F$ respectively.
Thus we have proved

\begin{cor}
	The blockwise Alperin weight conjecture holds for both $\wtG$ and $G$.
\end{cor}

According to Proposition \ref{equiv-bijection}, if Assumption \ref{uni-assumption} holds, then 
there is a blockwise bijection between $\IBr(\wtG)$ and $\Alp(\wtG)$ equivariant under the actions of $\Lin_{\ell'}(\wtG/G)\rtimes \langle F_p \rangle$.

\subsection{}

To consider the local property,
we need some detailed information concerning normalizers of radical subgroups of $G$.
As before, we always abbreviate the superscript ``tw'' when we work in twisted groups.

\begin{rmk}\label{rmk:weyl-M}
Although the Weyl groups of $\wtG$ and $G$ are isomorphic,
the group $\wtW_0 \rtimes \wtW_1$ in \S\ref{subsect:Weyl-D0} affording the Weyl group of $\wtG$ are not in $G$.
We can give an explicit subgroup of $G$ affording the Weyl group of $G$.
Fix a basis (\ref{equ-basis}) of $V$.
Replace $\tw_{ij}$ and $\tw_i$ in \S\ref{subsect:Weyl-D0} by
\begin{align*}
\tw_{ij} &= \frac{1}{2}(\veps_i-\veps_j+\eta_i-\eta_j)(\veps_i-\veps_j-\eta_i+\eta_j)
(\veps_i+\eta_i)(\veps_i-\eta_i); \\
\tw_i &= z_V(\veps_i-\eta_i)\prod_{k=1}^n(\veps_k+\eta_k)(\veps_k-\eta_k).
\end{align*}
Then direct calculation show that $\pi(\tw_{ij}), \pi(\tw_i) \in [H,H]$,
thus $\tw_{ij}, \tw_i \in G$.
Set $\wtW= \set{ \tw_{ij},\tw_i \mid 1\leq i\neq j\leq n }$,
then $\wtW/\wtW\cap T \cong \ZZ_2\wr\fS_n$ is the Weyl group of $G$.
Now, replace $\tv_{m,\alp,\ga}$ and $\tdelta_{m,\alp,\ga}$ in \S\ref{subsect:key-lemma-weights}
by some elements in the group $\wtW$ above
affording the same elements in $\ZZ_2\wr\fS_n$ as $v_{m,\alp,\ga}$ and $\delta_{m,\alp,\ga}$ respectively.
Set $\wtV_{m,\alp,\ga}=\grp{\tv_{m,\alp,\ga},\tdelta_{m,\alp,\ga}}$ when $\veps=1$,
while set $\wtV_{m,\alp,\ga}=\grp{\tv_{m,\alp,\ga}}$ when $\veps=-1$.
Then $\wtV_{m,\alp,\ga}$ is a subgroup of $G$.
Note that the above process from $\wtG$ to $G$ is similar as the situation from general linear groups to special linear groups as in \cite{FLZ20a}.

Now, we consider the part $\wtM_{m,\alp,\ga}$ in $\wtN_{m,\alp,\ga}$ in \S\ref{subsect:key-lemma-weights}.
First, note that the subgroup $M_{m,\alp,\ga}^0$ in \S\ref{subsect:CN-SO}
satisfies that $\det(M_{m,\alp,\ga}^0)$ is a square
in $\F_{q^{e\ell^{\alp}}}$ or $\F_{q^{2e\ell^{\alp}}}$ according to $\veps=1$ or $-1$;
see \cite[\S3.A]{FLZ20a}.
In fact, $\det(M_{m,\alp,\ga}^0)=1$ except for a special case when $\ell=3$,
in which case, $\det(M_{m,\alp,\ga}^0)$ is a cyclic group of order $3$,
thus a square.
So $M_{m,\alp,\ga}$ is in fact in $[H,H]$.
Thus we can let $\wtM_{m,\alp,\ga}$ be the pre-image of $M_{m,\alp,\ga}$ in $G$.

In short, we can adjust the construction of $\wtN_{m,\alp,\ga}$
such that the parts $\wtV_{m,\alp,\ga}$ and  $\wtM_{m,\alp,\ga}$
are contained in $G$.
\end{rmk}

\subsection{Local property} 
Now we prove the local property.

\begin{prop}\label{local-condition}
For every weight $(R,\vph)$ of $G$ one has $\N_{\wtG D}(R)_\vph = \N_{\wtG}(R)_{\vph} \N_{GD}(R)_\vph$, where $D=\langle F_p\rangle$.
\end{prop}

\begin{proof}
As in \cite[Prop.~7.10]{FLZ20a} (whose proof is entirely similar with \cite[Prop.~5.13]{CS17}), we may transfer to twisted radical subgroups.
Using the notation for the radical subgroups and weights  in \S\ref{sec:weights},
we work on the twisted groups and we may let $\wtG=\wtbG^{\tv F}$ and  $G=\bG^{\tv F}$.
As before, ``tw'' will be dropped.
Then the basic radical subgroups are $D$-stable.
It suffices to show that if $(R,\vph)$ is a weight of $G$ with $R=\wtR\cap G$, where $\wtR$ is a direct product of basic subgroups, then
$(\N_{\wtG}(R)\rtimes D)_\vph = \N_{\wtG}(R)_{\vph} \rtimes D_\vph$.

Let $V=V_0\perp (V_1 \perp \cdots\perp V_1)\perp  \cdots \perp (V_u\perp\cdots\perp V_u)$~(where $V_i$ appears $t_i$ times for $i\ge 1$) be the decomposition corresponding to
$\wtR=\wtR_0\prod_{i=1}^u (\wtR_i)^{t_i}$, a central product over $\F_q^\ti\fe$, where $\wtR_0=\F_q^\ti\fe$ and $\wtR_i=\wtR_{m_i,\alp_i,\ga_i,\bc_i}$.
Then $\wtC=\C_{\wtG}(\wtR)\cong \wtC_0\prod_{i=1}^u(\wtC_i)^{t_i}$ is a central product over $\F_q^\ti\fe$ and $\wtN=\N_{\wtG}(\wtR)=\wtN_0\prod_{i=1}^u (\wtN_i)^{t_i}\rtimes \fS(t_i)$,	
where $\wtC_0=\wtN_0\cong D_0(V_0)$, $\wtC_i=\wtC_{m_i,\alp_i,\ga_i,\bc_i}$, 
$\wtN_i=\wtN_{m_i,\alp_i,\ga_i,\bc_i}$.
The products between $\wtN_i$'s are central products over $\F_q^\ti\fe$.
Let $\wtM_i=\wtM_{m_i,\alp_i,\ga_i}$, $\wtV_i=\wtV_{m_i,\alp_i,\ga_i}$, $\wtT_i=\N_{\fS(\ell^{|\bc_i|})}(A_{\bc_i})/A_{\bc_i}.$
Then $\wtN_i/\wtR_i=(\wtR_i\wtC_i\wtM_i/\wtR_i)\wtV_i\ti\wtT_i$,
where $\wtR_i\wtC_i\wtM_i/\wtR_i$ is a central product of $\wtR_i\wtC_i/\wtR_i$ and $\wtR_i\wtM_i/\wtR_i$ over $\grp{-\fe}$.
 Let  $\wtG_i=D_0(V_i)$ and $G_i=\Spin(V_i)$.
Then by Remark \ref{rmk:weyl-M},
$(\wtN_i\cap G_i)/(\wtR_i\cap G_i)=((\wtR_i\wtC_i\cap G_i)\wtM_i/(\wtR_i\cap G_i))\wtV_i\ti\wtT_i$.
As in \cite[\S 7]{FLZ20a}, we may replace $\fS(t_i)$ by some group $\cS_i\le G$ such that $\wtN=\wtN_0\prod_{i=1}^u (\wtN_i)^{t_i}\cS_i$.

Let  $\wtG'=\wtG_0\ti \prod_{i=1}^u \wtG_i\cS_i$.
Here $\cS_i$ permutes the components via the surjection $\cS_i\to\fS(t_i)$ with kernel $\grp{-\fe}$.
Let $\wtC'=\wtC_0\ti(\wtC_1)^{t_1}\ti\cdots\ti (\wtC_u)^{t_u}$, and let $\wtN''=\wtN_0\ti(\wtN_1)^{t_1}\ti\cdots\ti (\wtN_u)^{t_u}$ and
$\wtN'=\wtN_0\ti\prod_{i=1}^u \wtN_i\cS_i$.
Then $\wtC',\wtN'\le\wtG'$.
Let $$\wtZ'=\F_q^\ti\fe\ti(\F_q^\ti\fe)^{t_1}\ti\cdots\ti (\F_q^\ti\fe)^{t_u}\le \wtC'$$ and take 
$\widehat{\wtZ'}=\{ (z_1,z_2,\ldots)\in\wtZ'\mid z_1z_2\cdots=1 \}$.
Then $\wtC'/\widehat{\wtZ'}\cong \wtC$ and $\wtN'/\widehat{\wtZ'}\cong \wtN$.

Let $\wtG_i\to \F_q^\ti$ be the map defined in \S\ref{subsect:Clifford} denoted by $\eta_i$ and let $$\eta=\eta_0\ti(\eta_1)^{t_1}\ti\cdots\ti (\eta_u)^{t_u}:\wtG_0\ti(\wtG_1)^{t_1}\ti\cdots\ti (\wtG_u)^{t_u}\to \F_q^\ti.$$
Let $\cS$ be the direct product of the $\cS_i$'s.
Then $\C_{G}(R)=C'/\widehat{\wtZ'}$, $\N_G(R)=N'/\widehat{\wtZ'}$,
where $C'=\Ker(\eta|_{\wtC'})$ and $N'=\Ker(\eta|_{\wtN''})\cS$.
Since $D$ acts on $\wtG_i$, $\wtN_i$ and $\wtC_i$ for every $i$, we can define the actions of $D$ on $\wtG'$, $\wtC'$ and $\wtN'$. Note that $D$ acts trivially on $\cS$.
Let $\vph'$ be the lift of $\vph$ to $N'$.
By \S\ref{subsec:cen-prod},
it suffices to show that $(\wtN'\rtimes D)_{\vph'}=\wtN'_{\vph'}\rtimes D_{\vph'}$.

In this way, we are in the situation which is entirely analogous with \cite[\S 7]{FLZ20a} for type $\mathsf A$.
Our groups $\wtG'$, $\wtN'$, $\wtC'$, $D$ here play the same roles as $\wtG$, $\wtN$, $\wtC$, $E$ there.
Thus the proof is the same with \cite[Prop.~7.18]{FLZ20a}.
\end{proof}

\subsection{} 
Finally, we establish the inductive BAW condition for simple groups of type $\mathsf B$ and odd primes.

\begin{proof}[Proof of Theorem \ref{main-thm}]
Let $G$ be the universal $\ell'$-covering group of simple group $S=\Omega_{2n+1}(q)$ (with odd $q$).
Then one has $G=\Spin_{2n+1}(q)$ unless $n=q=3$.

We first consider $S=\Omega_7(3)$ so that $|G|=2^{10}\cdot 3^{10}\cdot 5\cdot 7 \cdot 13$.
Since $\ell\nmid q$ and $\ell$ is odd, one has $\ell=5,7,13$. Then the Sylow $\ell$-subgroups of $G$ are cyclic and the inductive BAW condition holds for all $\ell$-blocks of $G$ by \cite{KS16}.

Now we assume that $G=\Spin_{2n+1}(q)$.
Then we check the conditions of Theorem \ref{thm:criterion}.
As before, $\wtG$ is the corresponding special Clifford group and $D=\langle F_p \rangle$.
Then (1) is obvious.
By Proposition \ref{equiv-bijection}, if Assumption \ref{uni-assumption} holds, then 
there is a blockwise bijection between $\IBr(\wtG)$ and $\Alp(\wtG)$ equivariant under the actions of $\Lin_{\ell'}(\wtG/G)\rtimes D$, and thus condition (2) holds (we note that the condition (2.ii)  always holds; see Remark \ref{condition-J}).
By Corollary \ref{split-stabilizer}, condition (3) holds if the Assumption \ref{uni-assumption} is true.
Finally, (4) is checked in Proposition \ref{local-condition}.
This completes the proof.
\end{proof}



\begin{thebibliography}{99}

\bibitem{Al87}
{\sc J. L. Alperin}, Weights for finite groups. In: \emph{The Arcata Conference
  on Representations of Finite Groups, Arcata, Calif. (1986), Part I}. 
  Proc. Sympos. Pure Math., {vol. \bf 47}, Amer. Math. Soc., Providence, 1987,
  pp.369--379.
  
 \bibitem{AF90}
{\sc J. L. Alperin, P. Fong}, Weights for symmetric and general linear groups.
\emph{J. Algebra \bf 131} (1990) 2--22.

\bibitem{An94}
{\sc J. An}, Weights for classical groups. \emph{Trans. Amer. Math. Soc.
  \bf 342} (1994), 1--42.
  
\bibitem{AHL21}
{\sc J. An, G. Hiss,  F. L\"ubeck},   The inductive Alperin weight condition for the Chevalley groups $F_4(q)$.
arXiv:2103.04597.


\bibitem{BM11}
{\sc C. Bonnaf\'e, J. Michel}, Computational proof of the Mackey formula for $q>2$. \emph{J. Algebra \bf 327} (2011), 506--526.


\bibitem{BS20}
{\sc J. Brough, B. Sp\"ath}, A criterion for the inductive Alperin weight  condition. arXiv:2009.02074.

\bibitem{BDT20}
{\sc O. Brunat, O. Dudas, J. Taylor},
Unitriangular shape of decomposition matrices of unipotent blocks.
\emph{Ann. of Math. (2)  \bf 192}~(2020), 583--663.

\bibitem{CE94}
{\sc M. Cabanes, M. Enguehard}, On unipotent blocks and their ordinary characters. \emph{Invent. Math.\bf  117} (1994), 149--164.

\bibitem{CE99} 
{\sc M. Cabanes, M. Enguehard}, On blocks of finite reductive groups and twisted induction. \emph{Adv. Math. \bf
	145} (1999), 189--229.

\bibitem{CE04}
{\sc M. Cabanes, M. Enguehard}, \emph{Representation Theory of Finite Reductive
  Groups}. New Math. Monogr., {vol. \bf 1}, Cambridge University Press,
  Cambridge, 2004.

\bibitem{CS13}
{\sc M. Cabanes, B. Sp\"ath}, Equivariance and extendibility in finite reductive groups with connected center. \emph{ Math. Z. \bf 275}~(2013), 689--713.


\bibitem{CS17}
{\sc M. Cabanes, B. Sp\"ath}, Equivariant character correspondences and inductive McKay condition for type $\mathsf A$. \emph{J. Reine Angew. Math. \bf 728} (2017), 153--194.

\bibitem{CS19}
{\sc M. Cabanes, B. Sp\"ath}, Descent equalities and the inductive McKay condition for types $\mathsf B$ and $\mathsf E$. \emph{Adv. Math. \bf 356}~(2019), 106820.

\bibitem{Chev}
{\sc C. Chevalley}, \emph{The Algebraic Theory of Spinors and Clifford Algebras, Collected works, vol. 2}. Springer--Verlag, Berlin, 1997.

\bibitem{DM20}
{\sc F. Digne, J. Michel}, \emph{Representations of Finite Groups of Lie Type}.  2nd ed. London Math. Soc. Stud. Texts, {vol. \bf 95}, Cambridge University Press, Cambridge, 2020.


\bibitem{En08}
{\sc M. Enguehard},  Vers une d\'ecomposition de Jordan des blocs des groupes r\'eductifs finis. \emph{J. Algebra \bf 319}~(2008), 1035--1115. 

\bibitem{Fe19}
{\sc Z. Feng}, The blocks and weights of finite special linear and unitary groups. \emph{J. Algebra \bf 523}~(2019), 53--92.


\bibitem{FLZ20a} 
{\sc Z. Feng, C. Li, J. Zhang}, Equivariant correspondences and the inductive Alperin
weight condition for type $\mathsf A$. arXiv:2008.05645.

\bibitem{FLZ20b} 
{\sc Z. Feng, C. Li, J. Zhang}, On the inductive blockwise Alperin weight condition for type $\mathsf A$. arXiv:2008.06206.

\bibitem{FLZ19}
{\sc Z. Feng, Z. Li, J. Zhang}, On the inductive blockwise Alperin weight
condition for classical groups. \emph{J. Algebra \bf 537} (2019), 381--434.



\bibitem{FM20} 
{\sc Z. Feng, G. Malle}, The inductive blockwise Alperin weight condition for type $\mathsf C$ and
the prime 2. 
	J. Austral. Math. Soc. (2020),
https://doi.org/10.1017/S1446788720000439.

\bibitem{FS82}
{\sc P. Fong, B. Srinivasan}, The blocks of finite general linear and unitary
  groups. \emph{ Invent. Math. \bf 69} (1982), 109--153.

\bibitem{FS86} 
{\sc P. Fong, B. Srinivasan}, Generalized Harish--Chandra theory for unipotent characters of finite classical groups. \emph{J. Algebra \bf 104} (1986) 301--309.

\bibitem{FS89}
{\sc P. Fong, B. Srinivasan}, The blocks of finite classical groups. \emph{J.
  Reine Angew. Math. \bf 396} (1989), 122--191.



\bibitem{GM20}
{\sc M. Geck, G. Malle}, \emph{The Character Theory of Finite Groups of Lie
  Type: A Guided Tour}. Cambridge University Press, Cambridge, 2020.

\bibitem{GLS98}
{\sc D. Gorenstein, R. Lyons, R. Solomon}, \emph{The Classification of the Finite Simple Groups, Number 3}. Math. Surveys Monogr., {vol. \bf 40}, American Mathematical Society, Providence, RI, 1998.

\bibitem{IMN07}
{\sc I. M. Isaacs, G. Malle, G. Navarro}, A reduction theorem for the McKay conjecture. \emph{Invent. Math. \bf 170} (2007), 33--101.

\bibitem{KM15} 
{\sc R. Kessar, G. Malle}, Lusztig induction and $\ell$-blocks of finite reductive groups. \emph{Pacific J. Math. \bf 279}
(2015), 269--298.

\bibitem{KS16}
{\sc S. Koshitani, B. Sp\"ath},
The inductive Alperin--McKay and blockwise Alperin weight conditions for blocks with cyclic defect groups and odd primes.
\emph{J. Group Theory \bf 19} (2016), 777--813.

\bibitem{Li19}
{\sc C. Li}, An equivariant bijection between irreducible Brauer characters and weights for $\Sp(2n,q)$. \emph{J. Algebra \bf 539} (2019), 84--117.

\bibitem{Li21}
{\sc C. Li}, The inductive blockwise Alperin weight condition for $\PSp_{2n}(q)$ and odd primes. \emph{J. Algebra \bf 567} (2021), 582--612.

\bibitem{LZ18}
{\sc C. Li, J. Zhang}, The inductive blockwise Alperin weight condition for $\mathrm{PSL}_n(q)$ and $\mathrm{PSU}_n(q)$ with cyclic outer automorphism groups, \emph{J. Algebra \bf 495}~(2018), 130--149.

\bibitem{Ma14}
{\sc G. Malle}, On the inductive Alperin--McKay and Alperin weight conjecture for groups with abelian Sylow subgroups. \emph{J. Algebra \bf 397} (2014),
190--208.




\bibitem{MT11}
{\sc G. Malle, D. Testerman}, \emph{Linear Algebraic Groups and Finite Groups of Lie Type}. Cambridge Stud. Adv. Math., {vol. \bf 133}, Cambridge University Press, Cambridge, 2011.



\bibitem{NT11}
{\sc G. Navarro, P. H. Tiep}, A reduction theorem for the Alperin weight
  conjecture. \emph{Invent. Math. \bf 184} (2011), 529--565.



\bibitem{Ol93}
{\sc J. B. Olsson}, \emph{Combinatorics and Representations of Finite Groups}.
Vorlesungen aus dem Fachbereich Mathematik der Universit\"at Essen, Heft {\bf 20},
Universit\"at Essen, Essen, 1993.


  
 \bibitem{Sh80} 
{\sc K. Shinoda}, The characters of Weil representations associated to finite fields. \emph{J. Algebra \bf 66} (1980), 251--280.

\bibitem{Sp13}
{\sc B. Sp\"ath}, A reduction theorem for the blockwise Alperin weight
  conjecture. \emph{J. Group Theory \bf 16} (2013), 159--220.

\bibitem{Sch16}
{\sc E. Schulte}, The inductive blockwise Alperin weight condition for $G_2(q)$ and $^3D_4(q)$. \emph{J. Algebra \bf 466} (2016), 314--369.

\bibitem{SS70}
{\sc T. A. Springer, R. Steinberg}, Conjugacy classes. In: \emph{Seminar on Algebraic Groups and Related Finite Groups, Lecture Notes in Math. \bf 131}, Springer--Verlag, New York, 1970,  pp. 167--266.

\bibitem{Ta18} 
{\sc J. Taylor}, On the Mackey formula for connected centre groups. \emph{J. Group Theory \bf 21} (2018), 439--448.


\bibitem{Th95}
{\sc J. Th\'evenaz}, \emph{$G$-algebra and Modular Representation Theory}. Oxford Math. Monogr., The Clarendon Press, Oxford Univ. Press, New York, 1995. 

\end{thebibliography}
\end{document}